\documentclass[a4paper,DIV=11,abstract=true,twoside=semi]{scrreprt}

\usepackage{graphicx}
\usepackage{amsfonts} 
\usepackage{amsmath}
\usepackage{amsthm} %Proof Environment
\usepackage{amssymb}
\usepackage{marvosym}
\usepackage{stmaryrd} %\lightning
\usepackage{latexsym} %\leadsto

\usepackage[utf8]{inputenc}
\usepackage[T1]{fontenc}
\usepackage[english]{babel} 

\usepackage{pdfpages} % include pdf pages
\usepackage{hyperref}
\usepackage{graphicx}

\usepackage{enumitem}
\usepackage{multicol}
\usepackage{chngcntr}

\usepackage{xifthen}

\usepackage{listings} % source listings
\usepackage{xcolor} % defining own colors: \color

\usepackage{dsfont} % \1 for identity functions

\usepackage{etoolbox}

\usepackage{booktabs}

\usepackage{accents}
\newcommand{\ubar}[1]{\underaccent{\bar}{#1}}

\newcount\xveccount
\renewcommand*\vec[1]{
	\global\xveccount#1
	\begin{pmatrix}
		\vecnext
	}
	\def\vecnext#1{
		#1
		\global\advance\xveccount-1
		\ifnum\xveccount>0
		\\
		\expandafter\vecnext
		\else
	\end{pmatrix}
	\fi
}

\mathcode`\*="8000
{\catcode`\*\active\gdef*{\cdot}}

\newcommand{\N}{\mathbb{N}}

\newcommand{\R}{\mathbb{R}}

\DeclareMathOperator{\Pot}{\mathcal{P}}

\DeclareMathOperator{\inv}{^{-1}}

\newcommand{\abs}[1]{\left| #1 \right|}

\newcommand{\lintegral}[3]{\int_{#1} #2 \, d#3}

\newcommand{\lra}{\Leftrightarrow} % equivalence arrow
\newcommand{\1}[1][]{\mathds{1}\ifthenelse{\isempty{#1}}{}{_{#1}}}
\newcommand{\dummydot}{\,\cdot\,}

\newcommand{\innerprod}[2]{\left\langle #1, #2 \right\rangle}
\DeclareMathOperator{\supp}{supp}
\newcommand{\transposed}{^{T}}
\DeclareMathOperator{\argmax}{argmax}

\newcommand{\colonlra}{\mathrel{\vcentcolon\Leftrightarrow}} % properly typeset :<=>
\DeclareMathOperator{\id}{id}

\newcommand{\ifequals}[4]{\ifthenelse{\equal{#1}{#2}}{#3}{#4}}
\definecolor{verylightgray}{gray}{0.97}
\definecolor{purple}{RGB}{127,0,116} % used for keyword coloring in source code listings
\definecolor{brickred}{rgb}{0.56, 0.175, 0.231} % used for string coloring in source code listing

\newcommand{\norm}[1]{\left\lVert#1\right\rVert}

\newcommand{\liminfty}[1]{\lim\limits_{#1 \rightarrow \infty}}
\newcommand{\ser}[1]{\sum\limits_{\ifthenelse{\isin{=}{#1}}{#1}{#1=1}}^{\infty}}
\newcommand{\toinfty}[1]{\xrightarrow{\;#1 \to \infty\;}}

\newcommand{\twocases}[4]
{
    \begin{cases}
        #1, &\text{if } #2 \\
        #3, &\ifthenelse{\isempty{#4}}{\text{else}}{\text{if } #4}
    \end{cases}
}

\newcommand{\pars}[1]{\left(#1\right)}
\newcommand{\bigpars}[1]{\bigl(#1\bigr)}
\newcommand{\biggpars}[1]{\biggl(#1\biggr)}
\newcommand{\set}[1]{\{#1\}}

\newcommand{\biggset}[1]{\biggl\{#1\biggr\}}
\newcommand{\bigmid}{\bigm\vert}
\newcommand{\biggmid}{\biggm\vert}

\makeatletter
\patchcmd{\upbracefill}{\m@th}{\scriptscriptstyle\m@th}{}{}
\patchcmd{\upbracefill}{$\braceld$}{$\scriptstyle\braceld$}{}{}
\patchcmd{\upbracefill}{\bracelu}{\bracelu\mkern-1mu}{}{}
\patchcmd{\upbracefill}{\hfill\braceru}{\hfill\mkern-1mu\braceru}{}{}
\makeatother

\newcommand{\smallmat}[1]{\left(\begin{smallmatrix}#1\end{smallmatrix}\right)}
\newcommand{\undersetbrace}[2]{\underset{#2}{\underbrace{#1}}}

\newcommand\restr[2]{{% we make the whole thing an ordinary symbol
        \left.\kern-\nulldelimiterspace % automatically resize the bar with \right
        #1 % the function
        \vphantom{\big|} % pretend it's a little taller at normal size
        \right|_{#2} % this is the delimiter
}}

\usepackage[citestyle=alphabetic,bibstyle=alphabetic,maxbibnames=99]{biblatex}
\usepackage{csquotes} % needed for biblatex

\usepackage{setspace}
\usepackage{mathtools} % \bigtimes
\usepackage{caption}
\usepackage{subcaption}
\usepackage{thmtools}

\usepackage{textcomp}
\usepackage{eurosym}
\usepackage{adjustbox}
\usepackage{changepage}
\usepackage{hyperref}

\let\realincludegraphics\includegraphics
\renewcommand{\includegraphics}[2][]{\realincludegraphics[#1]{Lowres#2}}

\usepackage{parskip}
\makeatletter
\def\thm@space@setup{%
    \thm@preskip=\parskip \thm@postskip=0pt
}
\makeatother

\DeclareMathOperator{\A}{\mathcal{A}}

\newcommand{\Rp}{\mathbb{R}_{\geq 0}}
\let\epsilon\varepsilon
\let\phi\varphi
\newcommand{\B}{\mathcal{B}}
\newcommand{\M}{M}

\theoremstyle{definition}
\newtheorem{thm}{Theorem}[chapter] % reset theorem numbering for each chapter
\newtheorem*{thm*}{Theorem}
\newtheorem{lemma}[thm]{Lemma} % same for Lemma numbers
\newtheorem*{lemma*}{Lemma}
\newtheorem{cor}[thm]{Corollary}
\newtheorem*{cor*}{Corollary}
\newtheorem{defn}[thm]{Definition} % definition numbers are dependent on theorem numbers
\newtheorem*{defn*}{Definition}
\newtheorem{ex}[thm]{Example} % same for example numbers
\newtheorem*{ex*}{Example}
\newtheorem{rem}[thm]{Remark}
\newtheorem*{rem*}{Remark}
\newtheorem{notation}[thm]{Notation}
\newtheorem*{notation*}{Notation}

\usepackage[headsepline]{scrlayer-scrpage}
\ihead{Bachelor Thesis -- Distribution-Valued Games}
\ohead{Vincent Bürgin}

\title{\huge Distribution-Valued Games}
\subtitle{\LARGE Overview, Analysis, and a \\Segmentation-Based Approach}

\begin{document}
    \pagenumbering{roman}
    \begin{titlepage}
        \newcommand{\authornamefirst}{Vincent}
        \newcommand{\authornamelast}{Bürgin}
        \newcommand{\worktitle}{Distribution-Valued Games}
        \newcommand{\worksubtitle}{Overview, Analysis, and a Segmentation-Based Approach}
        \newcommand{\thesistype}{Bachelor Thesis}
        \newcommand{\thesisdate}{September 25, 2020}
        \newcommand{\thesisprofDeMeer}{Prof. Dr.-Ing. Hermann de Meer}
        \newcommand{\thesisprofWirth}{Prof. Dr. Fabian Wirth}
        \newcommand{\supervisor}{Ali Alshawish,~M.~Sc.}
        \newcommand{\chairDeMeer}{Chair of Computer Networks \& Communications}
        \newcommand{\chairWirth}{Chair of Dynamical Systems}
        \vspace{1cm}
        
        \begin{center}
            \begin{tabular}{c}
                \includegraphics[width=6.5cm]{Pictures/logouni.pdf}
            \end{tabular}
            
            \vspace{3cm}
            \Large University of Passau \\
            \Large Faculty of Computer Science and Mathematics \\
            \vspace{0.3cm}
            {\Large \textbf{\chairDeMeer} } \\
            {\large \thesisprofDeMeer } \\[0.4cm]
            {\Large \textbf{\chairWirth} } \\
            {\large \thesisprofWirth}
        \end{center}
        
        \vspace{2.5cm}        
        \begin{center}
            {\textbf{\huge \thesistype}} % Master Thesis, Programming Project
        \end{center}
        
        \begin{center}
%            \settowidth{\baselineskip}{0.4cm}
            {
                \LARGE \worktitle \\[0.25cm]
                \Large \worksubtitle
            }  \\[0.8cm]
            {\Large
                \authornamefirst~\authornamelast
            }
        \end{center}
        
        \vfill { \vfill {\large
                \begin{tabular}[l]{llll}
                    Date:       & \thesisdate %%(\LaTeX{}$2_\epsilon$ run \today)
                    \smallskip \\
                    Supervisors:   & \thesisprofDeMeer \\
                    &\thesisprofWirth \\
                    & \supervisor
            \end{tabular}}
        }
    \end{titlepage}
%    
%    
%    \newpage\thispagestyle{empty}~
    \begin{abstract}
        The paper \cite{bib:rassGameRiskManagI} introduced \emph{distribution-valued games}. This game-theoretic model uses probability distributions as payoffs for games in order to express uncertainty about the payoffs. 
        The player's preferences for different payoffs are expressed by a stochastic order which we call the \emph{tail order}.
        
        \sloppypar{
        This thesis formalizes distribution-valued games with preferences expressed by general stochastic orders, and specifically analyzes properties of the tail order.
        It identifies sufficient conditions for tail-order preference to hold, but also finds that some claims in \cite{bib:rassGameRiskManagI} about the tail order are incorrect, for which counter-examples are constructed.
        In particular, it is demonstrated that a proof for the totality of the order on a certain set of distributions contains an error;
        the thesis proceeds to show that the ordering is not total on the slightly less restricted set of distributions with non-negative bounded support.
        It is also shown that not all tail-ordered games have mixed-strategy Nash equilibria, and in fact almost all tail-ordered games with finitely-supported payoff distributions can only have a Nash equilibrium if they have a pure-strategy Nash equilibrium.}
        
        The thesis subsequently extends an idea from \cite{bib:tweakableStochasticOrders} and proposes a new solution concept for distribution-valued games. This concept is based on constructing multi-objective real-valued games from distribution-valued games by segmenting their payoff distributions.
    \end{abstract}

%    \newpage\thispagestyle{empty}~

    \setstretch{1.08}
    \tableofcontents
    \thispagestyle{empty}
    \setstretch{1.18}
    
%    \newpage\thispagestyle{empty}~
    
    \chapter{Introduction}
    \pagenumbering{arabic}
    Game theory studies \emph{games} that are played by multiple independent players with different objectives, and analyzes the players' strategic possibilities.
    It has a wide range of applications in economics and risk management, and can be fruitfully used in cyber security as well.
    A \emph{game} models a situation that is determined by the actions its players take independently: Every player has a set of strategies to choose from, and the \emph{outcomes} (or \emph{payoffs}) the players obtain depend on the combination of all the strategies the players choose.
    In the classical setting, the payoffs are represented by numbers: A numeric payoff can for example be interpreted as a monetary reward the player gets, or as a more abstract \emph{utility} the outcome situation has for the player.
    
    However real-world settings tend to involve a lot of uncertainty, and it may be hard to specify a clear-cut number as outcome of a certain situation. The economical branch of decision theory provides tools for dealing with stochastic outcomes instead: It uses \emph{preference relations} between \emph{lotteries} to formalize rational decisions in such situations.
    
    A natural idea is to extend the existing theory of real-valued games by allowing probability distributions as payoffs, and such a model will be the basis for this bachelor thesis. This model of \emph{distribution-valued games} was first introduced and used in several papers by Stefan Rass et al (e.g. \cite{bib:rassGameRiskManagI,bib:rassTotalOrderingOnLossDistributions}). The idea is that the outcomes of the players are stochastic experiments, modeled by probability distributions, and each player’s preferences for particular outcomes are represented by a preference relation on the distributions, a \emph{stochastic order}. As an example, the preference of a player might express that the player tries to maximize the expected gain, but is also willing to sacrifice some expected gain if this reduces the risk of very unfortunate outcomes.
    The thesis focuses on a particular stochastic order, which we call the tail order: It was introduced in \cite{bib:rassGameRiskManagI} together with distribution-valued games, and expresses preferences that are maximally risk-averse. It is intended to be used on loss (instead of payoff) distributions, and it prefers one distribution over another if the other distribution assigns some (arbitrarily tiny) larger probability to a higher loss: In other words, its aim is to minimize the worst-case loss. 
    The ordering is defined based on moment sequences, and an analysis of its properties makes up a substantial part of this thesis. In particular, we show that not all of its properties that were claimed when it was first introduced hold true: For one, we show that an asserted characterization of the ordering based on the density functions of the involved distributions only holds in one direction in form of a sufficient condition. This result leads to the interesting question between which distributions the ordering is total: In particular, the thesis shows that even the set of distributions with bounded non-negative support contains elements that are incomparable by the ordering. The second issue we go into is that not all distribution-valued games with tail-order preferences have Nash equilibria: In particular, we show that they can only have mixed-strategy equilibria under very specific conditions.
    Finally, the thesis presents an alternative stochastic ordering based on segmenting loss distributions which was constructed by Ali Alshawish: This ordering was introduced with the goal of tweaking the tail order such that risk attitudes not maximally pessimistic are possible. The thesis then proposes a new approach of using the segmentation idea to turn a distribution-valued game into a multi-objective real-valued game, and shows how such games can be solved via the existing theory of Pareto-Nash equilibria.
    
    The primary contributions of this thesis are that it presents a complete formalization of distribution-valued games (Sections \ref{sec:generalizedPayoffGames} and \ref{sec:distributionValuedNormalFormGames}), critically analyzes the tail order, provides counterexamples to several misconceptions in the original publications about this ordering (Sections \ref{sec:stochasticTailOrder} and \ref{sec:equilibriaInTailOrderedGames}), and proposes the Pareto-Nash equilibria of a game obtained by distribution segmentation as an alternative solution concept for distribution-valued games (Chapter \ref{chap:segmentingLossDistributions}).
    To provide the basis for the discussions in those chapters, Chapter \ref{chap:nonCooperativeRealValuedGameTheory} reviews the basics of classical non-cooperative game theory. This serves as an introduction to the concepts that are later generalized, and we also introduce some results that are needed for the analysis of Nash equilibria with respect to the tail order.
    Chapter \ref{chap:mathematicalPreliminaries} introduces mathematical preliminaries: We shortly go over important concepts from probability theory that the subsequent chapters rely on, and review basic notions from decision theory.

    \subsubsection{Related Work}
    
    Directly related to this work are the papers by Stefan Rass et al \cites{bib:rassGameRiskManagI,bib:rassGameRiskManagII,bib:rassTotalOrderingOnLossDistributions} that introduce the tail order and use it in the context of risk management, and by Ali Alshawish et al \cite{bib:alshawishQuasiPurificationOfMixedGameStrategies,bib:tweakableStochasticOrders} which use the tail order and define the related \emph{tweakable stochastic order}, which serves as a basis for the ideas of Chapter \ref{chap:segmentingLossDistributions}.
    
    Regarding further related work, there are of course many publications on classical (real-valued) game theory, and the books mainly used for this thesis are \cites{bib:fudenbergGameTheory,bib:nisanAlgorithmicGameTheory,bib:matsumotoGameTheory}.
    There is also literature on stochastic orders, e.g. \cite{bib:shakedStochasticOrders}.
    There are however surprisingly few publications that concern models similar to distribution-valued games, or even other generalized payoffs that are not real numbers. There are multiple well-known types of games that include randomness, but none of them match the model of distribution-valued games: 
    There is the notion of \emph{stochastic games} (see \cite{bib:shapleyStochasticGames}) which are played in multiple rounds and include a game state that changes between the rounds and determines the payoff structure. In these games, only the state transitions depend on chance, but not the payoffs. 
    Another variant are \emph{moves by nature} in the theory of \emph{extensive-form games}, which occur for example in \emph{Bayesian games} (see \cite[Chapter 6 and Section 8.3]{bib:fudenbergGameTheory}): While payoffs can depend on chance in such models, an important difference lies in their equilibrium concepts, since the payoff distributions are condensed to an expected value before comparing them, unlike the comparison by \emph{stochastic orders} in our model. 
    The models closest to Rass' distribution-valued games published prior to it seem to be \emph{stochastic cooperative games} which use distribution-valued outcomes rated by stochastic orders, and \emph{non-cooperative games with fuzzy-number payoffs}. 
    The former (e.g. \cite{bib:suijsCooperativeGamesWithStochasticPayoffs,bib:fernandezCoresOfStochasticCoopGamesWithStochasticOrders}) however concern only cooperative games whose theory differs from the theory of non-cooperative games that our model lives in. 
    The latter uses fuzzy numbers instead of probability distributions (e.g. \cite{bib:maedaCharacterizationOfEquilibriumFuzzyPayoff, bib:cevikelSolutionsFuzzyMatrixGames}): The exact relationship between fuzzy and probabilistic methods is rather complicated (see e.g. the discussion in \cite{bib:kandelDistinctionBetweenFuzzyAndStatisticalMethods}), yet they are certainly different concepts.
    
    There is also not much literature on the more general case of games with payoffs in an arbitrary set ordered by a preorder.
    This thesis defines its own generalized model in Section \ref{sec:generalizedPayoffGames} so we can properly work with such games without ambiguity about the definitions. Similar definitions are given in \cite{bib:rozenEquilibriaInGamesWithOrderedOutcomes}, and apparently already in the much older Russian-language article \cite{bib:vorobevThePresentStateOfGameTheory} it cites, but at least \cite{bib:rozenEquilibriaInGamesWithOrderedOutcomes} does not include a definition for mixed extensions as general as we need.
    It should be remarked though that \cite[Section 1.2.1]{bib:nisanAlgorithmicGameTheoryCh1Basic} mentions an even more general model which, instead of using payoffs, defines the players' preferences directly between the strategy profiles.
    
    The tail-ordered games defined by Rass are closely related to games with vector-valued payoffs ordered by a lexicographic order, as will be worked out in Section \ref{sec:equilibriaInTailOrderedGames}. No literature on such lexicographically-ordered games could be found either:
    The closest examples are the articles \cites{bib:salvadorLeximin, bib:quantPropernessProtectiveness} about \emph{leximin} preferences used in social choice theory and the related concept of \emph{protective behavior} in games, and the article \cite{bib:bouyerConcurrentGamesWithOrderedObjectives} which concerns graph games and orders objectives lexicographically (among other ways).
    The tail order itself is defined based on moment sequences of probability distributions: A relevant question for the thesis thus is which real sequences are moment sequences, a question known as the \emph{moment problem}. The literature on this subject is rich, and there are both classical and more recent publications, e.g. \cite{bib:hausdorffMomentprobleme,bib:akhiezerClassicalMomentProblem,bib:chiharaIndeterminateHamburgerMomentProblems,bib:schmuedgenTheMomentProblem}. 
    However the question whether two moment sequences can alternate, which is relevant to the tail order and is answered in Section \ref{subsec:tailOrderTotality}, seems not to have been considered in the literature before.
    
    \enlargethispage*{2\baselineskip}
    Finally, the theory of multi-objective games and Pareto-Nash equilibria we put to use in {Chapter \ref{chap:segmentingLossDistributions}} is well-developed, starting with the papers \cite{bib:blackwellVectorPayoffs} and \cite{bib:shapleyMultiobjectiveEquilibriumPoints}. The more recent \cite{bib:paretoNashEquilibria} generalizes some theorems from the earlier papers. Furthermore, an overview of solution concepts other than Pareto-Nash equilibria is given in \cite{bib:ghoseSolutionConceptsMultiobjective}.

    \chapter{Mathematical Preliminaries: Probability and Decision Theory}
    \label{chap:mathematicalPreliminaries}
    
    The following pages present basic concepts from probability and decision theory that are relevant as background for the thesis.
    The first section contains standard definitions from probability and measure theory that can be found in textbooks such as \cite{bib:billingsleyProbabilityAndMeasure}.
    The section about decisions under uncertainty and risk is adopted from \cite[Sections 3.5 - 5.3]{bib:doersamGrundlagenDerEntscheidungstheorie} and \cite[Sections I.1, I.2]{bib:wakkerProspectTheory}.
    The thesis uses standard mathematical notation, yet some notations used are worth mentioning:
    We write $\N = \set{1, 2, 3, \dots}$ for the natural numbers not including zero, and $\N_0 \coloneqq \N \cup \set{0}$.
    The non-negative real numbers are denoted by $\Rp$. If $n \in \N$, we write $[n] \coloneqq \set{1, 2, \dots, n}$ for the first $n$ natural numbers.
    If $A \subseteq \R$ is a set, we write $\1_A$ for its \emph{characteristic function} that is defined by $\1_A(x) = 1$ iff $x \in A$, and $\1_A(x) = 0$ otherwise.
    We say that two real sequences $(a_n)_{n \in \N_0}$, $(b_n)_{n \in \N_0}$ \emph{alternate} if for every $K \in \N_0$ there are indices $m, l \geq K$ such that $a_m > b_m$ and $a_l < b_l$. We say that two functions $f, g: \R \to \R$ \emph{alternate} on an interval $[a, b]$ if for every $x_0 \in (a, b)$ there are $x_1, x_2 \in (x_0, b)$ such that $f(x_1) < g(x_1)$ and $f(x_2) > g(x_2)$.
    
    \section{Probability Theory}
    Probability theory presents a framework to do calculations with probabilities assigned to the outcomes of probabilistic experiments. It uses measure theory to unify the different settings needed for finitely, countably infinitely and uncountably infinitely many outcomes (\cite[Section 1]{bib:billingsleyProbabilityAndMeasure}). This section quickly goes over the basic concepts needed in the thesis.
    
    Assume $\Omega$ is some set, usually representing probabilistic outcomes. 
    A \emph{sigma algebra} $\A \subseteq \Pot(\Omega)$ represents probabilistic events and is a non-empty system of subsets of $\Omega$ which is closed under complements and countable intersections.
    A \emph{signed measure} $\mu$ on $(\Omega, \A)$ is a map $\mu: \A \to \R \cup \set{-\infty, \infty}$ that satisfies $\mu(\emptyset) = 0$, and $\mu(\bigcup_{n \in \N} A_n) = \sum_{n \in \N}\mu(A_n)$ for any countable collection $(A_n)_{n \in \N}$ of pairwise-disjoint sets from $\A$ \cite[Problem 32.12]{bib:billingsleyProbabilityAndMeasure}.
    $\mu$ is a \emph{measure} if it only takes non-negative values, and is \emph{finite} if $\mu(\Omega) < \infty$.
    A \emph{probability measure} $P$ is a measure that satisfies $P(\Omega) = 1$.
    The tuple $(\Omega, \A, P)$ forms a \emph{probability space}, and for any set $A \in \A$, $P(A)$ represents its probability \cite[Section 2]{bib:billingsleyProbabilityAndMeasure}.
    A convex combination (\emph{mixture}) $\alpha P_1 + (1-\alpha) P_2,\, \alpha \in [0, 1]$ of probability measures $P_1, P_2$ is again a probability measure.
    
    If $\Omega = \R$, one needs to find a suitable sigma algebra: While the power set $\Pot(\R)$ is a sigma algebra, it is “too large” as it contains non-well behaved sets which prevent useful measures (\emph{Vitali's Theorem}).
    One resorts to the \emph{Borel sigma algebra} $\B$, defined as the smallest sigma algebra containing all intervals $[a, b] \subseteq \R$.
    The sets $B \in \B$ are called \emph{Borel sets}, and measures on $\B$ are called \emph{Borel measures}.
    An ubiquitous Borel measure is the \emph{Lebesgue measure} $\lambda$ which assigns Borel sets their natural “volume” and is uniquely determined as the Borel measure that maps closed intervals to their length, i.e. $\lambda([a, b]) \coloneqq b-a$, $a \leq b$.
    The \emph{point-mass} (or \emph{Dirac}) \emph{measure} $\delta_x$ is the Borel probability measure that assigns all probability mass to one point $x \in \R$, i.e. $\delta_x(B) = \1_B(x)$.
    \cite[p.23, p.45-47, p.177]{bib:billingsleyProbabilityAndMeasure}

    A function $f: \Omega \to \R$ is \emph{Borel-measurable} if $\forall B \in \B: f\inv(B) \in \A$.
    Its \emph{Lebesgue integral with respect to a Borel measure $\mu$}, denoted by $\lintegral{}{f}{\mu}$, can be defined in three cases:
    First, the integral can be defined if $f$ only takes non-negative values.
%    , it is at least \emph{quasi-integrable}. 
    Secondly it can be defined if $f$ takes negative values, but either its positive part $f^+ = \max(0, f)$ or its negative part $f^- = \max(0, -f)$ have a finite integral.
    If both parts have a finite integral, $f$ is called \emph{integrable}. This is equivalent to $\lintegral{}{\abs{f}}{\mu} < \infty$.
    In contrast, one part has an infinite integral, the integral of $f$ is either $\infty$ or $-\infty$; if both parts have infinite integrals, the integral cannot be defined (for example, $\id: x \mapsto x$ has no integral).
    The integration can be restricted to a set $A \in \A$, denoted as $\lintegral{A}{f}{\mu} \coloneqq \lintegral{}{\1_A*f}{\mu}$. To specify the integration variable, one writes $\lintegral{A}{f(x)}{\mu(x)}$.
    If $\mu = \lambda$ and $\Omega = \R$, the integral coincides with the Riemann integral in many cases that occur in practice, in particular if $f$ is bounded, has bounded domain and is Riemann-integrable. \cite[Sections 13, 15-17]{bib:billingsleyProbabilityAndMeasure}
    
    A \emph{real-valued random variable} on $(\Omega, \A)$ is a function $X: \Omega \to \R$ which is Borel-measurable, i.e. $\forall B \in \B: X\inv(B) \in \A$. 
    A convenient notation for such preimages is $\set{X \in B} \coloneqq X\inv(B)$, such that $P(\set{X \in B})$ denotes the probability that “$X$ takes a value in $B$”.
    $X$ induces a Borel probability measure, its \emph{distribution} or \emph{pushforward measure} $P^X: B \mapsto P(\set{X \in B})$. One can think of $P^X$ as only describing the random variable's distribution while ignoring the details of the underlying $(\Omega, P)$.
    Closely related is the \emph{(cumulative) distribution function} (cdf) $F_X: \R \to [0, 1], x \mapsto P(\set{X \leq x})$.
    $P^X$ and $F_X$ uniquely determine each other and both represent the distribution of $X$; on the other hand, there can be many different random variables on a fixed probability space that all have the same distribution.
    Yet many properties of $X$ only depend on $P^X$, and can therefore be formulated in terms of probability measures.
    \cite[Sections 14, 20]{bib:billingsleyProbabilityAndMeasure}
    
    Let $P$ be a Borel probability measure.
    $P$ is \emph{discrete} if $P(S) = 1$ for a countable set $S = \set{x_1, x_2, \dots}$: It can then be represented by its \emph{probability mass function} (pmf) $f: \R \to [0, 1], x \mapsto P(\set{x})$ that satisfies $\sum_{x_i \in S} f(x_i) = 1$, and $x \notin S \Rightarrow f(x) = 0$.
    $P$ is \emph{absolutely continuous} (AC) if it has a \emph{(probability) density function} (pdf) $f: \R \to \Rp$, such that $\forall B \in \B: P(B) = \lintegral{B}{f(x)}{\lambda(x)}$.
    $P$ is \emph{continuous} if its distribution function $F$ is continuous, or equivalently $P(\set{X = c}) = 0$ for all $c \in \R$.
    Absolutely continuous probability measures are continuous, but the converse is not true in general (a counterexample is the \emph{Cantor distribution}): In particular, there are probability measures which are neither AC nor discrete (nor a mixture of AC and discrete measures).
    Both absolute continuity and discreteness are special cases of a more general concept: $P$ has a density $f$ \emph{with respect to a measure $\mu$} if $\forall B \in \B: P^X(B) = \lintegral{B}{f(x)}{\mu(x)}$; For discrete distributions, the mass function $f$ can be seen as a density with respect to the \emph{counting measure} $\#$ which assigns to each set its cardinality.
    If $P$ has a $\mu$-density $f$, this allows to compute the integral of some $P$-integrable function $g$ with respect to $P$ as $\lintegral{A}{g}{P} = \lintegral{A}{f g}{\mu}$. 
    \cite[Sections 16, 20, 31]{bib:billingsleyProbabilityAndMeasure}
    
    Some set $A \subseteq \R$ is \emph{null set} with respect to measure $\mu$ if $A$ is contained in a measurable set $\tilde{A}$ with $\mu(\tilde{A}) = 0$.
    A condition holds \emph{almost surely} with respect to a probability measure $P$ if the set where it does not hold is a null set.
    If $P$ has a $\mu$-density $f$, this implies that all null sets with respect to $\mu$ are null sets with respect to $P$ ($P$ is \emph{absolutely continuous} with respect to $\mu$).
    A probability measure $P$ is \emph{supported} on a measurable set $A$ if $P(A) = 1$. If $P$ is a Borel probability measure on $\R$, we write $\supp(P)$ for the \emph{support} of $P$: We use the convention that $\supp(P)$ is the smallest closed set with probability one, i.e. the intersection of all closed $B \in \B$ such that $P(B) = 1$.
    \cite[p.63, p.170, Theorem 31.7]{bib:billingsleyProbabilityAndMeasure}
    
    The \emph{expected value} of a random variable is denoted by $E(X) \coloneqq \lintegral{\Omega}{X}{P} = \lintegral{\R}{x}{P^X(x)}$. It is only defined if $X$ is integrable, i.e. $E(\abs{X}) = \lintegral{\Omega}{\abs{X}}{P} < \infty$. We write $E(P) \coloneqq \lintegral{\R}{x}{P(x)}$ if $P$ is a Borel probability measure and $\lintegral{\R}{\abs{x}}{P(x)} < \infty$.
    The expected value is the first of the distribution's \emph{moments}:
    If $p \in \N_0$, we say that $X$ is \emph{$p$-integrable}, or $X \in \mathcal{L}^p$, if $E(\abs{X}^p) < \infty$;
    In this case, define the \emph{$p$-th moment} of $X$ as $E(X^p)$. Analogously, define the $p$-th moment of a Borel measure $P$ by $m_p(P) \coloneqq \lintegral{\R}{x^p}{P(x)}$ if $\lintegral{\R}{\abs{x}^p}{P(x)} < \infty$.
    $P$ is a probability measure if and only if $m_0(P) = 1$, because $m_0(P) = P(\R)$.
    While $\mathcal{L}^p$ contains all random variables over a fixed probability space that have a moment of $p$-th order, there is no standard notation for the set of Borel probability measures that have moments of $p$-th order. In this thesis, we denote this set by $M^p$, and write $M \coloneqq \bigcap_{p \in \N} M^p$ for the set of Borel probability measures that have moments of all orders.
    \cite[Section 21]{bib:billingsleyProbabilityAndMeasure}
    
    \section{Decision Theory}
    
    Decision theory is the theory of selecting one of multiple alternatives in a scenario where the exact outcomes of the alternatives are uncertain.
    There is a state space $S$, also called \emph{states of nature}: Exactly one of the states is considered to be true, but it is unknown which.
    There is a set of \emph{outcomes} which we will assume to be $\R$.
    The alternatives one has to decide between are called \emph{prospects} or \emph{lotteries}, and are modeled as maps $x: S \to \R$. Prospects are assumed to only take finitely many values. There is a short notation for prospects: E.g. if $S = \set{s_1, s_2, s_3}$, one writes $(s_1: 50, s_2: 30, s_3: 121)$ for the prospect that assigns to the three states the values $50, 30$, and $21$, respectively. \cite[Section 1.1]{bib:wakkerProspectTheory}
    
    The theory distinguishes between \emph{decisions under uncertainty} and \emph{decisions under risk}. The difference is that under \emph{risk}, the states of nature have probabilities assigned to them, while under \emph{uncertainty}, no probabilities are assumed.
    A common way to transform a probability under uncertainty into one under risk is to assume that all states of nature are equally likely (\emph{Laplace's method}).
    There are other methods to make a decision under uncertainty, for example the \emph{maximin} and \emph{maximax} methods, \emph{Huwicz' rule}, and the \emph{Savage-Niehans} or \emph{regret minimization rule}. \cite[Section 4]{bib:doersamGrundlagenDerEntscheidungstheorie}
%    All criteria have shortcomings, as no criterion manages to satisfy every entry of a list of axioms one would like to impose on a rational decision rule:
%    For the maximin and maximax methods and Hurwicz' rule, the decision can change if a constant is added to all outcomes of a certain state, and also a prospect that is dominated by another one, and strictly so in at least one entry, can still be chosen.
%    The Savage-Niehans rule can change its preference between prospects $a$ and $b$ if a third prospect $c$ is added.
%    Laplace's rule can change its decision when two states of nature with identical outcomes are merged to a single state.
    
    For decisions under risk, the short notation is changed to represent the probabilities instead of the states, e.g. $(5\%: 50, 10\%: 30, 85\%: 121)$.
    Prospects can be interpreted as random variables in this model, and it seems natural to compare them by their \emph{expected values}. However, this method is not considered to accurately represent every decision maker's attitude to risk:
    For example, it seems plausible that many people would prefer the certain payoff {(100\%: {10 000}\,\euro)} over the gamble (1\%:  1 000 001\,\euro, 99\%:  0\,\euro), even though the latter prospect has a greater expected value.
    Bernoulli argued that monetary rewards have diminishing marginal returns: The more money someone gets, the less he or she cares about getting one additional unit. In other words, the \emph{utility} the money has for the decision maker increases less than proportionally to the amount.
    In formal terms, one associates with a decision maker a \emph{utility function} $u: \R \to \R$, defined as a monotonically increasing function from the outcome set to the reals.
    To decide between two prospects by \emph{expected utility}, one applies a utility function $u$ to the outcomes and decides between the resulting prospects by expected value.
    A utility function that is \emph{concave}, i.e. grows slower than proportionally to its argument, is associated with \emph{risk-averse} behavior.
    On the other hand, a \emph{convex} utility function is associated with \emph{risk-seeking} behavior. Commonly, functions such as $x \mapsto x^2, x \mapsto \sqrt{x}$, or $x \mapsto \ln(x)$ are used.
     \cites[Sections 2.1, 2.2]{bib:wakkerProspectTheory}[Section 5]{bib:doersamGrundlagenDerEntscheidungstheorie}
    
    More general preferences between prospects can be captured by defining a binary \emph{preference relation} $\preccurlyeq$ between prospects:
    A relationship $x \succcurlyeq y$ expresses that the decision maker “is willing to choose $x$ from $\set{x, y}$” (\cite[p. 14]{bib:wakkerProspectTheory}).
    Reasonable assumptions for preference relations include \emph{reflexivity} and \emph{transitivity}, in which case the preference relation is a \emph{preorder} on the prospects.
    The theory as described in \cite{bib:wakkerProspectTheory} restricts prospects such that they can only take finitely many values. If interpreted as random variables, this means that only random variables with finite support are considered. In Chapter \ref{chap:gamesWithDistributionalPayoffs}, we will drop this restriction and define preference relations between arbitrary Borel probability distributions on the real numbers, then called \emph{stochastic orders} (e.g. \cite{bib:shakedStochasticOrders}).
    
    \chapter{Non-Cooperative Game Theory}
    \label{chap:nonCooperativeRealValuedGameTheory}
    In this chapter, we will introduce the basic notions of non-cooperative game theory.
    Game theory is applied in scenarios where multiple agents, called \emph{players}, make decisions independently of another, and each tries to achieve the best outcome for themselves. This theory is called \emph{non-cooperative game theory}, and it is characterized by players not being able to make enforceable agreements \cite[p.1]{bib:harsanyiTheoryOfEquilibriumSelection}.
    In contrast, there is \emph{cooperative game theory} which lets players cooperate and form coalitions to achieve a better outcome.
    The underlying theory of the two variants is quite different, and this thesis focuses only on the non-cooperative theory.
    The classical example of a non-cooperative game is the \emph{prisoner's dilemma}:
    \begin{ex}[Prisoner's Dilemma]
        Two criminals were caught and are being held in different cells. The police does not have substantial evidence against them, so a deal is offered to each of the two: If one prisoner confesses to the crime and hands over evidence that helps prosecuting his partner, the prisoner can go into a witness protection program and stay out of prison, while the partner will be sentenced to five years in prison. Yet if both prisoners choose to confess, the prosecution does not need a key witness, and both will have to serve an (only slightly reduced) sentence of four years. However, if both prisoners refuse to confess, the prosecutors, based on the little evidence they have, will only be able to sentence them to one year in prison each.
        
        The game can be represented by a table: 
        The rows represent the first player's strategies, the columns the second player's strategies,  and the cells contains the years the first and second player face in prison, respectively.
        \begin{gather*}
            \centering
            \begin{tabular}{r|c|c|}
            	                    & Prisoner 2 does not confess & Prisoner 2 confesses \\ \hline
            	Prisoner 1 does not confess &       1\;/\;1       &   5\;/\;0    \\ \hline
            	   Prisoner 1 confesses     &       0\;/\;5       &   4\;/\;4    \\ \hline
            \end{tabular}
        \end{gather*}

        So what should the prisoners do?
        If they were able to make a binding contract about the situation, they would surely agree not to confess, and both only spend one year in prison. But since there is no way to do so, each prisoner's fate depends on the decision of his partner, and both have to watch out not to be betrayed by their partner and get an even longer prison sentence than by confessing.
%        Obviously if the prisoners could make a binding agreement, they would be best of by both not confessing, and both only facing one year in prison.
%        However, if one prisoner does not confess, he risks being betrayed by the other prisoner and face even more time in prison; 
        Therefore in non-cooperative game theory, somewhat counter-intuitively, the solution to the game is that both prisoners confess and both face four years in prison instead of just one.
        They both have to accept going to prison for four years, since they cannot make a binding agreement, and this is the only way to avoid being betrayed by the other prisoner.
        \label{ex:prisonersDilemma}
        \label{ex:gameTheoryIntroductoryExample}
    \end{ex}

    We will now define such games and their solution concepts mathematically.

    \begin{defn}[see {\cites[p.4]{bib:fudenbergGameTheory}[p.19]{bib:matsumotoGameTheory}[p.9]{bib:nisanAlgorithmicGameTheoryCh1Basic}}]~\\
        A real-valued \emph{normal form game} $G = (n, (S_1, \dots, S_n), (u_1, \dots, u_n))$ consists of 
        \begin{itemize}
            \item the number of players $n \in \N$,
            \item for each player $k \in [n]$, a set $S_k$ of available strategies,
            \item for each player $k \in [n]$, a payoff function $u_k: S \to \R$, where $S = \bigtimes\limits_{i \in [n]} S_i$ denotes the set of all possible combinations of the players' strategies.
        \end{itemize}
        Elements of $S_k$ are called \emph{strategies}, elements of $S$ are called \emph{strategy profiles}. $G$ is called \emph{finite} if $S$ is a finite set.
        \label{def:realValuedGames}
    \end{defn}

    \begin{rem}
        In this chapter, with the term \emph{game} we always mean a real-valued normal form game as in Definition \ref{def:realValuedGames}.
    \end{rem}
    
    Instead of specifying payoff functions $u_k$, we can specify cost functions $c_k$ with the semantics that players want to maximize payoffs, but minimize costs. For example, the “years in prison” in Example \ref{ex:prisonersDilemma} correspond to costs instead of payoffs. We can switch between those viewpoints by setting $u_k = - c_k$. For a strategy profile $s = (s_1, \dots, s_n) \in S$ and some player $k$, it is sometimes convenient to use the notation $s_{-k}$ for $s$ with the $k$-th coordinate omitted, and write $u_k(s_k, s_{-k})$ instead of $u_k(s)$. \cite[p.9-10]{bib:nisanAlgorithmicGameTheoryCh1Basic}
    
    \begin{notation}
        In finite two-player games, it is convenient to specify payoffs by matrices:
        We will then use matrices $A = (a_{ij})_{\scriptscriptstyle i\in[\abs{S_1}],\,j\in[\abs{S_2}]}, B = (b_{ij})_{\scriptscriptstyle i\in[\abs{S_1}],\,j\in[\abs{S_2}]}$ such that $a_{ij}$ and $b_{ij}$ correspond to the first/second player's payoffs under the $i$-th strategy of the first player and the $j$-th strategy of the second player.
    \end{notation}
    
    A common special case are zero-sum games:
    
    \begin{defn}[{\cite[p.4]{bib:fudenbergGameTheory}}]
        A \emph{two-player zero-sum game} is a game with two players, such that
        \begin{gather*} 
            \forall s \in S: u_1(s) + u_2(s) = 0.
        \end{gather*}
    \end{defn}

    In a zero-sum game, the two players play strictly “against each other”, and are adversaries in every possible scenario:
    one player wins exactly what the other loses, and there is no outcome that corresponds to mutual benefit. As \cite{bib:fudenbergGameTheory} notes, the important feature of these games is that the payoffs sum to a constant. 
%    This is because positive affine transformations ($x \mapsto ax + b$, $a > 0$) of all the payoffs keep the Nash equilibria of the game unchanged.
    Choosing this constant as zero is only for normalization.
    
    \section{Solution Concepts}
    Reasoning about rational strategies, as done in Example \ref{ex:gameTheoryIntroductoryExample}, is formalized by \emph{solution concepts}.
    The most prominent one is the concept of \emph{Nash equilibria}.
    Before we introduce Nash equilibria, we start with the simpler solution concept of \emph{dominant strategies}, where a player's best strategy is independent the other players' actions.
    
    \begin{defn}[Dominant strategy, see \cite{bib:nisanAlgorithmicGameTheoryCh1Basic}]
        Let $G$ be a game with $n$ players.
        A strategy $s_k \in S_k$ for a player $k \in [n]$ is a \emph{dominant strategy for player $k$} if 
        \begin{gather*}  
            \forall \tilde{s} \in S: u_k(s_k, \tilde{s}_{-k}) \geq u_k(\tilde{s}). 
        \end{gather*} 
        A strategy profile $(s_1, \dots, s_n) \in S$ is a \emph{dominant strategy solution} if all its individual strategies $s_i$ are dominant strategies.
    \end{defn}
    
    In the prisoner's dilemma \ref{ex:prisonersDilemma}, confessing is a dominant strategy for both prisoners:
    For example, if player 2 confesses, then player 1 is best off by confessing as well. If on the other hand player 2 does not confess, player 1 is \emph{also} best off by confessing, i.e. betraying player 2 and going into witness protection without a prison sentence.
    Dominant strategies lead to a obvious solution of the game if they exist, but many games do not have a dominant-strategy solution.
    A more sophisticated solution concept are Nash equilibria, which encode that for a given strategy profile, no player has an incentive to change their strategy when all other player's strategies stay as before. Nash equilibria represent a certain form of stability in a strategy profile.
    
    \begin{defn}[Nash equilibrium, see {\cite[p.11]{bib:fudenbergGameTheory}}]
        Let $G$ be a game with $n$ players.
        A strategy profile $s \in S$ is a \emph{Nash equilibrium} if
        \begin{gather*} 
            \forall k \in [n], \forall \tilde{s}_k \in S_k:~ u_k(s_k, s_{-k}) \geq u_k(\tilde{s}_k, s_{-k}).
        \end{gather*} 
        \label{def:nashEquilibriumRealValued}
    \end{defn}

    \begin{lemma}[Row/Column Criterion, see {\cite[p.14]{bib:matsumotoGameTheory}}]
        In a finite two-player game with payoff matrices $A, B$, the Nash equilibria correspond
        exactly to the indices $(i, j)$ where  $a_{ij}$ is maximal in its column, and $b_{ij}$ is maximal in its row.
        If the game is zero-sum, these are just the indices where $a_{ij}$ is both maximal in its column and minimal in its row.
    \end{lemma}
    \begin{proof}
        The criterion follows directly from the definition:
        Let $s_{1, i} \in S_1$ and $s_{2, j} \in S_2$ be the $i$-th/$j$-th strategy, respectively.
        $a_{ij}$ is maximal in its column if and only if $u_1(s_{1, i}, s_{2, j}) \geq u_1(s_{1, \tilde{i}}, s_{2, j})$ for all $s_{1, \tilde{i}} \in S_1$.
        Likewise, $b_{ij}$ is maximal in its row if and only if $u_1(s_{1, i}, s_{2, j}) \geq u_1(s_{1, i}, s_{2, \tilde{j}})$ for all $s_{2, \tilde{j}} \in S_2$.
        In the zero-sum case, as $a_{ij} = -b_{ij}$, maximizing $b_{ij}$ over all $j$ is equivalent to minimizing $a_{ij}$ over all $j$.
    \end{proof}

    \begin{lemma}
        In a two-player zero-sum game, all Nash equilibria have the same payoff.
        The unique payoff of player 1 under a Nash equilibrium is then called the \emph{value} of the game.
    \end{lemma}
    \begin{proof}
        \cite[p.15]{bib:matsumotoGameTheory} gives a proof for the case of finitely many strategies using the row-column criterion, but the argument works in the general setting:
        Let $(s_1, s_2)$, $(\tilde{s}_1, \tilde{s}_2)$ be two Nash equilibria.
        Then
        \begin{gather*}
            u_1(s_1, s_2) \geq u_1(\tilde{s}_1, s_2) = -u_2(s_2, \tilde{s}_1) \geq -u_2(\tilde{s}_2, \tilde{s}_1) = u_1(\tilde{s}_2, \tilde{s}_1).
        \end{gather*}
        By a symmetric argument, $u_1(\tilde{s}_2, \tilde{s}_1) \geq u_1(s_1, s_2)$, and therefore $u_1(\tilde{s}_2, \tilde{s}_1) = u_1(s_1, s_2)$.
    \end{proof}
    
    The next example illustrates the solution concepts we introduced:
    
    \begin{ex}[Dominant Strategy Solutions and Nash Equilibria]
        In example \ref{ex:prisonersDilemma} we saw a game with a dominant strategy solution, which in fact also is the unique Nash equilibrium of the game.
        Now the game in (a)
        shows that Nash equilibria and dominant strategy solutions are indeed different concepts. The first player has no dominant strategy, so there is no dominant strategy solution. But the game does have a Nash equilibrium: If both players play their first strategy, the payoff 1 for the first player is maximal in its column, and the payoff 1 for the second player is maximal in its row.
        
        The game in (b)
        shows that Nash equilibria need not be unique: Both the upper-left and the lower-right cell correspond to Nash equilibria.
        
        \begin{figure}[h]
            \centering
            \begin{subfigure}[t]{0.49\textwidth}
                \centering
                \begin{tabular}{c|c|c|}
                	   &   b1    &   b2    \\ \hline
                	a1 & 1\,/\,1 & 0\,/\,0 \\ \hline
                	a2 & 0\,/\,3 & 1\,/\,2 \\ \hline
                \end{tabular}
                \caption{Game with Nash equilibrium, but no dominant strategy solution.}
                \label{fig:nashEquilibriumNoDominantStrategy}
            \end{subfigure}
            \begin{subfigure}[t]{0.49\textwidth}
                \centering
                \begin{tabular}{c|c|c|}
                    &   b1    &   b2    \\ \hline
                    a1 & 1\,/\,-1 & 0\,/\,-2 \\ \hline
                    a2 & 0\,/\,-2 & 5\,/\,3 \\ \hline
                \end{tabular}
                \caption{Game with two Nash equilibria.}   
                \label{fig:twoNashEquilibria}
            \end{subfigure}
        \end{figure}
    \end{ex}
    
    \section{Mixed-Strategy Extensions}
    The example games up to here always had a finite number of strategies for each player.
    However, such finite games do not always have Nash equilibria. We next introduce the concept of \emph{mixed strategies}: We will allow each player to “mix” between multiple strategies, interpreted as playing each of them with a certain probability.
    The next example shows how mixed strategies can be used to find an equilibrium for the Rock-Paper-Scissors game.
    
    \begin{ex}[Rock-Paper-Scissors]
        In this game-theoretic formulation of the well-known game \emph{Rock-Paper-Scissors},
        the players have strategy sets $S_1 = S_2 = \set{\text{Rock}, \text{Paper}, \text{Scissors}}$, where paper beats rock, rock beats scissors, and scissors beat paper. The game is zero-sum, and can be represented by the first player's payoff matrix:
        \begin{figure}[h]
            \centering
            \begin{tabular}{c|c|c|c|}
            	         & Rock & Paper & Scissors \\ \hline
            	  Rock   &  0   &  -1   &    1     \\ \hline
            	 Paper   &  1   &   0   &    -1    \\ \hline
            	Scissors &  -1  &   1   &    0     \\ \hline
            \end{tabular}
        \end{figure}
    
        There is no Nash equilibrium if the players only have those three strategies available: For example, if player 1 plays rock, player 2 can beat it by playing paper; but if player 2 plays paper, player 1 has an incentive to switch to scissors, and so on. Exactly this kind of instability is not allowed for a Nash equilibrium, so this example shows that not all games have Nash equilibria.
        
        Instead of committing to a single hand gesture and play it, players should rather play one of the three gestures unpredictably: While a player committed to a single one of the three strategies can be easily beaten, this is not the case if the player picks each of the three strategies with equal probability.
        \label{ex:rockPaperScissors}
    \end{ex}

    We will now define mixed-strategy extensions, where the randomization described in the example becomes possible: The strategy of playing each of the three gestures with probability $\frac{1}{3}$ becomes a valid strategy itself, a \emph{mixed strategy}. The original strategies where no mixing occurs are then called \emph{pure strategies}. The strategies in a mixed-strategy extension are functions that assign to each pure strategy a probability, and the payoffs are calculated as expected values.

    \begin{defn}[Mixed Extensions, e.g. \cite{bib:matsumotoGameTheory}]
        Let $G = (n{,}(S_1{,\dots,}S_n){, }(u_1{,\dots,}u_n))$ be a finite normal-form game.
        Its \emph{mixed extension} $\hat{G} = (n, (\Delta_1, \dots, \Delta_n), (\hat{u}_1, \dots, \hat{u}_n))$ is defined by the following components
        for each player $k \in [n]$:
        \begin{itemize}
            \item 
            The strategy set $\Delta_k$ represents \emph{mixed strategies}:
            \footnote{The exact notation used differs across the literature. Our notation $\Delta_k$ for the mixed-strategy sets is used, for example, in \cite{bib:quantPropernessProtectiveness}.}
%            , i.e. the probability distributions over $S_k$:
            \begin{gather*} 
                \Delta_k \coloneqq \biggset{ \delta : S_k \to \Rp \biggmid \sum_{s_k \in S_k} \delta(s_k) = 1 } \subseteq \Rp^{S_k}.
            \end{gather*} 
%            A strategy $p \in \Delta_k$ has the semantics that each $s_k \in S_k$ is played with probability $p(s_k)$.
            
            \item $\Delta \coloneqq \bigtimes\limits_{i \in [n]} \Delta_i$ denotes the set of \emph{mixed strategy profiles}.
            
            \item
            The utility function $\hat{u}_k: \Delta \to \R$ maps to each mixed strategy profile the \emph{expected value} of the $k$-th player's payoff under that strategy profile:
            \begin{gather}
                \hat{u}_k: 
                (\delta_1, \dots, \delta_n) 
                \mapsto
                \sum_{(s_1, \dots, s_n) \in S} \biggpars{\prod_{i=1}^{n} \delta_i(s_i)} * u_k( (s_1, \dots, s_n) ).
                \label{eq:mixedStrategyUtility}
            \end{gather}
        \end{itemize}
        The \emph{support} of a mixed strategy $\delta_k \in \Delta_k$  is the set $\supp \delta_k \coloneqq \set{s \in S_k \mid \delta_k(s) > 0}$ of pure strategies it mixes between with positive probability.
    \end{defn}

    \begin{rem}~
        \label{rem:mixedExtensionsRemark}
        \begin{enumerate}
            \item 
            We denote the mixed strategies as functions $\delta: S_k \to \Rp$, assigning to each strategy $s_k \in S_k$ its probability $\delta(s_k)$ of being played.
            An alternative point of view is to interpret $\Delta_k$ as a subset of $\R^{\abs{S_k}}$, where each mixed strategy is a probability vector.
            In this view, $\Delta_k$ is the standard $(\abs{S_k}-1)$-simplex. 
            There is no difference between the two variants except for notation, and we will switch to the variant using probability vectors wherever it is more useful.
            
            \item 
            The map $\hat{u}_k$ as defined in $\eqref{eq:mixedStrategyUtility}$ is linear in its coordinates (more formally, a restriction of a linear map on the convex set of valid probability vectors): Let $\delta = (\delta_1, \dots, \delta_n) \in S$, $k \in [n]$, and $\delta_k \in \Delta_k$ be a convex combination of the form $\delta_{k} = \alpha_1 \delta_{k,1} + \alpha_2 \delta_{k, 2}$ with strategies $\delta_{k,1}, \delta_{k, 2} \in \Delta_k$. Then
            \begin{align*}
                \hat{u}_k(\delta_k, \delta_{-k}) ~&= 
                \sum_{(s_1, \dots, s_n) \in S} \biggpars{\prod_{i=1, i \neq k}^{n} \delta_i(s_i)} * (\alpha_1 \delta_{k,1}(s_k) + \alpha_2 \delta_{k, 2}(s_k)) * u_k( (s_1, \dots, s_n) ) \\
                &= \sum_{j=1, 2} \alpha_j \pars{ \sum_{(s_1, \dots, s_n) \in S} \biggpars{\prod_{i=1, i \neq k}^{n} \delta_i(s_i)} * \delta_{k,j}(s_k) * u_k( (s_1, \dots, s_n) )} \\
                &= \alpha_1 \hat{u}_k(\delta_{k,1}, \delta_{-k}) + \alpha_2 \hat{u}_k(\delta_{k,2}, \delta_{-k}).
            \end{align*}
        
            \item 
            \label{item:mixedExtensionsRemark-matrixBimatrixGamesRemark}
            When a game is specified by a table or matrix of payoffs, the usual interpretation from now on is that the matrix represents the corresponding mixed-extension game.
            To distinguish between properties of a finite game and its mixed extension, we say that the game has a certain property \emph{in pure strategies} or \emph{in mixed strategies}:
            For example, we could say that rock-paper-scissors has no Nash equilibrium in pure strategies, but it does have one in mixed strategies.
            Keep in mind that the mixed extension game $\hat{G}$ is still a game that fits Definition \ref{def:realValuedGames}:
            Where possible, we will state results for general games without making distinctions for pure-strategy and mixed-strategy games, and use the notations $S_k$ and $u_k$ instead of $\Delta_k$ and $\hat{u}_k$ (so by writing $S_k$ or $u_k$, we do \emph{not} automatically refer only to finite games).
        \end{enumerate}
    \end{rem}

    \begin{defn}
        The mixed extension $\hat{G}$ of a finite game $G$ is called a \emph{bimatrix game}.
        If $G$ is a zero-sum game, $\hat{G}$ is called a \emph{matrix game}.
    \end{defn}

    Allowing mixed strategies is crucial for the existence of Nash equilibria, as there are many games that do not have pure Nash equilibria. There are results that show that randomly chosen games have pure Nash equilibria with decreasing probability as the game size grows, which are as summarized in the following theorem:
    \begin{thm}[{\cite{bib:goldbergProbabilityOfEquilibria}, also cf. \cites[p.15]{bib:matsumotoGameTheory}[Exercise 1.2]{bib:nisanAlgorithmicGameTheoryCh1Basic}}]~
        \label{thm:probabilityOfPureNashEquilibria}
        \begin{enumerate}
            \item
            Consider finite two-player \emph{matrix} (i.e. zero-sum) games with $m$ and $n$ pure strategies for player 1 and 2, where all the $mn$ payoffs are picked iid from the same continuous probability distribution.
            The probability that such a game has a Nash equilibrium in pure strategies is $p_{m, n} = \frac{m! n!}{(m+n-1)!}$ which approaches zero for large $m, n$.
            
            \item 
            Consider finite \emph{bimatrix} games with $m$ and $n$ pure strategies for player 1 and 2, where all the $2mn$ payoffs are picked iid from the same continuous probability distribution.
            The probability that such a game has a Nash equilibrium in pure strategies is 
            $\hat{p}_{m, n} = 1 - \sum_{k=0}^{\min(m, n)} (-1)^k k! \binom{m}{k} \binom{n}{k} \pars{\frac{1}{mn}}^k$,
            which approaches $1 - 1/e \approx 0.632$ for large $m, n$.
        \end{enumerate}
    \end{thm}

    To give some example numbers: In the zero-sum case, $p_{2, 2} = \frac{2}{3}, p_{3, 3}=\frac{3}{10}$, and $p_{10, 10} \approx 0.01\%$. In the bimatrix case, $\hat{p}_{2} = \frac{7}{8}, \hat{p}_{3} \approx 78.6\%, \hat{p}_{10} \approx 67.2\%$.
    The theorem shows that in the two-player case, increasingly large random zero-sum games have pure-strategy Nash equilibria with a probability converging to zero. Random non-zero-sum games have pure-strategy Nash equilibria with a surprisingly high probability, but in the limit, still over one third of those games do not have pure Nash equilibria.
    
    On the other hand, bimatrix games always have at least one mixed-strategy Nash equilibrium. This is one of the core results of non-cooperative game theory, and was famously proved by John Nash in \cite{bib:nashOnePageProofOfEquilibria}.
    
    \begin{thm}[Existence of Mixed-Strategy Nash Equilibria, e.g. {\cite[Section 1.3.1]{bib:fudenbergGameTheory}}]
        Every mixed extension of a finite game has a Nash equilibrium.
        \label{thm:existenceOfMixedStrategyEquilibria}
    \end{thm}
    We will give a proof of this theorem in Section \ref{sec:existenceOfMixedStrategyEquilibria}, but first introduce some more concepts relevant for the proof.
    
    \begin{ex}
        The Rock-Paper-Scissors game from Example \ref{ex:rockPaperScissors}
        has the mixed-strategy Nash equilibrium $\bigpars{(\frac{1}{3}, \frac{1}{3}, \frac{1}{3}), (\frac{1}{3}, \frac{1}{3}, \frac{1}{3})}$.
        We will later see a way to prove this, and partially go through the proof, when looking at methods to compute Nash equilibria.
    \end{ex}
    
    A different way to characterize Nash equilibria, which has interesting consequences for mixed-strategy games, is by \emph{best responses}.
    
    \begin{defn}[Best responses, see \cite{bib:fudenbergGameTheory}]
        Let $G$ be a game. For each $k \in [n]$, define
        \begin{gather*} 
            r_k: S \to \Pot(S_k),~ r_k(s) = \biggset{s_k \in S_k \bigmid u_k(s_k, s_{-k}) = \max_{\tilde{s}_k \in S_k} u_k(\tilde{s}_k, s_{-k}) }.
        \end{gather*} 
        We call a strategy $s_k \in r_k(s)$ a \emph{best response} to the strategy profile $s \in S$ (or alternatively to $s_{-k}$).
        Since $r_k(s)$ depends only on $s_{-k}$, we also write $r_k(s_{-k})$ instead of $r_k(s)$ where more convenient \cite{bib:fudenbergGameTheory}.
        \footnote{Nevertheless, $r_k$ takes arguments from $S$ as this will notationally simplify the proof of Theorem \ref{thm:existenceOfMixedStrategyEquilibria}.}
    \end{defn}
    Using this definition, a Nash equilibrium can be characterized as a strategy profile in which the strategy for each player is a best response to the profile.
    \begin{thm}
        \label{thm:nashEquilibriumCharacterizationByBestResponses}
        Let $G$ be a game. A strategy profile $(s_1, \dots, s_n) \in S$ is a Nash equilibrium if and only if 
        \begin{gather*}
            \forall k \in [n]: s_k \in r_k(s).
        \end{gather*}
    \end{thm}
    \begin{proof}
        If $s_k \in r_k(s)$ for all players $k \in [n]$, then $\forall k \in [n], \forall \tilde{s}_k \in S_k: u_k(s_k, s_{-k}) \geq u_k(\tilde{s}_k, s_{-k})$, making $s$ a Nash equilibrium.
        Otherwise if for some $k$, $s_k \notin r_k(s)$, there exists some strategy $\tilde{s}_k$
        such that $u_k(\tilde{s}_k, s_{-k}) \geq u_k(s_k, s_{-k})$,
        so $s$ is not a Nash equilibrium.
    \end{proof}

    An important fact is that best-response mixed strategies always mix between best-response pure strategies:

    \begin{thm}[see {\cite[Theorem 2.1]{bib:nisanAlgorithmicGameTheoryCh2ComplexityNash}}]
        \label{thm:bestResponseMixing}
        Let $G$ be the mixed extension of a finite game.
        Let $s_k \in \Delta_k$ be a mixed strategy with $\supp s_k = \set{s_{k,1}, \dots, s_{k,m}}$, i.e.
        $s_k$ is a convex combination $s_k = \sum_{i=1}^m \alpha_i s_{k, i}, ~ \sum_{i=1}^m\alpha_i = 1$.
        Then for any strategy profile $s \in S$:
        \begin{gather*}
            s_k \in r_k(s) \lra \forall j \in [m]: s_{k, j} \in r_k(s).
        \end{gather*}
    \end{thm}
    \begin{proof}
        \sloppypar{
        Because $u_k(\dummydot, s_{-k})$ is linear (see Remark \ref{rem:mixedExtensionsRemark}), we get $u_k(s_k, s_{-k}) = \sum_{i=1}^m \alpha_i u_k(s_{k, i}, s_{-k})$.
        Assume that one $s_{k, j}$ has a smaller payoff than $s_k$: Then some other $s_{k, i}$ must have a greater payoff than $s_k$, as else $\sum_{i=1}^n \alpha_i u_k(s_{k, i}, s_{-k}) < u_k(s_k, s_{-k})$ because the sum is a weighted average.
        Because there is a strategy with greater payoff, $s_k$ is not a best response, leading to a contradiction.
        For the converse, assume that all $ s_{k, j} $ are best responses, i.e. have equal payoffs. By the linearity, $s_k$ has the same payoff, making it a best response as well.
    }
    \end{proof}

    \begin{cor}
        \label{cor:equilibriumStrategiesSupportHaveEqualPayoffs}
        If $s = (s_1, \dots, s_n) \in S$ is a Nash equilibrium, and the $k$-th player's strategy $s_k$
        mixes between pure strategies $s_{k,1}, \dots, s_{k,m} \in \Delta_k$, then the payoff of all the $s_{k,j}$ with respect to $s_{-k}$ is equal;
        Furthermore, the payoff of any strategy mixing between them is the same as well:
        \begin{gather}
            \forall j: u_k(s_{k,j}, s_{-k}) = u_k(s_k, s_{-k}), \\
            \forall \tilde{s}_k \in \Delta_k: \supp \tilde{s}_k \subseteq \supp s_k \Rightarrow u_k(\tilde{s}_k, s_{-k}) = u_k(s_k, s_{-k}).
        \end{gather}
    \end{cor}
    \begin{proof}
        Since $s$ is a Nash equilibrium, the mixed strategy $s_k$ is a best response to $s$ by theorem \ref{thm:nashEquilibriumCharacterizationByBestResponses}.
        By theorem \ref{thm:bestResponseMixing}, all pure strategies $s_{k, j}$ are best responses as well.
        By the definition of best responses, they therefore must all have the same payoff.
        By the linearity of $u_k(\dummydot, s_{-k})$, a mixed strategy mixing between pure strategies with equal payoff has the same payoff.
    \end{proof}
    
    A possible interpretation of Theorem \ref{thm:bestResponseMixing} and Corollary \ref{cor:equilibriumStrategiesSupportHaveEqualPayoffs} is that a player, given some strategies $s_{-k}$ of other players, does not mix his own strategies in order to achieve a better payoff; the purpose of mixing strategies is rather to enable a Nash equilibrium, since only the right mixing leads to a situation where the other players have no incentive to deviate.
    We will need these results in later chapters when analyzing more general games with lexicographically-ordered outcomes, but most importantly, we will see how they can be applied to compute Nash equilibria later in this chapter.
    
    \section{Proof of Existence of a Mixed-Strategy Nash Equilibrium}
    \label{sec:existenceOfMixedStrategyEquilibria}
    The existence of a mixed-strategy Nash equilibrium in every bimatrix game (Theorem \ref{thm:existenceOfMixedStrategyEquilibria}) was first proved by John Nash in a one-page article \cite{bib:nashOnePageProofOfEquilibria}.
    His proof is based on the characterization of Nash equilibria by best responses (Theorem \ref{thm:nashEquilibriumCharacterizationByBestResponses}) and non-constructively finds an equilibrium point by Kakutani's Fixed Point Theorem, a generalization of Brouwer's Fixed Point Theorem to set-valued functions.
    
    \begin{defn}[e.g. {\cite[p.30]{bib:fudenbergGameTheory}}]
        Let $S \subseteq \R^n$. A set-valued function $\phi: S \to \Pot(S)$ \emph{has a closed graph} if for all convergent sequences $(x_n)_{n \in \N}$, $(y_n)_{n \in \N}$ in $S$,
        \begin{gather*} 
            (\forall n: y_n \in \phi(x_n)) \implies \liminfty{n} y_n \in \phi\pars{\liminfty{n} x_n}.
        \end{gather*}
    \end{defn}

%    Having a closed graph clearly is a kind of continuity property for set-valued functions.
    Kakutani's original paper does not use the closed graph property, but instead uses the concept of \emph{upper semi-continuity}, which is equivalent in the case we are looking at. This is because in the following theorem $S$ is compact, and $\phi: S \to \Pot(S)$ takes only closed (and therefore compact) sets as values (see \cites{bib:kakutaniFixedPointTheorem}[Proposition 11.9, (a)-(b)]{bib:borderFixedPointTheorems}).

    \begin{thm}[Kakutani's Fixed Point Theorem, see \cite{bib:kakutaniFixedPointTheorem}, {\cite[p.29f]{bib:fudenbergGameTheory}}]
        Let $S \subseteq \R^n$. Let $\phi: S \to \Pot(S)$ be a function with the following properties:
        \begin{enumerate}
            \item $S$ is non-empty, compact and convex. \label{item:KakFP-nonEmptyCompactConvexDomain}
            \item $\forall s \in S: \phi(s)$ is  non-empty, convex and closed. \label{item:KakFP-nonEmptyConvexClosedOutputs}
            \item $\phi$ has a closed graph. \label{item:KakFP-closedGraph}
        \end{enumerate}
        Then $\phi$ has a fixed point $x$, i.e. an $x \in S$ such that $x \in \phi(x)$.
    \end{thm}

    We are not giving a proof for Kakutani's Theorem, but equipped with it, we can prove Theorem \ref{thm:existenceOfMixedStrategyEquilibria} that every bimatrix game has a Nash equilibrium in mixed strategies.
    
    \begin{proof}[Proof of Theorem \ref{thm:existenceOfMixedStrategyEquilibria}, cf. {\cite[p.29]{bib:fudenbergGameTheory}}]
        In the context of this proof, let each player $k \in [n]$ have $m_k$ different pure strategies and set $m \coloneqq m_1 + \dots + m_n$.
        We represent mixed strategy profiles as points in $\R^m$: Let $\Delta_k \coloneqq \set{(p_1, \dots, p_{m_k}) \in \R^{m_k} \mid \sum_{i=1}^{m_k} p_i = 1}$ (the standard $(m_k - 1)$-simplex) and let $\Delta \coloneqq \bigtimes_{k=1}^n \Delta_k \subseteq \R^m$. In other words, here we do not view mixed strategies as functions $\delta: S_k \to \Rp$, but instead interpret them as real vectors.
        
        \sloppypar{Next we define $r: \Delta \to \Pot(\Delta)$ as the \emph{best-response correspondence}. 
        Recall that ${r_k: \Delta \to \Delta_k}$ maps mixed-strategy profiles to the set of \emph{best-response mixed strategies} for the $k$-th player. Now $r$ incorporates this information for all players at once, mapping mixed-strategy profiles to the set of \emph{best-response mixed-strategy profiles}:}
        \begin{gather}
            r: \Delta \to \Pot(\Delta), (\delta_1, \dots, \delta_n) \mapsto r_1(\delta_1, \dots, \delta_n) \times \dots \times r_n(\delta_1, \dots, \delta_n).
            \label{eq:bestResponseCorresponenceDefinition}
        \end{gather}
        By Theorem \ref{thm:nashEquilibriumCharacterizationByBestResponses}, mixed-strategy Nash equilibria are exactly the fixed points of $r$, i.e. strategy profiles that are best responses to themselves: So if we show that the conditions of Kakutani's Fixed Point Theorem are satisfied, this shows that $r$ has a fixed point, and we have proved that mixed-strategy Nash equilibria always exist.
        
        On \ref{item:KakFP-nonEmptyCompactConvexDomain} (conditions on $\Delta$): The $\Delta_k$ are clearly non-empty, compact and convex as simplices. Therefore $\Delta$ is also non-empty, as well as
        compact and convex as the finite Cartesian product of compact and convex sets.
        
        On \ref{item:KakFP-nonEmptyConvexClosedOutputs} (the set $r(s)$ of best-response profiles is non-empty, convex and closed for all $s \in \Delta$): It suffices to show that these conditions hold for each $r_k(s)$, since $r(s)$ is the finite product of those. Let $k \in [n]$. By Theorem \ref{thm:bestResponseMixing}, best-response mixed strategies are exactly the convex combinations of best-response pure strategies. Therefore $r_k(s)$ is the convex hull of the $k$-th player's pure best response strategies to $s_{-k}$. As a convex hull of a finite set, $r_k(s)$ is convex and closed.
%        , in fact a simplex, and therefore convex and closed. 
        Furthermore, $r_k(s)$ is non-empty: Mixed strategies cannot have larger payoffs than the best pure strategy in their support. Since there are only finitely many pure strategies in the support, at least one of them maximizes $u(\dummydot, s_{-k})$.
        
        On \ref{item:KakFP-closedGraph} ($r$ has a closed graph): We need to show that if a sequence of mixed-strategy profiles converges, and a corresponding sequence of best-response profiles also converges, then the limit of the best responses is a best response to the limit of the strategy profiles. Let $\pars{s^{(i)}}_{i \in \N}$ be a convergent sequence of strategy profiles, and $\pars{t^{(i)}}_{i \in \N}$ a convergent sequence of best responses to the $s^{(i)}$:
        \begin{align*}
            &s^{(i)} = \pars{s_1^{(i)}, \dots, s_n^{(i)}} \toinfty{i} (s_1, \dots, s_n) =: s, \\
            &t^{(i)} = \pars{t_1^{(i)}, \dots, t_n^{(i)}} \toinfty{i} (t_1, \dots, t_n) =: t, \quad
            \forall i: t^{(i)} \in r\pars{s^{(i)}}.
        \end{align*}
        If we show that $t_k \in r_k(s)$ for all players $k$, we get $t \in r(s)$ by \eqref{eq:bestResponseCorresponenceDefinition}.
        Let 
%        $k \in \set{1, \dots, n}$, and let 
        $\tilde{t}_k \in \Delta_k$ be some arbitrary response. We show that $\tilde{t}_k$ is not a better response than $t_k$:
        First, for any $i \in \N$, $t^{(i)}_k$ is a best response to $s^{(i)}_{-k}$, so
        \begin{gather*}
            u_k\pars{t^{(i)}_k, s^{(i)}_{-k}} \geq u_k\pars{\tilde{t}_k, s^{(i)}_{-k}}.
        \end{gather*}
        The payoff function $u_k$ is continuous, since it is a restriction of a linear function between finite-dimensional spaces. Therefore in the limit as $i \to \infty$, the left side converges to $u_k(t_k, s_{-k})$ while the right side converges to $u_k(\tilde{t}_k, s_{-k})$.
        Recall that the ordering $\geq$ on the reals is preserved under limits (in other words, it is closed as a subset of $\R \times \R$).
        \footnote{A consequence of the “sandwich theorem”. We emphasize this here since this continuity property does not hold for general orderings on topological spaces:
        in particular, it will be important later that it does \emph{not hold} for the lexicographic ordering on $\R^n$.}
        Therefore,
        \begin{gather*}
            u_k(t_k, s_{-k}) \geq u_k(\tilde{t}_k, s_{-k}).
        \end{gather*}
        So $t_k$ maximizes the payoff over all responses to $s_{-k}$, therefore $t_k \in r(s_{-k})$. This shows that $r$ has a closed graph.
        Thus $r$ satisfies the conditions of Kakutani's theorem and therefore has a fixed point, which is a Nash equilibrium for the game by Theorem \ref{thm:nashEquilibriumCharacterizationByBestResponses}.
    \end{proof}

    \section{Computation of Nash Equilibria}
    It is important to have a way to compute Nash equilibria, especially for practical purposes, but also for theoretical justification of Nash equilibria as a prediction of rational behavior: As \cite[p.30]{bib:nisanAlgorithmicGameTheoryCh2ComplexityNash} cites Kamal Jain, “If your laptop cannot
find it, neither can the market.” There are various exact and numerical algorithms to compute Nash equilibria, and for the two-player case we will look at the simple exact method called the \emph{support enumeration algorithm} which can be executed by hand, and also the numerical \emph{fictitious play algorithm} which approximates a Nash equilibrium by simulating several rounds of play and refining the strategies over time based on the past actions.
    
    \subsection{Exact Computation}
    \label{subsec:exactComputationNashEquilibriaSupportEnumerationAlgorithm}
    Mixed Nash equilibria in bimatrix games can be computed exactly by a simple method called the \emph{support enumeration algorithm}.
    Our presentation follows \cite{bib:nisanAlgorithmicGameTheoryCh3EquilibriumComputation}.
    We assume the game's payoffs are specified by matrices $A, B \in \R^{n\times m}$, and the players have pure strategies $S_1 = \set{s_1,\dots,s_n}$, $S_2 = \set{t_1,\dots,t_m}$, respectively.
    Observe that \eqref{eq:mixedStrategyUtility} simplifies in the following way: If $x = (x_1,\dots,x_n) \in \Rp^n$ represents a mixed strategy of player 1 and $y = (y_1,\dots,y_m) \in \Rp^m$ one of player 2, then the players' payoffs are given by $u_1(x,y) = x\transposed A y$ and $u_2(x,y) = x\transposed B y$.
    
    If $(x, y)$ is a Nash equilibrium, then the first player's pure strategies in $\supp(x)$ must all be best responses to $y$, i.e. they all have equal payoff, and no other pure strategy can have greater payoff under $y$. This is a consequence of Theorem \ref{thm:bestResponseMixing}, and in \cite[p.55]{bib:nisanAlgorithmicGameTheoryCh3EquilibriumComputation} is stated in the following form:
    \begin{gather}
        \forall i: x_i > 0 \Rightarrow (Ay)_i = \max_{\tilde{i}} (Ay)_{\tilde{i}}.
        \label{eq:equalMaxPureStrategyPayoffsLinearSystem}
    \end{gather}
    The same holds with the player roles reversed.
    
    Suppose we want to find a Nash equilibrium where player one mixes between rows with indices in $I \subseteq [n]$, and player 2 mixes between columns with indices in $J \subseteq [m]$. Then player 1 must mix in a way that makes player 2 indifferent between the columns in $J$, and player 2 must mix in a way that makes player 1 indifferent between the rows in $I$.
    This is expressed in two linear systems of equations, where $i_1 \coloneqq \min(I), j_1 \coloneqq \min(J)$ are the smallest indices in $I$ and $J$, respectively:
    \begin{gather}
        (Ay)_{i_k} = (Ay)_{i_1} ~~~(\forall i_k \in I \setminus\set{i_1}), \qquad\;\, \sum_{j \in J} y_j = 1 \label{eq:linearSystemNecessaryForNashEquilibrium} \\
        (xB)_{j_k} = (xB)_{j_1} ~~~(\forall j_k \in J \setminus\set{j_1}), \qquad \sum_{i \in I} x_i = 1
    \end{gather}
    If some strategy profile $(x, y)$ solves this system of equations, it is a candidate for a mixed-strategy Nash equilibrium.
    It remains to check that all probabilities in the solution are non-negative. Also we have to make sure the mixed strategies actually have maximal payoffs as demanded by \eqref{eq:equalMaxPureStrategyPayoffsLinearSystem}: For each $\tilde{i}\in I^\complement$ we have to check that $(Ay)_{\tilde{i}} \leq (Ay)_{i_1}$, and analogously for $J$.    
    
    The algorithm we discuss needs one additional assumption, that the games involved are \emph{non-degenerate}:
    \begin{defn}[{\cite[Definition 3.2]{bib:nisanAlgorithmicGameTheoryCh3EquilibriumComputation}}]
        A mixed extension $G$ of a two-player finite game is \emph{degenerate} if one player has a mixed strategy of support size $m$ that has more than $m$ pure best responses by the other player. Otherwise, it is \emph{non-degenerate}.
        \label{def:degenerateRealValuedGame}
    \end{defn}
    In \cite[p.54]{bib:nisanAlgorithmicGameTheoryCh3EquilibriumComputation}, it is noted that “almost all” (two-player mixed-extension) games are non-degenerate.
    We can now put together the \emph{support enumeration algorithm} \cite[Algorithm 3.4]{bib:nisanAlgorithmicGameTheoryCh3EquilibriumComputation}: We iterate over all possible pairs of support indices $(I, J)$ with $\abs{I} = \abs{J}$, for each one go through the steps outlined above, and output all solutions of the linear system that satisfy the two additional properties (no negative probabilities and only best-response pure strategies in the support). 
    The assumption that the game is non-degenerate assures that the linear systems that occur in the algorithm do not have more than one solution: If we included degenerate games, we could not simply output all solutions, since there could be infinitely many.
    There are ways to deal with degenerate games as well which we will not go into here (see \cite[p.65]{bib:nisanAlgorithmicGameTheoryCh3EquilibriumComputation}, where an algorithm is discussed in detail).
    
    \begin{ex}[Computation of Rock-Paper-Scissors Equilibrium]
        We can compute the rock-paper-scissors Nash equilibrium from Example \ref{ex:rockPaperScissors} using the support enumeration algorithm.
        We will not go through the whole computation, but only look at the cases $(I, J) = (\set{1, 2}, \set{1, 2})$, $(I, J) = (\set{2, 3}, \set{1, 2})$ (which are not Nash equilibria) and $(I, J) = (\set{1, 2, 3}, \set{1, 2, 3})$ (which is a Nash equilibrium). Remember that the payoff matrices are given by $\pm \smallmat{0 & -1 & 1 \\ 1 & 0 & -1 \\ -1 & 1 & 0}$. We denote by $p_1, p_2, p_3$ the probabilities the strategy of player 1 assigns to the rows, and by $q_1, q_2, q_3$ the probabilities the strategy of player 2 assigns to the columns.
        
        For the first pair of indices, $p_1, p_2$ should be chosen in a way that makes player 2 indifferent between the first two columns. The corresponding equations are $0p_1 - 1p_2 = 1p_1 + 0p_2$ and $p_1 + p_2 = 1$. There is no solution, so this pair does not lead to a Nash equilibrium.
        
        For the second pair, the corresponding equations are $-p_2 + p_3 = -p_3, p_2 + p_3 = 1$ with the solution $p_2 = \frac{2}{3}, p_3 = \frac{1}{3}$. The payoff player 2 has for each of the first two columns in this case is $-\frac{1}{3}$. However the payoff player 2 has for the third column is $\frac{2}{3}$, so the first two columns are not best responses: This combination does also not lead to a Nash equilibrium.
        
        Finally for the third pair of indices, the equations are $-p_2 + p_3 = p_1 - p_3$, $-p_2 + p_3 = -p_1 + p_3$, and $p_1 + p_2 + p_3 = 1$, with the solution $p_1 = p_2 = p_3 = \frac{1}{3}$. Similarly the equations for player 2 to make player 1 indifferent between all rows are $-q_2 + q_3 = q_1 - q_3$, $-q_2 + q_3 = -q_1 + q_2$, and $q_1 + q_2 + q_3 = 1$, again with the solution $q_1 = q_2 = q_3 = \frac{1}{3}$. Since all solutions are positive, and no other rows/columns could be better responses, this shows that $\pars{(\frac{1}{3}, \frac{1}{3}, \frac{1}{3}), (\frac{1}{3}, \frac{1}{3}, \frac{1}{3})}$ is a Nash equilibrium of the Rock-Paper-Scissors game.
    \end{ex}
    
    \subsubsection{Complexity Considerations}
    The support enumeration algorithm is obviously quite inefficient, since it enumerates all subsets of the pure strategy sets for both players, and therefore is exponential in the number of strategies.
    \cite{bib:nisanAlgorithmicGameTheoryCh3EquilibriumComputation} investigates more sophisticated methods: The \emph{vertex enumeration algorithm} finds possible supports of mixed strategies by iterating over the vertices of certain polyhedra (“best-response polytypes”) and outputs all Nash equilibria.
    The \emph{Lemke-Howson algorithm} finds one Nash equilibrium by traversing a path in those polytypes.
    However, these improvements still have exponential worst-case complexities, and complexity-theoretic results suggest that finding Nash equilibria in general, even in two-player games, is an intractable problem. As discussed in detail in \cite{bib:nisanAlgorithmicGameTheoryCh2ComplexityNash}, the problem \textsc{Nash} of finding a Nash equilibrium is complete for the complexity class PPAD, which contains other problems for which an efficient algorithm is thought unlikely to exist, like finding fixed points in the context of Brouwer's fixed point theorem.
    Finding Nash equilibria does not fit into the more common intractability notion of NP completeness, because Nash equilibria are guaranteed to exist -- the problem can therefore not be stated suitably as a decision problem. However, there are a number of closely related variations where existence is not guaranteed and which are known to be NP-complete:
    For example, deciding whether a game has more than one Nash equilibrium, whether a Nash equilibrium with at least a given utility exists, and whether a Nash equilibrium exists with a given strategy in (or not in) its support, are all NP-complete problems (\cite{bib:gilboaEquilibriaComplexityConsiderations}, cited in \cite{bib:nisanAlgorithmicGameTheoryCh2ComplexityNash}).

    \subsection{Approximate Computation: Fictitious Play}
    A different approach for finding Nash equilibria is to simulate repeated play of the game by players that are learning from past outcomes.
    We follow the presentation of \cite{bib:daskalakisFictitiousPlayLectureNotes}.
    We again assume a bimatrix game with payoff matrices $A, B$.
    In every new round, each player averages over the past strategies of their opponent. Under the assumption that the opponent will play this average mixed strategy, the players picks their own pure strategy for this round that maximizes their payoff. Formally, denote by $s^{(k)} \in S_1, t^{(k)} \in S_2$ the strategies played in the $k$-th round, and $x^{(k)}, y^{(k)}$ the corresponding average strategies, defined iteratively by:
    \begin{gather*}
        x^{(k)} = \frac{1}{k} \sum_{i=1}^{k} s^{(i)}, \quad y^{(k)} = \frac{1}{k} \sum_{i=1}^{k} t^{(i)}, \\
        s^{(1)} \in S_1,\; t^{(1)} \in S_2 \text{ arbitrary}, \\
        s^{(k+1)} \in \mathop{\argmax}\limits_{\tilde{s} \in S_1} (u_1(\tilde{s}, y^{(k)})), \quad
        t^{(k+1)} \in \mathop{\argmax}\limits_{\tilde{t} \in S_2} (u_2(x^{(k)}, \tilde{t})).
    \end{gather*}
    Where there are multiple strategies that could be pick, we arbitrarily define the algorithm to always prefer the one with the lowest index.
    
    In the special case of zero-sum games, this process was shown to converge by \cite{bib:robinsonFictitiousPlay}; we will state the result without proof.
    In the non-zero-sum case, however, there are examples of bimatrix games for which fictitious play does not converge.
    
    \begin{thm}[Convergence of Fictitious Play, \cite{bib:robinsonFictitiousPlay}]
        \label{thm:fictitiousPlayConverges}
        If the game is a zero-sum game, i.e. $A=-B$, then 
        \begin{enumerate}
            \item The sequence $(u_1(x_k, y_k))_{k \in \N}$ converges towards the value of the game.
%             (player 1's unique payoff under a Nash equilibrium).
%             (i.e. $u_1(x, y)$ under a Nash equilibrium $(x, y)$).
            \item The sequence $((x_k, y_k))_{k \in \N}$ converges towards a Nash equilibrium $(x, y)$.
        \end{enumerate}
    \end{thm}

    Note that it is quite possible that for no value of $k$, $(x_k, y_k)$ actually forms a Nash equilibrium.
    To deal with situations like this, there is the notion of $\epsilon$-approximate Nash equilibria, where deviation from the equilibrium gains players at most $\epsilon$ additional payoff.
        
    \begin{defn}
        Let $\epsilon > 0$.
        A strategy profile $s = (s_1, \dots, s_n) \in S$ is an \emph{$\epsilon$-approximate Nash equilibrium} if
        \begin{gather*} 
            \forall k: \forall \tilde{s_k} \in S_k: u_k(s_k, s_{-k}) \geq u_k(\tilde{s_k}, s_{-k}) - \epsilon.
        \end{gather*} 
    \end{defn}

    The second result of \ref{thm:fictitiousPlayConverges} can now be stated as follows:
    \begin{cor}
        If the game is zero-sum, for any $\epsilon > 0$, there is some $K \in \N$ such that for all $k \geq K$, the result after $k$ rounds of fictitious play $(x_k, y_k)$ is an $\epsilon$-Nash equilibrium.
    \end{cor}

    \chapter{Games with Distributional Payoffs}
    \label{chap:gamesWithDistributionalPayoffs}
    The previous chapter established the basic concepts of standard game theory, where payoffs are real numbers.
    The goal of this chapter is to introduce and analyze a theory of games that instead have probability distributions as payoffs.
    To specify the players' preferences for outcomes in this setting, we use stochastic orders which compare probability distributions.
%    Now that we covered the basic concepts of standard game theory, we are ready to introduce the new theory of games that have probability distributions as payoffs.
    The model of distribution-valued games was first introduced by Stefan Rass
    in the context of IT security (\cite{bib:rassGameRiskManagI,bib:rassGameRiskManagII,bib:rassGameRiskManagIII}), and in this model probabilistic outcomes are rated based on a stochastic order we call the \emph{tail order} (\cite{bib:rassTotalOrderingOnLossDistributions}).
    We start with a generalization of this model and first introduce games with an arbitrary payoff set, where preferences are expressed by preorders on this set, and define distribution-valued games as a special case of these.
    We then focus on the tail order, discuss its properties as an ordering, and analyze the existence of Nash equilibria in distribution-valued games with tail order preferences.
    
    \section{Normal-Form Games with Generalized Payoffs}
    \label{sec:generalizedPayoffGames}    
    As a general framework, we define games where the payoffs lie in an arbitrary set. The players' preferences between payoffs are expressed by preorders on the payoff set.
    \begin{defn}[Preorder]
        A \emph{preorder} $\leq$ on some set $A$ is a reflexive and transitive binary relation on $A$:
        \begin{align*}
            &\forall a \in A: a \leq a. \tag{Reflexivity} \\
            &\forall a, b, c \in A: a \leq b, b \leq c \implies a \leq c. \tag{Transitivity}
        \end{align*}
        Additional properties of orders that a preorder {does \emph{not} need to satisfy} are:
        \begin{align*}
            &\forall a, b \in A: a \leq b, a \geq b \implies a = b. \tag{Antisymmetry} \\
            &\forall a, b \in A: a \leq b \vee a \geq b. \tag{Totality}
        \end{align*}
    \end{defn}
    We define $a \geq b \coloneqq b \leq a$ and $a < b \colonlra a \leq b \wedge \neg( a \geq b)$. If both $a \leq b$ and $a \geq b$, we say that $\leq$ is \emph{indifferent} between the two elements.
    Antisymmetric preorders are never indifferent between different elements. Total preorders order any pair of elements, so no two elements are incomparable. Since neither property is required, preorders in general may exhibit indifference as well as incomparability.
    
    % How we define preference at the moment...
    \let\popref\geq
    \let\pononpref\leq
    \let\postrpref>
    \let\postrnonpref<    

    \begin{defn}[Game with Generalized Payoffs]
%        Let $A$ be a set.
        A \emph{game with generalized payoffs in $A$}, $G = (n, A, (S_1, \dots, S_n), (u_1, \dots, u_n))$ with $n$ players consists of a payoff set $A$, strategy sets $S_k$ and a payoff function $u_k: S \to A$ for each player $k \in [n]$, where $S \coloneqq \bigtimes_{k \in [n]} S_k$.        
        Such a game can be \emph{equipped with preorders} $(\leq_1, \dots, \leq_n)$, written $G_{(\leq_1, \dots, \leq_n)}$, where each $\leq_k$ is a preorder on $A$ and represents the preferences of player $k \in [n]$ for the payoffs in $A$.
        If all players have the same preference preorder $\leq$, we write $G_\leq$.
        $G$ is \emph{finite} if $S$ is a finite set.
    \end{defn}
    We adopt the convention that the greater payoff with respect to a preorder is preferred, i.e. if $x \leq y$, then $y$ is preferred. However as in real-valued games, it will sometimes be more convenient to talk about costs instead of utility (cf. Definition \ref{def:realValuedGames}). We cannot simply define $c_k = -u_k$ in the general setting of a preordered set $A$, however we can effectively treat the payoffs as costs by choosing the preference preorders accordingly.
    
    \begin{defn}[Nash Equilibrium of Game with Generalized Payoffs]
        Let $G$ be a game with generalized payoffs that is equipped with preorders $(\leq_1, \dots, \leq_n)$.
        A \emph{Nash equilibrium of} $G_{(\leq_1, \dots, \leq_n)}$ (or a \emph{Nash equilibrium of $G$ with respect to $(\leq_1, \dots, \leq_n)$})
        is a strategy profile $s = (s_1, \dots, s_n) \in S$ such that for each player $k \in [n]$:
        \begin{gather}
            \forall \tilde{s}_k \in S_k: u_k(s_k, s_{-k}) \popref_k u_k(\tilde{s}_k, s_{-k}).
            \label{eq:nashEquilibriumGeneralizedPayoffs}
        \end{gather}
    \end{defn}
    
    Similar as in the real-valued theory, we are especially interested in two-player zero-sum games.
    But we are in some trouble defining what zero-sum is supposed to mean: Obviously, in our general setting we have no notion of two payoffs $a_1, a_2 \in A$ summing to zero.
    We can base our definition on another crucial property of zero-sum games: Zero-sum games are \emph{antagonistic} -- when one player wins, the other loses (see \cite{bib:andersonAntagonisticGames}). This property \emph{can} be generalized to our model.
    
    \begin{defn}[Antagonistic and Zero-Sum Games]
        A two-player game with generalized payoffs $G_{(\leq_1, \leq_2)}$ is \emph{antagonistic} if
        \begin{gather*}
            \forall s, t \in S:~ u_1(s) \popref_1 u_1(t) ~\lra~ u_2(s) \pononpref_2 u_2(t).
        \end{gather*}
        $G_{(\leq_1, \leq_2)}$ is \emph{zero-sum} if $u_1 = u_2$, and $\leq_2 {=} \geq_1$: The players always get equal payoffs, but prefer them just in reverse order.
        \label{def:zeroSumGeneralizedPayoffs}
    \end{defn}
    
    The advantage of zero-sum games over antagonistic games is that the zero-sum property is preserved when taking mixed extensions (which we will define shortly), while the mixed extension of an antagonistic game needs not be antagonistic. 
    This is analogous to real-valued games, where mixed extensions also preserve the zero-sum property, but not necessarily the antagonism property (see \cite{bib:andersonAntagonisticGames}).
    \footnote{Note that not only zero-sum satisfy this property, but constant-sum games as well, so our naming convention seems a bit arbitrary.
        Yet “zero-sum” appears like the most recognizable and easily-understood term for this concept, so we'll stick to this name even though nothing actually sums to zero.}
    
    It is possible to define mixed extensions if the payoff set $A$ has additional vector-space structure. Since mixed strategies correspond to convex combinations of pure strategies, we require $A$ be a convex subset of a real vector space.
    \begin{defn}[Mixed extension]
        Let $A$ be a convex subset of a real vector space.
        Let $G = (n, A, (S_1, \dots, S_n), (u_1, \dots, u_n))$ be a finite game with generalized payoffs in $A$.
        Its \emph{mixed extension} $\hat{G} = (n, A, (\Delta_1, \dots, \Delta_n), (\hat{u}_1, \dots, \hat{u}_n))$ consists of the following components for each player $k \in [n]$:
        \begin{itemize}
            \item $\Delta_k \coloneqq \biggset{ \delta \in \Rp^{S_k} \biggmid \sum_{s_k \in S_k} \delta(s_k) = 1 } $ are the \emph{mixed strategies}, $\Delta \coloneqq \bigtimes\limits_{i \in [n]} \Delta_i$ the \emph{mixed strategy profiles}.
            
            \item
            The utility function $\hat{u}_k: \Delta \to A$ maps each mixed strategy profile to
            its payoff:
            \begin{gather}
                \hat{u}_k: 
                (\delta_1, \dots, \delta_n)
                \mapsto
                \sum_{(s_1, \dots, s_n) \in S} \biggpars{\prod_{i=1}^{n} \delta_i(s_i)} * u_k( (s_1, \dots, s_n) ).
                \label{eq:generalizedPayoffsMixedStrategyUtility}
            \end{gather}
        \end{itemize}
    \end{defn}

    \begin{rem}
        As for real-valued games, we can specify payoffs by one or more matrices. We use terms analogous to the real-valued theory, see Remark \ref{rem:mixedExtensionsRemark}, \ref{item:mixedExtensionsRemark-matrixBimatrixGamesRemark}:
        When talking about the \emph{bimatrix game} specified by two matrices, we mean the mixed extension of the finite game corresponding to the matrices.
        In the zero-sum case, where only one matrix is specified, we will talk about \emph{matrix games}.
    \end{rem}
    
    As a last concept, we define the notion of \emph{isomorphic} games, which lets us switch between different payoff sets that behave the same with respect to compatible orderings.
    \begin{defn}
        Let $G$, $H$ be two $n$-player games with generalized payoffs sharing the same strategy sets $S_k$, where $G$ has payoffs $u_k: S \to A$ and $H$ has payoffs $v_k: S \to B$.
        Let $G$ be equipped with $(\leq_1^{\scriptscriptstyle G}, \dots, \leq_n^{\scriptscriptstyle G})$, and $H$ be equipped with $(\leq_1^{\scriptscriptstyle H}, \dots, \leq_n^{\scriptscriptstyle H})$.
        Then $G_{(\leq_1^{\scriptscriptstyle G}, \dots, \leq_n^{\scriptscriptstyle G})}$ is \emph{isomorphic} to $H_{(\leq_1^{\scriptscriptstyle H}, \dots, \leq_n^{\scriptscriptstyle H})}$
        if there is a bijection $\phi: A \to B$ such that
        \begin{align}
            &\forall k  \in [n]:~ v_k = \phi \circ u_k, \\
            &\forall k  \in [n]:~ \forall a_1, a_2 \in A: a_1 \leq_k^{\scriptscriptstyle G} a_2 \lra \phi(a_1) \leq_k^{\scriptscriptstyle H} \phi(a_2).
            \label{eq:isomorphicGeneralizedPayoffGames}
        \end{align}
    \end{defn}

    \begin{rem}
        \sloppypar{
        The isomorphism property is symmetric:  $G_{(\leq_1^{\scriptscriptstyle G}, \dots, \leq_n^{\scriptscriptstyle G})}$ is isomorphic to $H_{(\leq_1^{\scriptscriptstyle H}, \dots, \leq_n^{\scriptscriptstyle H})}$ iff  $H_{(\leq_1^{\scriptscriptstyle H}, \dots, \leq_n^{\scriptscriptstyle H})}$ is isomorphic to $G_{(\leq_1^{\scriptscriptstyle G}, \dots, \leq_n^{\scriptscriptstyle G})}$, as we can use the bijection $\phi\inv$.
        Therefore we can say that two games \emph{are isomorphic} without specifying a direction.}
    \end{rem}

    The reason we use the concept of isomorphic games is that it lets us change the payoff set of a game while preserving its Nash equilibria. This is shown by the next lemma.
    
    \begin{lemma}
        If two games with generalized payoffs $G_{(\leq_1^{\scriptscriptstyle G}, \dots, \leq_n^{\scriptscriptstyle G})}$, $H_{(\leq_1^{\scriptscriptstyle H}, \dots, \leq_n^{\scriptscriptstyle H})}$ are isomorphic,
        they have the same Nash equilibria.
    \end{lemma}
    \begin{proof}
        Let $s = (s_1, \dots, s_n) \in S$ be a Nash equilibrium of $G_{(\leq_1^{\scriptscriptstyle G}, \dots, \leq_n^{\scriptscriptstyle G})}$.
        We show that $s$ is a Nash equilibrium of $H_{(\leq_1^{\scriptscriptstyle H}, \dots, \leq_n^{\scriptscriptstyle H})}$. Let $k \in [n]$, $\tilde{s}_k \in S_k$.
        We have $v_k(\tilde{s}_k, s_{-k}) = \phi(u_k(\tilde{s}_k, s_{-k}))$. Because $s$ is a Nash equilibrium of $G_{(\leq_1^{\scriptscriptstyle G}, \dots, \leq_n^{\scriptscriptstyle G})}$, we get $u_k(\tilde{s}_k, s_{-k}) \leq_k^{\scriptscriptstyle G} u_k(s_k, s_{-k})$.
        By \eqref{eq:isomorphicGeneralizedPayoffGames}, this implies $\phi(u_k(\tilde{s}_k, s_{-k})) \leq_k^{\scriptscriptstyle H} \phi(u_k(s_k, s_{-k})) = v_k(s_k, s_{-k})$.
        So $v_k(\tilde{s}_k, s_{-k}) \leq_k^{\scriptscriptstyle H} v_k(s_k, s_{-k})$, proving the claim.
        By the symmetry pointed out in the previous remark, it also holds that any Nash equilibrium of $H_{(\leq_1^{\scriptscriptstyle H}, \dots, \leq_n^{\scriptscriptstyle H})}$ is a Nash equilibrium of $G_{(\leq_1^{\scriptscriptstyle G}, \dots, \leq_n^{\scriptscriptstyle G})}$, concluding the proof.
    \end{proof}
    
    \section{Distribution-Valued Normal-Form Games}
    \label{sec:distributionValuedNormalFormGames}
    \newcommand{\DP}{\mathcal{D}} % Shortcut for the set of distributions on the reals
    \newcommand{\Dgeqzero}{\DP_{\geq0}}
    \newcommand{\Dgeqone}{\DP_{\geq1}}
    \newcommand{\Dab}{\DP_{[a, b]}}
    
    \sloppypar{
    In this section we use the framework from the previous section to introduce distribution-valued games.
    This comes down to picking the right payoff set and choosing a preorder on it.
    For the payoff set, let $\DP$ denote the set of Borel probability measures on $\R$:}
    \begin{gather*}
        \DP \coloneqq \set{P: \B \to \Rp \mid P \text{ is a probability measure on } (\R, \B)}.
    \end{gather*}
    Each element of $\DP$ represents a probability distribution on the real numbers.
    We write $\Dgeqzero$, $\Dgeqone$, or $\Dab$ for the subsets of $\DP$ whose elements' supports are contained in $[0, \infty)$, $[1, \infty)$, or some interval $[a, b]$, respectively.
    Recall that for $k \in \N$, $\M^k \coloneqq \set{P \in \DP \mid \lintegral{\R}{\abs{x}^k}{P(x)} < \infty}$ denotes the set of measures in $\DP$ that have a finite moment of order $k$.
    
    % Defining shortcuts for the symbols used for generic stochastic orders. Can be changed easily here. “dleq” = “distributionally less-or-equal”
    \let\dleq\preccurlyeq
    \let\dgeq\succcurlyeq
    \let\dless\prec
    \let\dgreater\succ
    
    % Defining symbols for some particular orderings
    % Expected-value order
    \newcommand{\leqE}{\dleq_E}
    % Usual Stochastic Order
    \newcommand{\leqst}{\dleq_{\text{st}}}
    % Stochastic Tail Order
    \newcommand{\leqtail}{\dleq_{\text{tail}}}
    \newcommand{\geqtail}{\dgeq_{\text{tail}}}
    \newcommand{\lesstail}{\dless_{\text{tail}}}
    \newcommand{\greatertail}{\dgreater_{\text{tail}}}
    % Reflected Lexicographic Order
    \newcommand{\leqRlex}{\leq_{\textsc{Rlex}}}
    \newcommand{\lessRlex}{<_{\textsc{Rlex}}}
    \newcommand{\geqRlex}{\geq_{\textsc{Rlex}}}
    \newcommand{\greaterRlex}{>_{\textsc{Rlex}}}
    \newcommand{\leqLex}{\leq_{\textsc{Lex}}}
    % Tweakable Stochastic Order
    \newcommand{\leqtw}[1][]{\dleq_{\text{\tiny tw,$\scriptscriptstyle\ifstrempty{#1}{\mathcal{I}}{#1}$}}}
    \newcommand{\geqtw}[1][]{\dgeq_{\text{\tiny tw,$\scriptscriptstyle\ifstrempty{#1}{\mathcal{I}}{#1}$}}}
    \newcommand{\lesstw}[1][]{\dless_{\text{\tiny tw,$\scriptscriptstyle\ifstrempty{#1}{\mathcal{I}}{#1}$}}}
    \newcommand{\greatertw}[1][]{\dgreater_{\text{\tiny tw,$\scriptscriptstyle\ifstrempty{#1}{\mathcal{I}}{#1}$}}}
    % Pareto Order (weak pareto)
    \newcommand{\leqpareto}{\leq_{\text{par}}}
    \newcommand{\geqpareto}{\geq_{\text{par}}}
    \newcommand{\lesspareto}{<_{\text{par}}}
    \newcommand{\greaterpareto}{>_{\text{par}}}
    \newcommand{\indiffpareto}{\sim_{\text{par}}}
    % Dominance Order (strong pareto)
    \newcommand{\leqdom}{\leq_{\text{dom}}}
    \newcommand{\geqdom}{\geq_{\text{dom}}}
    \newcommand{\lessdom}{<_{\text{dom}}}
    \newcommand{\greaterdom}{>_{\text{dom}}}
    
    \begin{defn}[Stochastic Order]
%        We call a preorder on $\DP$ a \emph{stochastic order}.
%        A \emph{stochastic order} $\dleq$ is a preorder on $\DP$.
        A preorder on $\DP$ is called a \emph{stochastic order}.
        \footnote{The reference on stochastic orders \cite{bib:shakedStochasticOrders} defines stochastic orders between random variables, but for our purposes it is more convenient to define them between distributions instead.}
    \end{defn}
    
    \begin{ex}~
        \label{ex:stochasticOrdersLeqELeqst}
        \begin{enumerate}
            \item 
            Let $\leqE$ be the ordering that compares distributions by expected value:
            \begin{gather*}
                P_1 \leqE P_2 \colonlra \text{ $P_1, P_2 \in \M^1$ and $E(P_1) \leq E(P_2)$}.
            \end{gather*}
            Then $\leqE$ is a stochastic ordering, but neither antisymmetric nor total. Antisymmetry fails to hold since many distributions can have the same expected value, in which case $\leqE$ is indifferent between them. Totality fails because not all distributions have a well-defined expected value. However $\leqE$ is a total stochastic order on $\M^1$, which includes all distributions with bounded support.
            
            \item Let the \emph{usual stochastic order} $\leqst$ (see \cite{bib:shakedStochasticOrders}) be defined by 
            \begin{gather*}
                P_1 \leqst P_2 \colonlra \forall x \in \R: P_1((x, \infty)) \leq P_2((x, \infty)).
            \end{gather*}
            Phrased in terms of distribution functions, this holds iff $F_1 \geq F_2$ point-wise: For any $x \in X$, $F_1(x) \geq F_2(x) \lra P_1((-\infty, x]) \geq P_2((-\infty, x])$, which is equivalent to $P_1((x, \infty)) \leq P_2((x, \infty))$. The ordering captures a strong notion of one distribution tending to take smaller values than the other. Because of this strong requirement, it is not surprising that $\leqst$ is not total and there are many pairs of incomparable distributions. However the ordering is antisymmetric, because $\leqst$ is indifferent only between probability measures with equal distribution functions, and a distribution function uniquely determines the probability measure.
        \end{enumerate}
    \end{ex}

    % How we define preference at the moment...
    \let\dpref\dgeq
    \let\dnonpref\dleq
    \let\dstrpref\dgreater
    \let\dstrnonpref\dless
    
    % ...vs. how it should be defined if utility was switched to cost
    %    \let\dpref\dleq
    %    \let\dnonpref\dgeq
    %    \let\dstrpref\dless
    %    \let\dstrnonpref\dgreater

    \begin{defn}[Distribution-Valued Game]
        A game $G$ with generalized payoffs in $\DP$ is called a \emph{distribution-valued game}.
        \label{def:distributionValuedGame}
    \end{defn}
    The underlying semantics are that a distribution-valued game models a situation where the outcomes have some probabilistic uncertainty associated with them, whose distribution is known before-hand: When playing the game, first an outcome distribution for each player is determined by the selected strategies, and then the players gets their actual real-valued payoff drawn from their distributions independently.
    However, the drawing from the outcome distributions at the last step is not explicitly modeled. The stochastic orders of the players reflect how each player values particular outcome distributions.
    \autoref{fig:distributionValuedGame} illustrates the definition with an example of the first player's payoffs in a distribution-valued game.
    
    \begin{figure}
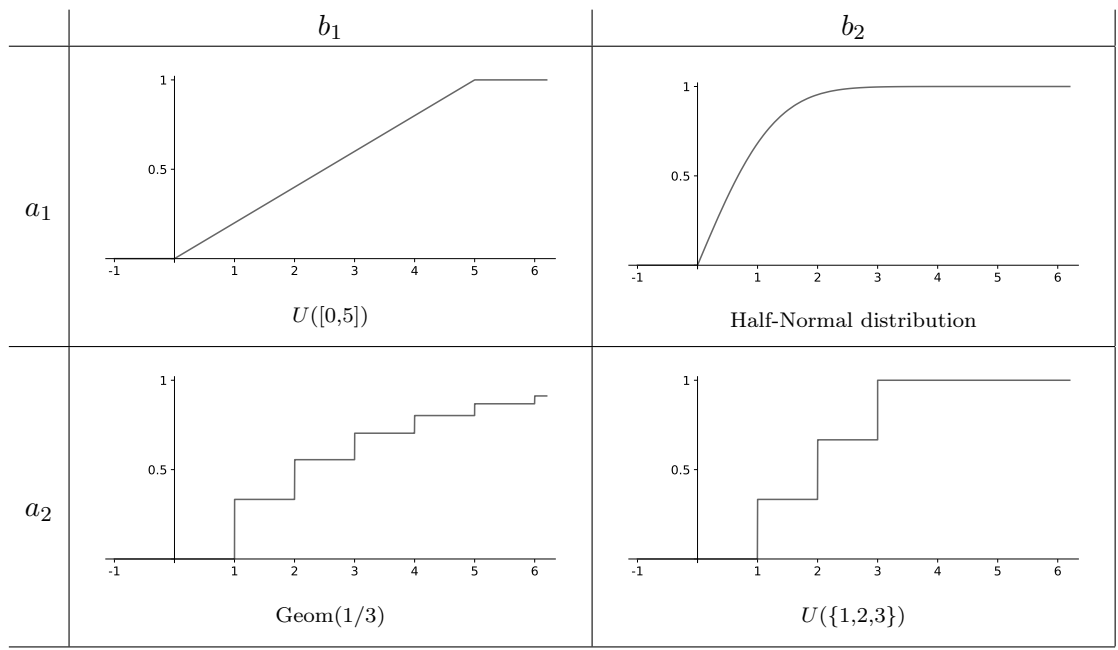

        \centering
        \begin{tabular}{c|c|c|}
            &        $b_1$        &        $b_2$        \\ \hline
            \raisebox{45pt}{$a_1$} & \shortstack{\adjustbox{margin=5pt 5pt 5pt 8pt}{\includegraphics[width=0.40\textwidth]{Pictures/example-distval-game-F1}}  \\ $\scriptstyle U([0, 5])$ \\~ } & \shortstack{\adjustbox{margin=5pt 5pt 5pt 8pt}{\includegraphics[width=0.40\textwidth]{Pictures/example-distval-game-F2}}  \\ $\scriptstyle \text{Half-Normal distribution}$ \\~ } \\\hline
            \raisebox{45pt}{$a_2$} & \shortstack{\adjustbox{margin=5pt 5pt 5pt 8pt}{\includegraphics[width=0.40\textwidth]{Pictures/example-distval-game-F3}}  \\ $\scriptstyle \text{Geom}(1/3)$ \\~ } & \shortstack{\adjustbox{margin=5pt 5pt 5pt 8pt}{\includegraphics[width=0.40\textwidth]{Pictures/example-distval-game-F4}}  \\ $\scriptstyle U(\set{1, 2, 3})$ \\~ } \\ \hline
        \end{tabular}
        \caption{Player 1 payoff matrix in a distribution-valued 2x2-game with both AC and discrete distributions as payoffs (represented by their distribution functions).}
        \label{fig:distributionValuedGame}
    \end{figure}

    The definitions of Nash equilibria, zero-sum games and mixed extensions of distribution-valued games can be adopted from Section \ref{sec:generalizedPayoffGames}.
    In particular, mixed extensions can be defined because $\DP$ is a convex subset of the real vector space of signed finite Borel measures: A convex combination of several distributions from $\DP$, obtained by a mixed strategy, corresponds to a mixture of the involved distributions.
    
    The definition of zero-sum games may need some additional discussion: The argument for the definition given in Section \ref{sec:generalizedPayoffGames} was that without knowing more about the structure of the payoff set, we could not define zero-sum games by requiring that payoffs actually sum to zero.
    On the other hand, now that we have a more concrete payoff set, possibly a more natural definition could be made.    
    However several seemingly “natural candidates” for such a definition do not work: Let $s \in S$ be some strategy profile.
    Clearly, we cannot require $u_1(s) + u_2(s) = 0$ (the zero-everywhere measure), since probability measures only assign non-negative values.
    If instead we let $X_1 \sim u_1(s), X_2 \sim u_2(s)$ be independent random variables modeling the random payoffs for the players, and require that $X_1 + X_2 = 0$ almost surely, this also fails: This condition will never hold except for trivial distributions, because the sum of independent random variables is distributed as the convolution of their distributions.
    Alternatively, we could require that the distributions $u_1(s), u_2(s)$ are symmetric mirror images around the origin, in the sense that for all Borel sets $B$, $u_1(s)(B) = u_2(s)(-B)$. For distributions with densities $f_1, f_2$, this would mean $f_1(x) = f_2(-x)$ for (almost) all $x$.
    While this requirement seems to make sense, this symmetry is not necessarily reflected by the stochastic orders: There is no requirement for stochastic orders to satisfy $P_1(\dummydot) \dleq P_2(\dummydot) \lra P_1(-\dummydot) \dgeq P_2(-\dummydot)$. So if we chose this definition, there could be zero-sum games without the antagonism property, where a certain strategy profile is preferred to another profile by both players.
    For the same reason, we cannot base our definition on requiring the expected values $E(u_1(s)), E(u_2(s))$ to sum to zero: It is possible that $E(u_1(s)) + E(u_2(s)) = E(u_1(t)) + E(u_2(t)) = 0$, yet $u_1(s) \dstrnonpref_1 u_1(t)$ and $u_2(s) \dstrnonpref_2 u_2(t)$.
    Since no alternative definition seems meaningful, we stick to the zero-sum definition made in Definition \ref{def:zeroSumGeneralizedPayoffs}.

    In the classical real-valued model discussed in Chapter \ref{chap:nonCooperativeRealValuedGameTheory}, mixed strategies are rated by their expected payoff. It is worth noting that this classical model emerges as a special case of the distribution-valued model if we choose the right stochastic orderings.
    \begin{ex}
        Let $\delta_x: A \mapsto \1_A(x)$ denote the Dirac probability measure.
        For a given finite real-valued game $G^{(\R)}$ with payoffs $u_k^{(\R)}$, define its distribution-valued finite counterpart $G^{(\DP)}$ to have payoffs
        \begin{gather*}
            \forall k \in [n]: ~u_k^{(\DP)}(s) \coloneqq \delta_{u_k^{(\R)}(s)}.
        \end{gather*}
        Then the real-valued mixed extension $\hat{G}^{(\R)}$ is isomorphic to the distribution-valued mixed extension $\hat{G}^{(\DP)}$ with respect to $\leqE$ (see Example \ref{ex:stochasticOrdersLeqELeqst}).
    \end{ex}
    Games constructed like $\hat{G}^{(\DP)}$ in the previous example might be interesting in their own right, with respect to other stochastic orders than $\leqE$:
    The model allows to talk about the “payoff distribution” of the real-valued game $G$, and rate outcomes by other orders than the expected value. 
    In decision-theoretic terms, this allows to model risk-averse and risk-seeking attitudes, as opposed to the risk-neutral attitude of expected value. We will however not further pursue this line of thought, as our main interest lies in mixed extensions of games where the pure-strategy payoffs already are non-trivial distributions.
    
    \section{The Stochastic Tail Order}
    \label{sec:stochasticTailOrder}
    We have laid out the foundations of distribution-valued games in the last section, but the usefulness of the model mostly depends on choosing the right stochastic orders.
    The stochastic orders we have seen so far are not satisfying in many cases: The expected value ordering $\leqE$ basically leads us back to the theory of real-valued games by taking expected values. The usual stochastic order $\leqst$ takes more information from the distributions into account, but fails to compare many distributions since a decision is only made for distributions where one distribution function dominates the other in a strong way.
    In an attempt to extend the idea of the usual stochastic order $\leqst$, Stefan Rass in a series of papers (\cite{bib:rassGameRiskManagI,bib:rassGameRiskManagII,bib:rassTotalOrderingOnLossDistributions}) defined an ordering which we call the \emph{stochastic tail order} and denote by $\leqtail$
    \footnote{The papers by Rass do not give a name to the ordering. We use the name \emph{tail order} because it compares distributions by their tails, which will become clear soon.}.
    It was introduced in the context of risk assessment for critical infrastructure, where risks with higher impact are to  be avoided at all costs, even if very unlikely.
    Therefore it is usually not applied to compare \emph{payoff distributions}, but rather to compare \emph{loss} (or \emph{cost}) \emph{distributions}. This detail does not affect our presentation of the theory, however, as we can simply consider or distribution-valued games to be ordered with respect $\geqtail$ instead of $\leqtail$ if the “payoffs” are supposed to represent losses.
    The tail order is defined based on moment sequences, but we will later show that in many cases, it can be understood as a kind of lexicographic ordering: The original goal for its introduction was to have an ordering that applies the criterion of $\leqst$, but only from some point $x_0$ on, so that more distributions are actually comparable (see \cite{bib:rassTotalOrderingOnLossDistributions}). This works to some extent, however in this section we will also see that some properties of the ordering claimed in \cite{bib:rassTotalOrderingOnLossDistributions} can fail in pathological cases.

    \subsection{Definition and Basic Properties}
    \label{subsec:tailOrderDefinitionAndBasicProperties}
    Recall that $m_k(P) \coloneqq \lintegral{R}{x^k}{P}$ denotes the $k$-th moment of $P \in \DP$ if $\lintegral{R}{\abs{x}^k}{P} < \infty$, and $M$ denotes the set of probability measures from $\DP$ that have moments of all orders.
    \begin{defn}[{Stochastic Tail Order, cf. \cite[Definition 2]{bib:rassTotalOrderingOnLossDistributions}}]
        \label{def:stochasticTailOrder}
        The \emph{stochastic tail order} $\leqtail$ is defined for $P_1, P_2 \in \DP$ by the following condition:
        \begin{multline*}
            P_1 \leqtail P_2 \colonlra \\
            (P_1 = P_2) \vee \bigpars{(P_1, P_2 \in \M) \wedge (m_k(P_1) \leq m_k(P_2) \text{ for all but finitely many $k \in \N_0$})}.
        \end{multline*}
    \end{defn}
    The definition differs slightly from \cite[Definition 2]{bib:rassTotalOrderingOnLossDistributions} in that we do not model payoffs by random variables, but instead by probability measures, and we relax the assumptions on orderings to be comparable:
    According to the original definition, only distributions in $\Dgeqone$ with bounded support that are either discrete, or absolutely continuous with continuous density function, should be comparable. To keep our discussion as general as possible, we avoid these rather strict assumptions and only demand that comparable distributions must have moments of all orders.
    However it must be noted, and will become clear in the theorems of this section, that $\leqtail$ in its original intention to generalize $\leqst$ only makes sense for elements of $\Dgeqone \cap M$.

    \begin{lemma}
        The tail order $\leqtail$ is a stochastic order, i.e. a preorder on $\DP$.
    \end{lemma}
    \begin{proof}
        The tail order is reflexive since for any $P_1 \in \DP$, $P_1 \leqtail P_1$ by definition.
        It is also transitive: If $P_1, P_2, P_3 \in \DP$ and $P_1 \leqtail P_2, P_2 \leqtail P_3$, then $P_1, P_2, P_3 \in M$.        
        Therefore $m_k(P_1) \leq m_k(P_2)$ for all $k \geq K_1$ and $m_k(P_2) \leq m_k(P_3)$ for all $k \geq K_2$, and $m_k(P_1) \leq m_k(P_3)$ for all $k \geq \max(K_1, K_2)$, so $P_1 \leqtail P_3$.
    \end{proof}
    
    The next question is whether $\leqtail$ is antisymmetric and/or total.
    In the general context, it has neither of the properties: The order is trivially not total on $\DP$ as it can only compare distributions from $M$.
    From Definition \ref{def:stochasticTailOrder} it is clear that two distributions from $M$ are incomparable iff their moment sequences alternate. Moment sequences of two distributions from $M$ can alternate because distributions can take negative values -- as an example, consider the Dirac distributions $\delta_{-1}$ and $\delta_{0}$: The moment sequence of the former is given by $(-1)^n$, alternating around the sequence of the latter which is constantly $0$, making the distributions not $\leqtail$-comparable.
    
    Antisymmetry does not hold on all of $M$ neither: There exist examples of different distributions in $\M$ which have equal moment sequences (an example is given in \cite[Example 3.15]{bib:romanoCounterexamplesInProbability}).
    On the other hand, if a probability measure $P_1 \in \DP$ has bounded support, then no other probability measure from $\DP$ has the same moment sequence (cf. \cite[Corollary 4.2]{bib:schmuedgenTheMomentProblem}).
    This even holds with the weaker condition that the moment-generating function of the probability measure exists in a neighborhood of 0 (cf. \cites[Lemma 2.3]{bib:rassGameRiskManagI}[p.414]{bib:billingsleyProbabilityAndMeasure})
     \footnote{The moment-generating function of a probability measure $P \in \DP$ is given by $M(s) \coloneqq \lintegral{\R}{e^{sx}}{P(x)}$ for $s \in \R$ if the integral is finite, cf. \cite[(21.21)]{bib:billingsleyProbabilityAndMeasure}}.
    However, these results do not guarantee the antisymmetry of $\leqtail$ on the respective subset of $\DP$, as $\leqtail$ could be indifferent between two distributions with different moment sequences if they disagree only finitely often. It is not clear if such a case can actually occur, so whether or not $\leqtail$ is antisymmetric between distributions with bounded support, or another sufficiently large subset of $\DP$, remains open.
    
    In a similar way, we can try to identify a sufficiently large class of distributions on which $\leqtail$ is total. We will discuss this question in more detail later in Section \ref{subsec:tailOrderTotality}.
    
    \subsection{Sufficient Conditions for Tail Order Preference}
    Before we do so, let us first get a better understanding of what the tail order means apart from moment sequences.
    In \cite{bib:rassGameRiskManagI}, the discussion is restricted to random variables taking values in $[1, \infty)$, since the behavior of moment sequences differs greatly for values between $[0, 1)$, or even negative values. Additionally, the support is assumed to be bounded. Also, the discussion focuses only on discrete distributions with finite support, or absolutely continuous distributions with additional continuity/differentiability constraints on the densities.
    We will not need such strict assumptions everywhere, and will always mention the exact requirements.
    
    An important intuition for understanding the tail order is that it behaves essentially like a lexicographic ordering on the density functions in many cases: Specifically, if the density function of one distribution overtakes the density function of another at some point, and dominates from there on, then the first distribution is greater with respect to $\leqtail$.
    That vague intuition will be formalized in the following theorems in the form of sufficient conditions, and an equivalence in the case of finite support.
    
    \begin{thm}
        Let $P_1, P_2 \in \Dgeqone \cap \M$. Assume there is some $x_0 \geq 1$ such that
        \begin{align}
            \forall B \in \B, &~B \subseteq [x_0, \infty): P_1(B) \leq P_2(B)  \label{eq:p2DominatesP1LeftOfX0},  \\
%            \exists B_0 \in \B: B_0 \subseteq (x_0, \infty), P_1(B_0) < P_2(B_0) 
            &~P_1([x_0, \infty)) < P_2([x_0, \infty)) \label{eq:p2StrictlyDominatesP1OnB0}.
        \end{align}
        Then $P_1 \lesstail P_2$.
        \label{thm:tailOrderSufficientConditionsGeneral}
    \end{thm}    
    \begin{proof}
        To show that $m_k(P_1) < m_k(P_2)$ for all $k$ large enough, we show that $m_k(P_2) - m_k(P_1) > 0$ for all $k$ large enough.
        The proof will proceed as follows: $m_k(P_2) - m_k(P_1)$ corresponds to an integral expression, and we decompose the integration domain $[1, \infty)$ into intervals $[1, x_1]$, $(x_1, x_0)$, and $[x_0, \infty)$. We then show that the integral over $[x_0, \infty)$ grows at least proportionally to $-x_0^k$ and the integral over $[1, x_1]$ shrinks at most proportionally to $-x_1^k$. We pick $x_1$ close enough to $x_0$ such that the integral over $(x_1, x_0)$, which is estimated to shrink at most proportionally to $x_0^k$, does not cancel out the one over $[x_0, \infty)$. Since $x_1 < x_0$, the positive growth of $x_0^k$ dominates the negative growth of $x_1^k$ eventually, showing that $m_k(P_2) - m_k(P_1) > 0$ for large enough $k$.
        
        We start by calculating the difference of moments as follows:
        \begin{align*}
            m_k(P_2) - m_k(P_1)
            &= \lintegral{\R}{x^k}{P_2(x)} - \lintegral{\R}{x^k}{P_1(x)} \\
            &= \biggpars{\lintegral{[1, x_0)}{x^k}{P_2(x)} - \lintegral{[1, x_0)}{x^k}{P_1(x)}}
            + \lintegral{[x_0, \infty)}{x^k}{\undersetbrace{(P_2 - P_1)}{\eqqcolon \mu}(x)}.
        \end{align*}
        We define $\mu \coloneqq \restr{(P_2 - P_1)}{\B \cap [x_0, \infty)}$ as the difference ($P_2 - P_1)$ restricted to the Borel sub-$\sigma$-algebra on $[x_0, \infty)$. By \eqref{eq:p2DominatesP1LeftOfX0}, $\mu$ is non-negative and therefore a measure, and by \eqref{eq:p2StrictlyDominatesP1OnB0}, it is not the zero measure.
        We can now estimate
        \begin{gather*}
            \lintegral{[x_0, \infty)}{x^k}{\mu(x)}
%            \geq \lintegral{[x_1, \infty)}{x^k}{\mu(x)}
            \geq \lintegral{[x_0, \infty)}{x_0^k}{\mu(x)}
            = x_0^k * \mu([x_0, \infty)).
        \end{gather*}
        Next we want to pick $x_1 < x_0$ close enough to $x_0$ such that $P_1((x_1, x_0)) < \mu([x_0, \infty))$.
        We can represent the interval $[1, x_0)$ as $[1, x_0) = \bigcup_{n \in \N} [1, x_0 - \frac{1}{n}]$.
        By continuity from below of the probability measure $P_1$, we get $\lim\limits_{n \to \infty} P_1([1, x_0 - \frac{1}{n}]) = P_1([1, x_0))$.
        Therefore we find an $x_1 < x_0$ such that $P_1([1, x_1]) > P_1([1, x_0)) - \mu([x_0, \infty))$, i.e. $P_1((x_1, x_0)) < \mu([x_0, \infty))$.
        With this we estimate
        \begin{align*}
            &\lintegral{[1, x_0)}{x^k}{P_2(x)} - \lintegral{[1, x_0)}{x^k}{P_1(x)}
%            \geq - \lintegral{[1, x_0]}{x^k}{P_1(x)}
            \geq - \lintegral{[1, x_0)}{x^k}{P_1(x)} \\
            = - &\lintegral{[1, x_1]}{x^k}{P_1(x)} - \lintegral{(x_1, x_0)}{x^k}{P_1(x)}
            \geq - \lintegral{[1, x_1]}{x_1^k}{P_1(x)} - \lintegral{(x_1, x_0)}{x_0^k}{P_1(x)} \\
            =   -  & x_1^k * P_1([1, x_1]) - x_0^k * P_1((x_1, x_0)).
        \end{align*}
        In summary, we get
        \begin{multline*}
            m_k(P_2) - m_k(P_1)
            \geq x_0^k * \mu([x_0, \infty)) - x_0^k * P_1((x_1, x_0)) - x_1^k * P_1([1, x_1]) \\
            = x_0^k \bigpars{\undersetbrace{\mu([x_0, \infty)) - P_1((x_1, x_0))}{> 0}} - x_1^k * P_1([1, x_1]),
        \end{multline*}
        which is greater than zero for all $k$ large enough, because $x_1 < x_0$.
    \end{proof}
    
    Special cases of this theorem hold for absolutely continuous and discrete distributions. Recall for the next proof that we write $\lambda$ for the Lebesgue measure.
    
    \begin{thm}
        Let $P_1, P_2 \in \Dgeqone \cap \M$ be absolutely continuous, with densities $f, g$.
        \begin{enumerate}
            \item 
            If there is some $x_0 \geq 1$ such that on the interval $[x_0, \infty)$ we have $f \leq g$ almost everywhere, yet not $f = g$ almost everywhere,
            then $P_1 \lesstail P_2$.
            \label{item:acTailOrder-NonNullDominatingSetSuffCondition}
            
            \item 
            If there are $x_0 \geq 1, \delta > 0$ such that $f \leq g$ on $[x_0, \infty)$ and $f < g$ on $[x_0, x_0 + \delta]$, then $P_1 \lesstail P_2$.
            \label{item:acTailOrder-DominatingIntervalSuffCondition}
        \end{enumerate}
        \label{thm:acTailOrderSuffConditions}
    \end{thm}
    \begin{proof}
        ~        
        \sloppypar{
        \ref{item:acTailOrder-NonNullDominatingSetSuffCondition}:
        Our assumptions imply that $\lambda(\set{f > g} \cap [x_0, \infty)) = 0$ and ${\lambda(\set{f < g} \cap [x_0, \infty))} > 0$.
        Let $B \in \B, B \subseteq [x_0, \infty)$. Then $f \leq g$ almost everywhere on $B$.
        So $P_1(B) = \lintegral{B}{f}{\lambda} \leq \lintegral{B}{g}{\lambda} = P_2(B)$.}
        On the other hand, 
        \begin{align*}
            P_1([x_0, \infty)) &=  \lintegral{[x_0, \infty) \cap \set{f < g}}{f}{\lambda} +  \lintegral{[x_0, \infty) \cap \set{f = g}}{f}{\lambda} \\
            &< \lintegral{[x_0, \infty) \cap \set{f < g}}{g}{\lambda} +  \lintegral{[x_0, \infty) \cap \set{f = g}}{g}{\lambda} = P_2(B).
        \end{align*}
        So the conditions for Theorem \ref{thm:tailOrderSufficientConditionsGeneral} are satisfied, and $P_1 \lesstail P_2$.
        
        \ref{item:acTailOrder-DominatingIntervalSuffCondition}:
        This is as a special case of the previous condition.
        We have that $\lambda(\set{f > g} \cap [x_0, \infty)) = \lambda(\emptyset) = 0$, and $\lambda(\set{f < g} \cap [x_0, \infty)) > \lambda([x_0, x_0 + \delta]) > 0$, so the conditions of the first part are satisfied and $P_1 \lesstail P_2$.
    \end{proof}

    \begin{thm}
        Let $P_1, P_2 \in \Dgeqone \cap \M$ be discrete, with probability mass functions $f, g$.
        \begin{enumerate}
            \item 
            If there is some $x_0 \geq 1$ such that $f(x_0) < g(x_0)$, and $\forall x > x_0: f(x) \leq g(x)$, then $P_1 \lesstail P_2$.
            \label{item:discreteTailOrder-DominatingPointSuffCondition}
            
            \item 
            If $P_1, P_2$ both have finite support, an equivalence holds:
            \begin{gather}
                P_1 \lesstail P_2 \lra \exists x_0 \geq 1: f(x_0) < g(x_0) \wedge \forall x > x_0: f(x_0) \leq g(x_0).
                \label{eq:discreteFiniteTailOrder-DominatingPointEquivalence}
            \end{gather}            
            \label{item:discreteFiniteTailOrder-DominatingPointEquivalence}
        \end{enumerate}        
    \label{thm:discreteTailOrderSuffConditions}
    \end{thm}
    \vspace{-1.2cm}
    \begin{proof}
        \ref{item:discreteTailOrder-DominatingPointSuffCondition}:
        We can view discrete distributions as absolutely continuous with respect to the counting measure $\#$, and e.g. write $P_1(B) = \lintegral{B}{f}{\#}$.
        Let $B \in \B, B \subseteq [x_0, \infty)$. Then analogously to the previous proof, we get
        $P_1(B) = \lintegral{B}{f}{\#} \leq \lintegral{B}{g}{\#} = P_2(B)$.
        Also, $P_1([x_0, \infty)) = f(x_0) + P_1((x_0, \infty)) \leq f(x_0) + P_2((x_0, \infty)) < g(x_0) + P_2((x_0, \infty)) = P_2([x_0, \infty))$.
        This shows that the conditions of Theorem \ref{thm:tailOrderSufficientConditionsGeneral} hold, and $P_1 \lesstail P_2$.
        
        \ref{item:discreteFiniteTailOrder-DominatingPointEquivalence}:
        We show that in the finite case the condition is not only sufficient, but also necessary.
        Suppose $P_1 \lesstail P_2$, let $v_1 < \dots < v_m$ denote the values in the common support.
        Let $n$ be the maximal index such that $f(v_n) \neq g(v_n)$: Such an $n$ must exist; else, $f(v_i) = g(v_i)$ for all $i$ would imply $P_1 = P_2$ in contradiction to our assumption.
        Since by the above sufficient condition, $f(v_n) > g(v_n)$ would imply $P_1 \greatertail P_2$, we must have $f(v_n) < g(v_n)$. So $x_0 \coloneqq v_n$ fulfills the desired property.
    \end{proof}

    In the special case of finite supports, a comparison by $\leqtail$ is equivalent to comparing the probability masses of the support points lexicographically from the right. The next definition defines such a lexicographic ordering $\leqRlex$ on vectors, and the subsequent corollary shows how comparisons by $\leqtail$ and $\leqRlex$ are equivalent.

    \begin{defn}
        Let $\leqRlex$ denote the \emph{reflected lexicographic order} on $\R^m$:
        \begin{gather*}
            (x_1, \dots, x_m) \leqRlex (y_1, \dots, y_m) \colonlra x = y \vee (\exists k \in [m]: x_k < y_k \wedge x_i = y_i ~\forall i > k).
        \end{gather*}
    \end{defn}
    $\leqRlex$ is similar to the usual lexicographic order, but starts comparing from the right instead of from the left \footnote{There seems to be no consistent name for this ordering in the literature; The term \emph{reflected lexicographic order} is used by \cite{bib:oeisRefLexOrder}.}. Just as the lexicographic order, it is antisymmetric and total. 
    
    \begin{cor}
        Let $P_1, P_2 \in \Dgeqone$ be discrete distributions with probability mass functions $f, g$ and finite common support $(\supp P_1 \cup \supp P_2) = \set{v_1, \dots, v_m}$, $v_1 < \dots < v_m$. Then
        \begin{gather}            
            P_1 \leqtail P_2 \lra v^f \coloneqq (f(v_1), \dots, f(v_m)) \leqRlex (g(v_1), \dots, g(v_m)) \eqqcolon v^g.
            \label{eq:discreteFiniteTailOrder-RLexEquivalence}
        \end{gather}
        \label{cor:discreteFiniteTailOrder-RLexEquivalence}
    \end{cor}
    \vspace{-1.2cm}
    \begin{proof}
        From Theorem \ref{thm:discreteTailOrderSuffConditions}, \ref{item:discreteFiniteTailOrder-DominatingPointEquivalence}, it follows directly that $P_1 \lesstail P_2 \lra (\exists k \in [m]: f(v_k) < g(v_k) \wedge f(v_i) = g(v_i) ~\forall i > k)$.
        Since $\leqRlex$ is antisymmetric, by the definition of $\leqRlex$ this is equivalent to $v^f \lessRlex v^g$.
        If instead $P_1 \leqtail P_2$, but not $P_1 \lesstail P_2$, this is equivalent to $P_1 = P_2$: This is because if $P_1 \neq P_2$, the sufficient condition of Theorem \ref{thm:discreteTailOrderSuffConditions} implies either $P_1 \lesstail P_2$ or $P_1 \greatertail P_2$. $P_1 = P_2$ then is equivalent to $v^f = v^g$, so in summary, $P_1 \leqtail P_2 \lra v^f \leqRlex v^g$.
    \end{proof}    

    In particular, this characterization shows that $\leqtail$ is antisymmetric and total on the set of distributions with finite support.
    
    \subsubsection{Counterexamples to Equivalences in the Sufficient Conditions for the Tail Order}
    \newcommand{\DRass}{\tilde{\DP}}
    While finite support allows us to turn the sufficient condition for the tail order into an equivalent condition, this does not work in the general case of distributions in $\Dgeqone \cap \M$.
    To show that none of the sufficient conditions from Theorems \ref{thm:acTailOrderSuffConditions} and \ref{thm:discreteTailOrderSuffConditions},  \ref{item:discreteTailOrder-DominatingPointSuffCondition} can be turned into equivalences, we construct several counterexamples.
    The original motivation of these counterexamples was to disprove a statement made in \cite{bib:rassGameRiskManagI} (and reproduced in \cite{bib:rassTotalOrderingOnLossDistributions}): 
    An old version of \cite{bib:rassGameRiskManagI} contained a condition similar to Theorem \ref{thm:acTailOrderSuffConditions}, and claimed that it was equivalent, not only sufficient, for a $\leqtail$-relationship between two distributions. The proof wrongly assumed that of two density functions, one always dominates the other from some point on until the end of the support.
    An updated version was meanwhile uploaded to arXiv that clarifies this issue. Since the paper only concerns discrete and absolutely continuous distributions with bounded support, and in the absolutely continuous case with continuous density functions, we construct counterexamples with these properties: Define 
    \begin{gather*}
        \DRass \coloneqq \set{P \in \DP: \text{$\exists b: P \in \DP_{[1, b]}$, $P$ discrete or AC with continuous density function}}.
    \end{gather*}
    
    When we introduced the tail order, we motivated it as a replacement for the usual stochastic order $\leqst$ that works on a wider range of distributions.
    However, the reader might have noticed that the definition of $\leqst$ used distribution function dominance (see Example \ref{ex:stochasticOrdersLeqELeqst}), yet all our sufficient conditions use dominance of the mass or density functions.
    It is not clear whether a sufficient condition in the form of 
    \begin{gather*}
        F_1(x) \geq F_2(x) ~ \forall x \geq x_0, \lambda({F_1 > F_2} \cap [x_0, \infty)) > 0 \implies P_1 \lesstail P_2
    \end{gather*}
    holds, either on all of $(\Dgeqone \cap M)$ or on the subset of distributions with bounded support.
    We will not further investigate this issue, but it might be an interesting question for follow-up work.
    However, we will also show with our counterexamples that even if such a condition was sufficient, it could not be turned into an equivalence.
    \cite[Theorem 2.15]{bib:rassGameRiskManagI} shows that the reverse direction holds on $\DRass$ under the assumption that one density eventually dominates, but we show that it does not hold in general.
    
    The specific statements we disprove in our counterexamples are as follows, where we assume that $P_1, P_2 \in \DRass$, $P_1 \neq P_2$ and $P_1, P_2$ have probability mass functions or probability density functions $f, g$, and distribution functions $F, G$, respectively.
    % This code is necessary for items with custom labels to be  referenced.
    \makeatletter
    \newcounter{labelleditemcount}
    \newcommand{\labelleditem}[1][]{%
        \refstepcounter{labelleditemcount}%
        \item[#1]\protected@edef\@currentlabel{#1}%
    }
    \makeatother
    \begin{enumerate}
        \labelleditem[(1a)]
        “\textit{If $P_1, P_2$ are discrete, then there exists an $x_0$ such that either $\forall x \geq x_0: f(x) \geq g(x)$, or $\forall x \geq x_0: f(x) \leq g(x)$.}”
        \label{item:counterExampleClaim1-fg}
        
        \labelleditem[(1b)]
        “\textit{If $P_1, P_2$ are discrete, then there exists an $x_0$ such that either $\forall x \geq x_0: F(x) \leq G(x)$, or $\forall x \geq x_0: F(x) \geq G(x)$.}”
        \label{item:counterExampleClaim1-FG}
        
        \labelleditem[(2a)]
        “\textit{If $P_1, P_2$ are discrete and $P_1 \lesstail P_2$, then $\exists x_0: \forall x \geq x_0: f(x) \leq g(x)$.}”
        \label{item:counterExampleClaim2-fg}
        
        \labelleditem[(2b)]
        “\textit{If $P_1, P_2$ are discrete and $P_1 \lesstail P_2$, then $\exists x_0: \forall x \geq x_0: F(x) \geq G(x)$.}”
        \label{item:counterExampleClaim2-FG}
        
        \labelleditem[(3a)]
        “\textit{If $P_1, P_2$ are AC and $P_1 \lesstail P_2$, then $\exists x_0: \forall x \geq x_0: f(x) \leq g(x)$}.”
        \label{item:counterExampleClaim3-fg}
        
        \labelleditem[(3b)]
        “\textit{If $P_1, P_2$ are AC and $P_1 \lesstail P_2$, then $\exists x_0: \forall x \geq x_0: F(x) \geq G(x)$}.”
        \label{item:counterExampleClaim3-FG}
    \end{enumerate}

    The first counterexample disproves the statements \ref{item:counterExampleClaim1-fg} and \ref{item:counterExampleClaim1-FG}, the second one disproves the statements \ref{item:counterExampleClaim2-fg} and \ref{item:counterExampleClaim2-FG}, and the third one disproves the statements \ref{item:counterExampleClaim3-fg} and \ref{item:counterExampleClaim3-FG}.
    
    \begin{ex}
        \textbf{\textsl{An infinite family of discrete distributions with pairwise alternating mass and distribution functions and bounded support}.}
        
        \label{ex:geometricSeriesSufficientTailOrderConditionCounterexample}
        This first example disproves claims \ref{item:counterExampleClaim1-fg} and \ref{item:counterExampleClaim1-FG}. It is not strictly necessary since the next counter example is more general and even disproves the weaker claims \ref{item:counterExampleClaim2-fg} and \ref{item:counterExampleClaim2-FG}.
        In this example, on the other hand, the mass and density functions of two distributions are shown to alternate, but it could not be shown that one moment sequence dominates the other.
        The example is included nonetheless because it involves an interesting graphical proof of the fact that the distribution functions alternate, and the distributions involved are less artificial than in the other examples.
        \footnote{In particular, $P_c$ from the example can be obtained as follows: If $X\sim \text{Geom}(p)$ is a random variable that is geometrically distributed on $\N$, then $2 - (1-p)^X \sim P_{1/(1-p)}$.}
        
%        This example uses the geometric series and has a nice visual proof of alternating distribution functions, but we have no complete proof of moment-sequence dominance.
        For $c > 1$, let $P_c$ be defined by $P_c\biggpars{\biggset{2-\bigpars{\frac{1}{c}}^k}} = (c-1)\bigpars{\frac{1}{c}}^k$ for all $k \in \N$, and zero everywhere else.
        These probability masses are constructed to sum to $1$ by the geometric series.
        We let $a_k \coloneqq \bigpars{\frac{1}{c}}^k$, and will later use that $\sum_{j = 1}^k  a_j = \frac{1 - \pars{\frac{1}{c}}^k}{c - 1} = \frac{1 - a_k}{c - 1}$.
        The distribution function  $F_c$ of $P_c$ can be written as a sum of indicator functions as
        \begin{align*}
            F_c 
            = \sum_{k \geq 1} (c-1) a_k * \1_{[2-a_k, \infty)}
            &= \biggpars{\sum_{k \geq 1} (c-1) \biggpars{\sum_{j = 1}^k  a_j} * \1_{[2-a_k, 2-a_{k+1})}} + \1_{[2, \infty)} \\
            &= \biggpars{\sum_{k \geq 1} (1-a_k) * \1_{[2-a_k, 2-a_{k+1})}} + \1_{[2, \infty)}.
        \end{align*}
        In particular, on the interval (1, 2), $F_c$ is bounded from above by the line $x \mapsto x-1$, touching it exactly at its discontinuity points, i.e. $F_c(2-a_k) = 1-a_k$. 
        This lets us prove graphically that distribution functions of such distributions can alternate:
        If we choose $c, d$ such that the sets $\set{2 - \frac{1}{c^k} \mid k \in \N}, \set{2 - \frac{1}{d^k} \mid k \in \N}$ are disjoint,
        then $F_c, F_d$ alternate as seen in \autoref{fig:distributionsGeometricSeriesCounterexample-2-3}.
        This is because at every discontinuity, one distribution function overtakes the other since it jumps to the diagonal while the other function is strictly smaller.
        In particular, $\set{P_p \mid p \text{ prime}}$ is an infinite family of distributions with pairwise alternating distribution functions.
        It is clear that the corresponding mass functions $f, g$ alternate as well, because every discontinuity of $F$ or $G$ corresponds to a non-zero probability mass in one distribution, while the other distribution assigns zero mass to that point.
        
        The $n$-th moment of $P_c$ can be expressed as a sum with finitely many terms:
        \begin{align*}
            m_n(P_c) 
            &= \sum_{k \geq 1} (c-1) \frac{1}{c^k} (2-\frac{1}{c^k})^n
            = \sum_{k \geq 1} (c-1) \frac{(2c^k-1)^n}{c^{kn+k}} \\
            &= (c-1)\sum_{j=0}^n \binom{n}{j}2^{n-j}(-1)^j \biggpars{\sum_{k \geq 1}\frac{c^{-jk}}{c^k}}
            = (c-1)\sum_{j=0}^n \binom{n}{j}2^{n-j}(-1)^j \frac{1}{c^{j+1}-1}.
        \end{align*}
        To also disprove the claims \ref{item:counterExampleClaim2-fg} and \ref{item:counterExampleClaim2-FG} with this example, we would have to show $P_c \lesstail P_d$ for all $1 < c < d$.
        While computer calculations suggest that $m_n(P_c)$ increases monotonically with $c$, it seems hard to show this rigorously.
        We therefore leave the question open whether the construction from this example can also be used to disprove the statements \ref{item:counterExampleClaim2-fg} and \ref{item:counterExampleClaim2-FG}.
        \qed
    \end{ex}
    
    \begin{figure}
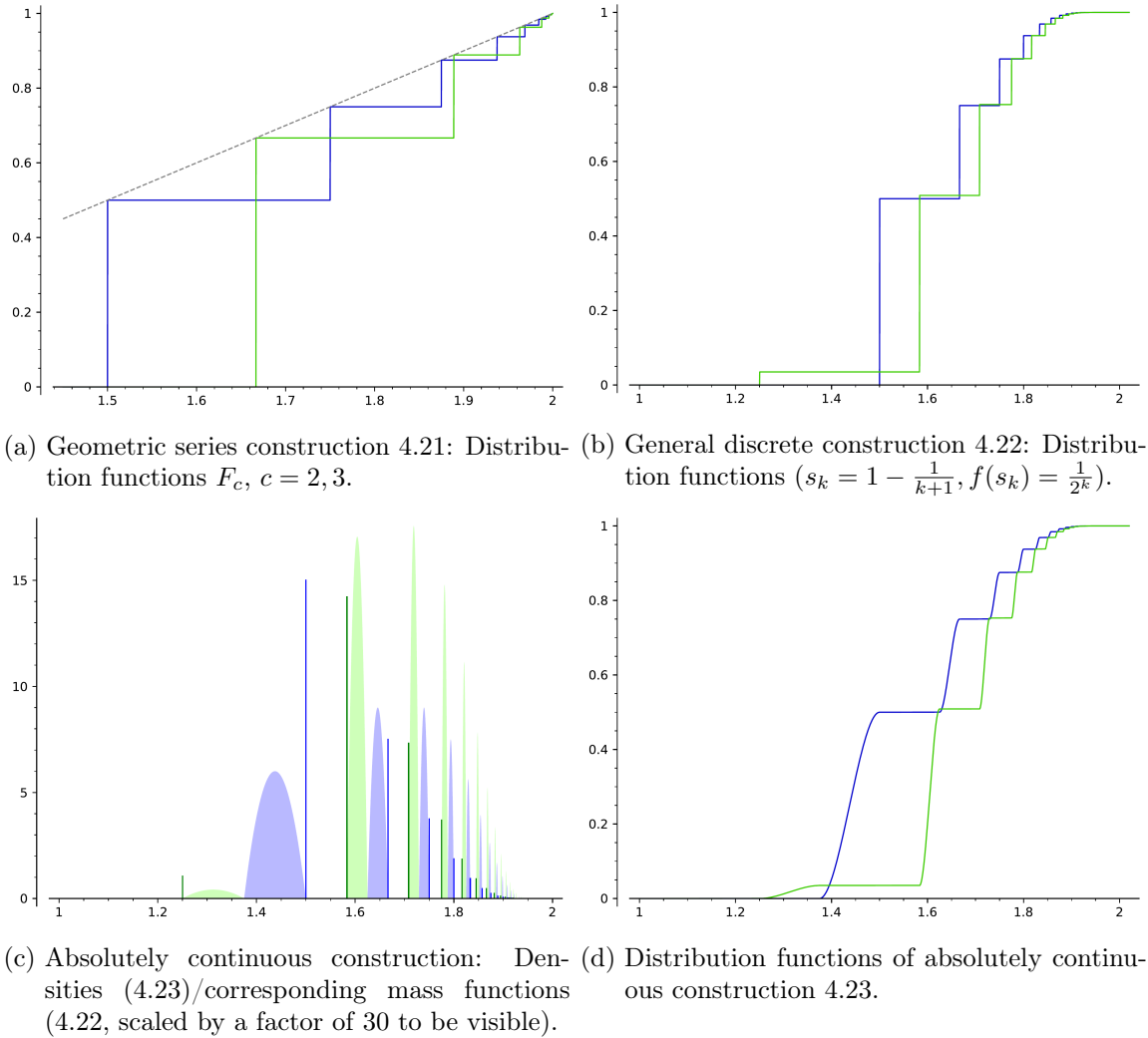

        \centering
        \begin{subfigure}[b]{0.49\textwidth}
            \includegraphics[width=\textwidth]{Pictures/geometric-series-distribution-2-3-half.png}
            \caption{Geometric series construction \ref{ex:geometricSeriesSufficientTailOrderConditionCounterexample}: Distribution functions $F_c$, $c=2, 3$.}
            \label{fig:distributionsGeometricSeriesCounterexample-2-3}
        \end{subfigure}
        \begin{subfigure}[b]{0.49\textwidth}
            \includegraphics[width=\textwidth]{Pictures/little-shifts-distribution.png}
            \caption{General discrete construction \ref{ex:discreteSufficientTailOrderConditionCounterexample}: Distribution functions ($s_k = 1-\frac{1}{k+1}, f(s_k) = \frac{1}{2^k}$).}
            \label{fig:littleShiftsCounterexample-distribution}
        \end{subfigure}
        
        \vspace*{0.01\textwidth}
        \begin{subfigure}[t]{0.49\textwidth}
            \includegraphics[width=\textwidth]{Pictures/ac-densities-masses.png}
            \caption{Absolutely continuous construction: Densities (\ref{ex:acSufficientTailOrderConditionCounterexample})/corresponding mass functions (\ref{ex:discreteSufficientTailOrderConditionCounterexample}, scaled by a factor of 30 to be visible).}
            \label{fig:acCounterexample-masses-densities}
        \end{subfigure}
        \begin{subfigure}[t]{0.49\textwidth}
            \includegraphics[width=\textwidth]{Pictures/ac-distribution.png}
            \caption{Distribution functions of absolutely continuous construction \ref{ex:acSufficientTailOrderConditionCounterexample}.}
            \label{fig:acCounterexample-distribution}
        \end{subfigure}
        
        \label{fig:sufficientTailOrderConditionsCounterexamples}
        \caption{Plots for the counterexamples \ref{ex:geometricSeriesSufficientTailOrderConditionCounterexample} - \ref{ex:acSufficientTailOrderConditionCounterexample}}
    \end{figure}

    \begin{ex}
        \textbf{
%            \textsl{A process to produce two discrete distributions with alternating mass and distribution functions and bounded support which are $\leqtail$-comparable}.
        \textsl{A $\leqtail$-ascending sequence of discrete distributions with alternating mass and distribution functions for consecutive elements and bounded support.}
        }
        
        \label{ex:discreteSufficientTailOrderConditionCounterexample}
        In this more general counterexample, we start with a discrete distribution $P$ whose support can be written as $\supp P = \set{s_i \mid i \in \N} \subseteq [a, b)$, $1 < a < b$, with the $s_i$ in increasing order, and whose probability mass function $f$ is strictly decreasing on the support, i.e. $\forall i < j: f(s_i) > f(s_j)$. We construct $\tilde{P}$ with mass function $g$, where each support point is shifted slightly to the right, with its probability adjusted to be only a little smaller, giving the remaining probability mass to the point $t_0 \coloneqq \frac{1+a}{2}$.
        First, set $t_i \coloneqq \frac{s_i+ s_{i+1}}{2}, i \in \N$.
        We adjust the probability of $t_i$ by a factor $c_i < 1$, i.e. $g(t_i) = c_i f(s_i)$, chosen such that the following properties are satisfied:
        \begin{enumerate}
            \item $c_i \geq \frac{s_i}{t_i}$,
            \item $c_i \geq 1-\frac{1}{2^i}$,
            \item $c_i \geq \frac{f(s_{i+1})}{f(s_i)}$.
        \end{enumerate}
        As we will show shortly, the first property ensures that $\tilde{P}$ has greater moments than $P$, the second property ensures that the distribution functions alternate, and the third property ensures that $g$ is monotonic on its support as well (so the process can again be applied to $\tilde{P}$).
        In summary, we define $g$ as follows:
        
        \begin{gather*}
            g(t_i) = c_i f(s_i), \quad
            c_i \coloneqq \max\biggset{\frac{s_i}{t_i}, 1-\frac{1}{2^i}, \frac{f(s_{i+1})}{f(s_i)}} < 1, \quad
            g(t_0) = \sum_{i \geq 1} (1-c_i) f(s_i).
        \end{gather*}
        Obviously the mass functions $f, g$ alternate, as $f(s_i) > 0, g(s_i) = 0$ and $g(t_i) > 0, f(t_i) = 0$ for all $i \in \N$. We now show that the moments of $\tilde{P}$ dominate the moments of $P$, and the distribution functions alternate.
        The first property of the $c_i$ implies that $\forall i \in \N: g(t_i)t_i  \geq  f(s_i)s_i$, so for the moments we have:
        \begin{multline*}
            m_n(\tilde{P}) - m_n(P)
            = g(t_0)t_0^n + \sum_{i \geq 1} g(t_i) t_i^n - f(s_i) s_i^n
            \geq g(t_0)t_0^n + \sum_{i \geq 1} g(t_i) t_i s_i^{n-1} - f(s_i) s_i^n  \\
            = g(t_0)t_0^n + \sum_{i \geq 1} s_i^{n-1} \undersetbrace{(g(t_i)t_i - f(s_i)s_i)}{\geq 0}
            \geq g(t_0)t_0^n > 0.
            %                \label{eq:momentPushingExample}
        \end{multline*}
        
        Secondly, the distribution functions alternate, see \autoref{fig:littleShiftsCounterexample-distribution}: On the one hand for $k \in \N$, 
        \begin{multline*}
            G(t_k) 
            = g(t_0) + \sum_{i=1}^k c_i f(s_i) 
            = g(t_0) - \biggpars{\sum_{i=1}^k (1-c_i)f(s_i)} + F(s_k) \\
            = \biggpars{\sum_{i \geq k+1} (1-c_i)f(s_i)} + F(s_k) > F(s_k) = F(t_k).
        \end{multline*}
        On the other hand, 
        \begin{multline*}
            G(s_k)
            = g(t_0) + \sum_{i=1}^{k-1} c_i f(s_i) 
            = g(t_0) + \sum_{i=1}^{k-1} f(s_i) - \sum_{i=1}^{k-1} (1-c_i) f(s_i) \\
            = F(s_k) - f(s_k) + g(t_0) - \sum_{i=1}^{k-1} (1-c_i) f(s_i)  
            = F(s_k) - f(s_k) + \sum_{i \geq k} (1-c_i) f(s_i).
        \end{multline*}
        Using the second property of $c_i$, the strict monotonicity of $f$ on its support, and the geometric sum identity $\sum_{i \geq k} \frac{1}{2^i} = 
%        \sum_{i \geq 1} \frac{1}{2^i} - \sum_{i \in [k-1]} \frac{1}{2^i} = 1 - \frac{1}{2}\frac{1 - 1/2^{k-1}}{1 - 1/2} = 
        \frac{1}{2^{k-1}}$, we further get:
        \begin{gather*}
            \sum_{i \geq k} (1-c_i)f(s_i) \leq \sum_{i \geq k} \frac{1}{2^i} f(s_i) < \sum_{i \geq k} \frac{1}{2^i} f(s_k) = \frac{1}{2^{k-1}} f(s_k) \leq f(s_k).
        \end{gather*}
        
        This shows that $G(s_k) < F(s_k)$. 
        So in summary we have $P \lesstail \tilde{P}$, and their probability mass, as well as distribution functions alternate.
        %            , such that the sufficient condition of Theorem \ref{thm:acTailOrderSuffConditions} is not satisfied for any $x_0$.
        Furthermore, $\tilde{P}$ again satisfies the conditions we originally made on $P$: By repeating the process, we get a whole sequence of distributions which is ascending with respect to $\lesstail$ and for any two consecutive distributions in it, the mass and distribution functions alternate (so in particular, the sufficient condition of Theorem \ref{thm:acTailOrderSuffConditions} is not satisfied).
        \qed
    \end{ex}
    \newpage
    \begin{ex}
        \textbf{
            \textsl{Two $\leqtail$-comparable absolutely continuous distributions with alternating density and distribution functions, continuous densities and bounded support.}
        }
        
        \label{ex:acSufficientTailOrderConditionCounterexample}
        For the absolutely continuous counterexample, we start with two discrete distributions $P_1, P_2$ where $P_2$ is obtained from $P_1$ by the process from Example \ref{ex:discreteSufficientTailOrderConditionCounterexample}: In particular, we require that $P_1 \lesstail P_2$, $\supp(P_1) = \set{s_1, s_2, \dots}$, $\supp(P_2) = \set{t_0, t_1, t_2, \dots}$ with $1 < t_0 < s_1 < t_1 < s_2 < t_2 < \dots < b$, and that their distribution functions $F, G$ alternate with $\forall k \geq 1: F(s_k) > G(s_k), F(t_k) < G(t_k)$.
        %            their mass functions $f, g$ should satisfy $f(s_i)s_i < g(t_i)t_i ~ \forall i$ (and therefore 
        
        We then shift the probability mass $P_1$ gives to each point $s_k$ to an interval to the left of $s_k$, and the probability mass $P_2$ gives to each point $t_k$ to an interval to the right of $t_k$, which preserves the order of the moments (illustrated in \autoref{fig:acCounterexample-masses-densities}). We do it in such a way that the distribution functions are again alternating, and the new density functions are continuous (\autoref{fig:acCounterexample-distribution}). We will call the new absolutely continuous distributions $\tilde{P}_1, \tilde{P}_2$, their density functions $\tilde{f}, \tilde{g}$, and their distribution functions $\tilde{F}, \tilde{G}$.
        
        To formalize this, let $h: \R \to \R, x \mapsto (6x - 6x^2)\1_{[0, 1]}(x)$. The function $h$ is a probability density function ($\lintegral{\R}{h}{\lambda} = 1$) supported on $[0, 1]$, which is continuous since $h(0) = h(1) = 0$.
        Denote by $h_{[a, b]}: x \mapsto \frac{1}{b-a}h\bigpars{\frac{x-a}{b-a}}$ the version of $h$ scaled to the interval $[a, b]$ in a way such that it still integrates to 1.
        %            Define $h_{a, b, s}: x \mapsto s*h(\frac{x-b}{b-a})*s$.
        %            Another way to view it is that the integral of $h$ continuously interpolates between the values 0 and 1 on the interval $[0, 1]$.
        Also, let $m_{k} = \frac{t_{k-1} + s_k}{2}$.
        Using this notation, we define $\tilde{f}, \tilde{g}$ by:
        \begin{gather}
            \tilde{f} = \sum_{k \geq 1}  f(s_k) h_{[m_k, s_k]}, \quad
            \tilde{g} = \sum_{k \geq 0}  g(t_k) h_{[t_k, m_{k+1}]}.
            \label{eq:acSufficientTailOrderConditionCounterexample-densitiesDefinition}
        \end{gather}
        By construction, it is clear that $\lintegral{\R}{\tilde{f}}{\lambda} = \lintegral{\R}{\tilde{g}}{\lambda} = 1$, since both of the sequences $(f(s_k))_{k \geq 1}, g(t_k))_{k \geq 0}$ sum to $1$. The distribution functions alternate:
        For $k \geq 1$,
        \begin{gather*}
            \tilde{F}(s_k) 
            = \lintegral{[1, s_k]}{\tilde{f}}{\lambda} 
            = \sum_{i=1}^{k} f(s_k) 
            = F(s_k)
            > G(s_k)
            = \sum_{i=1}^{k-1} g(t_k) 
            = \tilde{G}(s_k)
            ,\\                
            \tilde{G}(m_{k+1})                
            = \lintegral{[1, m_{k+1}]}{\tilde{g}}{\lambda} 
            = \sum_{i=1}^{k} g(t_k)
            = G(t_k)
            > F(t_k)
            = \sum_{i=1}^{k} f(s_k) 
            = \tilde{F}(m_{k+1}).
        \end{gather*}
        The ordering of the moments is preserved:
        \begin{multline*}
            m_n(\tilde{P_1})
            = \sum_{k \geq 1} \lintegral{[m_k, s_k]}{x^n h_{[m_k, s_k]}(x)}{\lambda(x)}
            < \sum_{k \geq 1} f(s_k)*s_k^n \\
            = m_n(P_1)
            < m_n(P_2) \\
            = \sum_{k \geq 0} g(t_k)*t_k^n
            < \sum_{k \geq 0} \lintegral{[t_k, m_{k+1}]}{x^n h_{[t_k, m_{k+1}]}(x)}{\lambda(x)}
            = m_n(\tilde{P_2}).
        \end{multline*}
        If the series in \eqref{eq:acSufficientTailOrderConditionCounterexample-densitiesDefinition} converge uniformly, continuity is preserved.
        For this we additionally require $\norm{f(s_k) h_{[m_k, s_k]}}_\infty = \frac{f(s_k)}{s_k - m_k} \toinfty{k} 0$, and similarly for $g$, which is the case if the probability masses $f(s_k)$ approach zero asymptotically faster than the consecutive differences of the $s_k$. For example, we can set $s_k = 2-\frac{1}{k+1}, f(s_k) = \frac{1}{2^k}$, and use $g, (t_k)_k$ constructed from it as in Example \ref{ex:discreteSufficientTailOrderConditionCounterexample}.
        \qed
    \end{ex}
    
    \subsection{Can the Tail Order Be Made a Total Order?}
    \label{subsec:tailOrderTotality}
    When we introduced the tail order in Section \ref{subsec:tailOrderDefinitionAndBasicProperties}, we shortly discussed its antisymmetry and totality properties: In short, $\leqtail$ is neither total nor antisymmetric on $M$. For antisymmetry, counterexamples of different distributions with equal moment sequences exist. Different distributions with equal moment sequences necessarily have unbounded support, and we do not know whether $\leqtail$ is antisymmetric on the subset of distributions with bounded support because it is not clear if the moment sequences of two different such distributions can disagree only finitely often.
%    however anti-symmetry \emph{does} hold on the reasonably large subset of distributions with bounded support.
    We can ask a similar question about the totality of $\leqtail$: While $\leqtail$ is not total on all of $M$,
    a natural question is if there is some useful subset of $\DP$ where $\leqtail$ is total, and we will discuss this question on the following pages.
    Corollary \ref{cor:discreteFiniteTailOrder-RLexEquivalence} shows that $\leqtail$ is total if only distributions with finite support are considered, yet limiting ourselves to finitely-supported distributions is quite restrictive.
    Since our only example of incomparable distributions so far relied on negative values in the support, one might hope that $\leqtail$ is total on $\DP_{\geq 0} \cap M$. However, we will show shortly that this is not the case.
    
    In \cite{bib:rassGameRiskManagI,bib:rassTotalOrderingOnLossDistributions}, the discussion focuses on the particular subset of distributions that have \emph{bounded support in $[1, \infty)$} and are either \emph{discrete with finite support}, or \emph{absolutely continuous with a continuous density function}.
    The lemmas \cite[Lemma 2]{bib:rassTotalOrderingOnLossDistributions} and \cite[Lemma 2.4]{bib:rassGameRiskManagI} wrongly claim that such distributions are always $\leqtail$-comparable: A proof is given for the absolutely continuous case; but as already discussed to motivate the previous counterexamples, it contains an error, as it implicitly assumes that of two density functions, one always dominates the other from some point on. This statement was disproved by the counterexamples \ref{ex:geometricSeriesSufficientTailOrderConditionCounterexample} - \ref{ex:acSufficientTailOrderConditionCounterexample}. Note that the latter paper was meanwhile updated to correct the error.
    Since the proof is erroneous, it is an interesting question whether the totality of $\leqtail$ on the set of distributions with bounded support in $[1, \infty)$ can be shown in another way, since it would be desirable for the application of $\leqtail$ to distribution-valued games if such a theorem could be proven.
    Unfortunately, it turns out that $\leqtail$ is not total on that set, and it is also not total if only absolutely continuous distributions are considered.
    It was conjectured during most the writing process of this thesis that this totality statement does hold, and the steps taken towards the desired proof are included on the following pages. However towards the end of the writing process, a counterexample was constructed by Jeremias Epperlein, and we will use this example at the end of this subsection to show that $\leqtail$ is not total on the set of distributions with bounded support in $[1, \infty)$.
    
    \subsubsection{The Moment Problem and its Variants}
    Since our questions depend on moment sequences by the definition $\leqtail$, it will be helpful to know about the properties of such sequences.
    The question to find out if a given sequence $(m_n)_{n \in \N_0}$ is a moment sequence for some Borel measure $\mu$ on $\R$ is known as the \emph{moment problem}, and there are several variants studied in the literature:
%    to characterize moment sequences is known as the \emph{moment problem}, and there are several variants:
    The \emph{Hamburger moment problem} concerns measures supported on a subset of $\R$, the \emph{Stieltjes moment problem} is about measures supported on a subset of $[0, \infty)$,
    and the \emph{Hausdorff moment problem} deals with measures supported on a subset of $[0, 1]$. For all three versions, conditions are known that are both sufficient and necessary for $(m_n)_{n \in \N_0}$ to be a moment sequence of the respective kind.
    The moment problem has been extensively analyzed in the literature, for example in \cite{bib:shohatTheProblemOfMoments}, \cite{bib:akhiezerClassicalMomentProblem}, or the more recent \cite{bib:schmuedgenTheMomentProblem}.

    \subsubsection{Non-Totality in the Stieltjes Case}
    The Hamburger moment problem is too general for our case, since negative values in the support can lead to alternating moment sequences, making the tail order non-total.
    However, we can make use of a result relating Stieltjes and Hamburger moment sequences:
    A moment sequence is called \emph{Hamburger-}/\emph{Stieltjes-determinate} if there is a \emph{unique} measure of the respective type with that moment sequence \cite[p.68]{bib:schmuedgenTheMomentProblem}
    \footnote{More specifically, a unique \emph{Radon measure}, see \cite[A.1]{bib:schmuedgenTheMomentProblem} for a definition. This makes no difference in our case: Every (locally) finite Borel measure on $\R$, and therefore every probability measure in $\DP$, is a Radon measure (e.g. \cite[Proposition II.3.1]{bib:malliavinIntegrationAndProbability}). Also, a moment sequence belongs to a probability measure if and only if $m_0 = 1$.}.
    The result we will use is that Stieltjes-determinateness in general does not imply Hamburger-determinateness:
    
    \begin{thm}[See {\cite[Fact A]{bib:linMomentProblemRecentDevelopments}}, citing {\cite[p.240]{bib:akhiezerClassicalMomentProblem}} and \cite{bib:chiharaIndeterminateHamburgerMomentProblems}]
        There exists a moment sequence that is Stieltjes-determinate, but not Hamburger-determinate.
    \end{thm}

    \begin{cor}
        There exist probability measures $P_1, P_2 \in M$ that have equal moment sequences and satisfy $P_1 \in \Dgeqzero, P_2 \notin \Dgeqzero$.
        \label{cor:equalMomentSequencesStieltjesHamburger}
    \end{cor}
    \begin{proof}
        \sloppypar{
        Let $(m_n)_{n \geq 0}$ be a moment sequence which is Stieltjes-determinate, but not Hamburger-determinate.
        Let $\mu_1$ be the unique measure in $\Dgeqzero$ with that moment sequence.
%         and support in $[0, \infty)$.
        Let $\mu_2 \neq \mu_1$ be a different measure with that moment sequence, which exists since the sequence is not Hamburger-determinate. 
        Since $\mu_1$ is unique in the Stieltjes sense, the support of $\mu_2$ must overlap with $(-\infty, 0)$.
        If $m_0 \neq 1$, the measures constructed are not probability measures: 
        We normalize them and define $P_1 = \frac{1}{m_0}\mu_1, P_2 = \frac{1}{m_0}\mu_2$, which are probability measures which both have the moment sequence $(\frac{m_n}{m_0})_{n \geq 0}$.
    }
    \end{proof}
    
    This result allows us to show that two Stieltjes moment sequences can alternate: 
    \begin{ex}
%        We construct a pair of probability measures supported on a subset of $[0, \infty)$, which have infinitely alternating moment sequences:
        Let $P_1, P_2 \in M$ be measures that satisfy $P_1 \in \Dgeqzero, P_2 \notin \Dgeqzero$ and have the same moment sequence $(m_n)_{n \geq 0}$, as constructed in Corollary \ref{cor:equalMomentSequencesStieltjesHamburger}.
        Define a measure $\tilde{P}_2: A \mapsto P_2\bigpars{A \cap [0, \infty)} + \1_A(0) P_2\bigpars{(\infty, 0)}$, which shifts the probability mass $P_2$ puts on the negative semi-axis to the point $0$, and has support in $[0, \infty)$.
        Then its moment sequence $(\tilde{m}_n)_{n \geq 0}$ is given by $\tilde{m}_0 = 1$, and for $n > 0$:
        \begin{gather*}
            \tilde{m}_n 
            = \lintegral{[0, \infty)}{x^n}{\tilde{P}_2}
            = \lintegral{[0, \infty)}{x^n}{P_2}
            = \lintegral{\R}{x^n}{P_2} - \lintegral{(-\infty, 0)}{x^n}{P_2}
            = m_n - \lintegral{(-\infty, 0)}{x^n}{P_2}.
        \end{gather*}
        By construction of $P_2$, the term $\lintegral{(-\infty, 0)}{x^n}{P_2}$ is strictly positive for $n$ even, and strictly negative for $n$ odd.
        Therefore $(\tilde{m}_n)_{n \geq 0}$ alternates around $(m_n)_{n \geq 0}$,
        and both sequences are moment sequences of probability measures supported on a subset of $[0, \infty)$.
    \end{ex}

    This shows that the set $(\Dgeqzero \cap \M)$ of probability measures supported on a subset of $[0, \infty)$ that have moments of all orders is still too large for $\leqtail$ to be total.
%    \todo{Give a counterexample explicitly.}
    
    \subsubsection{Distributions with Non-Negative Bounded Support}
    Next we look at distributions with support in a bounded interval $[a, b]$, $0 \leq a < b$.
    This is related to the Hausdorff moment problem where the bounded interval is $[0, 1]$.
    A sufficient and necessary condition for a sequence to be a Hausdorff moment sequence is based on repeatedly taking differences of successive terms.
    \begin{defn}
        The \emph{difference operator} $\Delta: \R^{\N_0} \to \R^{\N_0}$ maps a sequence of real numbers to the sequence of its successive differences:
        \begin{gather*} 
            \Delta((s_n)_{n \in \N_0}) \coloneqq (s_{n+1} - s_n)_{n \in \N_0}, ~ \text{ also written as }\,
            (\Delta s)_n \coloneqq s_{n+1} - s_n.
        \end{gather*} 
    \end{defn}

    \begin{thm}[\cite{bib:hausdorffMomentprobleme}]
        A sequence $m = (m_n)_{n \in \N_0}$ is a Hausdorff moment sequence, i.e. is the moment sequence of a measure $\mu$ on with support in $[0, 1]$,
        if and only if
        \begin{gather}
            \forall n, k \in \N: (-1)^k (\Delta^k m)_n \geq 0.
            \label{eq:completelyMonotonicSequence}
        \end{gather}
        A sequence that satisfies \eqref{eq:completelyMonotonicSequence} is called \emph{completely monotonic} \cite[Section III.4]{bib:widderTheLaplaceTransform}.
        \label{thm:hausdorffMomentSequenceCharacterization}
    \end{thm}
    \begin{proof}[Proof (necessity)]
        We only show here that condition \eqref{eq:completelyMonotonicSequence} is necessary, which is the easier part of the proof.
        Let $\mu$ be the measure with support in $[0, 1]$ which has $m$ as its moment sequence.
        We first prove by induction over $k$ that $(\Delta^k m)_n = \lintegral{[0, 1]}{x^n (x-1)^k}{\mu(x)}$:
        The statement holds for $k=0$, since
        \begin{gather*}
            (\Delta m)_n 
            = m_{n+1} - m_{n} 
            = \lintegral{[0, 1]}{x^{n+1}}{\mu(x)} - \lintegral{[0, 1]}{x^n}{\mu(x)} 
            = \lintegral{[0, 1]}{x^n(x-1)}{\mu(x)}.
        \end{gather*}
        For the induction step, assume the statement holds for $k \in \N$. Then for $k+1$:
        \begin{gather*}
            (\Delta^{k+1} m)_n
%            = (\Delta^k m)_{n+1} - (\Delta^k m)_n \\
            = \lintegral{[0, 1]}{x^{n+1} (x-1)^k}{\mu(x)} - \lintegral{[0, 1]}{x^n (x-1)^k}{\mu(x)}
            = \lintegral{[0, 1]}{x^{n}(x-1)^{k+1}}{\mu(x)}.
        \end{gather*}
        From this we directly get that $(-1)^k (\Delta^k m)_n = \lintegral{[0, 1]}{x^n (1-x)^k}{\mu(x)}$.
        Since we integrate over $[0, 1]$, the integrand is non-negative on the domain of integration, and $(-1)^k (\Delta^k m)_n \geq 0$.
    \end{proof}

    While Hausdorff's characterization works for distributions with support in $[0, 1]$, we are also interested in distributions supported in $[1, b]$ for some $b > 1$.
    Moment sequences behave somewhat differently in that case: For example, while they are monotonically decreasing in the former case, they are monotonically increasing in the latter, even growing without bound if there is some mass to the right of $1$.
    We can obtain a first necessary criterion, which looks similar to \eqref{eq:completelyMonotonicSequence}, for a sequence to be the moment sequence of such a distribution:
    \begin{cor}
        \label{cor:hausdorff-1-b-necessaryCondition}
        For any moment sequence $m$ of a measure $\mu$ with support in [1, b], all successive differences are non-negative:
        \begin{gather}
            \forall n, k \in \N: (\Delta^k m)_n \geq 0.
            \label{eq:hausdorff-1-b-necessaryCondition}
        \end{gather}
    \end{cor}
    \begin{proof}
        As in the last proof, we have that $(\Delta^k m)_n = \lintegral{[1, b]}{x^n (x-1)^k}{\mu(x)}$.
        Since in this case, $x-1 > 0$ for all $x \in [1, b]$, the integrand is non-negative, and $(\Delta^k m)_n \geq 0$.
    \end{proof}

    Another necessary condition similar to \eqref{eq:completelyMonotonicSequence} can be stated for moment sequences of distributions on $[0, b]$, using a modified difference operator:
    \newcommand{\DeltaB}[1]{\Delta\hspace{-.225cm}{\mathrel{\raisebox{0.045cm}{\scalebox{.9}{$\scriptscriptstyle#1$}}}\hspace{.22cm}}\hspace{-0.09cm}}
    \begin{lemma}
        \label{lemma:hausdorff-0-b-necessaryCondition}
        Let $b \geq 0$, $\DeltaB{b}: \R^{\N_0} \to \R^{\N_0}$, $(\DeltaB{b} m)_n \coloneqq m_{n+1} - b * m_n$.
        Then for any moment sequence $m$ of a measure $\mu$ with support in $[0, b]$:
        \begin{gather}            
            \forall n, k \in \N: (-1)^k (\DeltaB{b}^k m)_n \geq 0.
            \label{eq:hausdorff-0-b-necessaryCondition}
        \end{gather}
    \end{lemma}
    \begin{proof}
        Analogously to the proof of \ref{thm:hausdorffMomentSequenceCharacterization},
        we can prove by induction that $(\DeltaB{b}^k m)_n = \lintegral{[0, b]}{x^n (x-b)^k}{\mu(x)}$.
        Since the integrand in the expression $(-1)^k (\DeltaB{b}^k m)_n = \lintegral{[0, b]}{x^n (b - x)^k}{\mu(x)}$ is non-negative, we can conclude $(-1)^k (\DeltaB{b}^k m)_n \geq 0$.
    \end{proof}

    While Corollary \ref{cor:hausdorff-1-b-necessaryCondition} and Lemma \ref{lemma:hausdorff-0-b-necessaryCondition} give necessary conditions, it would be useful to have a condition that is also sufficient. However, we can use Lemma \ref{lemma:hausdorff-0-b-necessaryCondition} to show that Corollary \ref{cor:hausdorff-1-b-necessaryCondition} is not sufficient, as the following example demonstrates.
    
    \begin{ex}
        We construct a sequence that satisfies the condition of Corollary \ref{cor:hausdorff-1-b-necessaryCondition} and alternates around a moment sequence:
        Let sequences $s, q$ be given by $s_n = 4^n$ for $n \in \N_0$, $q_n = (-1)^n\frac{1}{2^n}$ for $n \in \N$ and $q_0 = 0$. $s$ is the moment sequence of $\delta_4$. We show that $(s + q)$ satisfies the conditions of Corollary \ref{cor:hausdorff-1-b-necessaryCondition}:
        Let $k \in \N$.
        By induction, one can show that $(\Delta^k s)_n = 3^k 4^n$ for all $n \in \N_0$, as well as $(\Delta^k q)_n = 3^k\frac{(-1)^{n+k}}{2^{n+k}}$ for $n \in \N$, and $(\Delta^k q)_0 = 3^k \frac{(-1)^k}{2^k} - (-1)^k$.
        Since $\Delta$ is a linear operator, we have ${(\Delta^k (s+q))_n} = {(\Delta^k s)_n + (\Delta^k q)_n}$ for all $n, k \in \N_0$. This sum is non-negative:
        If $n \geq 1$, we get $(\Delta^k (s+q))_n = 3^k\bigpars{4^n + \frac{(-1)^{n+k}}{2^{n+k}}} \geq 0$, since $4^n \geq 1, \frac{(-1)^{n+k}}{2^{n+k}} > -1$.
        If $n = 0$, we have ${(\Delta^k (s+q))_n} = {3^k\bigpars{1 + \frac{(-1)^{k}}{2^{k}}} - (-1)^k}$, which is non-negative as well:
        For $k = 0$, this becomes ${1*(1 + 1) - 1 = 1}$.
        For $k \geq 1$, we have $1 + \frac{(-1)^{k}}{2^{k}} \geq \frac{1}{2}$ and therefore $3^k\bigpars{1 + \frac{(-1)^{k}}{2^{k}}} - (-1)^k \geq \frac{3}{2} - (-1)^k > 0$.
        
%        for $k = 1$, $3*(1 - \frac{1}{2}) + 1 = \frac{7}{2}$
        
%          \geq 0$ for all $n, k$.
        So $(s+q)$ satisfies the condition of Corollary \ref{cor:hausdorff-1-b-necessaryCondition}. 
        By construction, it alternates around the moment sequence $s$, because $q$ alternates around 0.
        As the sequence also satisfies $s_0 + q_0 = s_0 = 1$ (necessary to be the moment sequence of a probability measure), these properties make it a candidate for a counterexample to the tail order being total for probability measures with non-negative, bounded support. 
        
        However, $(s+q)$ does not satisfy \eqref{eq:hausdorff-0-b-necessaryCondition} for any reasonable choice of $b$:
        If $(s+q)$ was the moment sequence of some measure $\mu \in \Dgeqzero$, then $\mu \in \DP_{[0, b]}$ would hold if and only if for all $n \geq 0: s_n + q_n < b^n$ (e.g. use Theorem \ref{thm:tailOrderSufficientConditionsGeneral} on $\mu, \delta_b$ for the “if” part, also cf. \cite[Proposition 4.1]{bib:schmuedgenTheMomentProblem}).
        We can take the bound $s_n + q_n \leq 5^n$, which is satisfied since $s_n \leq 4^n, q_n < 1$. So the support of $\mu$ must be a subset of $[0, 5]$.
        However, calculations show that $(-1)^2 (\DeltaB{5}^2 (s+q))_1$ is negative:
        We have $(s+q)_1 = 4 - \frac{1}{2} = \frac{7}{2}$, $(s+q)_2 = 16 + \frac{1}{4} = \frac{65}{4}$ and $(s+q)_3 = 64 - \frac{1}{8} = \frac{511}{8}$.
        From this we get $(\DeltaB{5} (s+q))_1 = \frac{65}{4} - 5 * \frac{7}{2} = - \frac{5}{4}$ and $(\DeltaB{5} (s+q))_2 = \frac{511}{8} - 5 * \frac{65}{4} = - \frac{139}{8}$, and finally $(\DeltaB{5}^2 (s+q))_1 = - \frac{139}{8} + 5*\frac{5}{4} = - \frac{89}{8}$.
        Therefore, $(-1)^2 (\DeltaB{5}^2 (s+q))_1  = - \frac{89}{8} < 0$, violating the necessary condition of Lemma \ref{lemma:hausdorff-0-b-necessaryCondition}.        
        So $(s+q)$ is not a moment sequence of a measure in $\DP_{[0, b]}$ even though the condition of Corollary \ref{cor:hausdorff-1-b-necessaryCondition} is satisfied, therefore this condition is not sufficient.
        \qed
    \end{ex}

    While the condition \eqref{eq:hausdorff-1-b-necessaryCondition} is not sufficient, the condition \eqref{eq:hausdorff-0-b-necessaryCondition} could well be:
    Possibly, a proof of \cite[Theorem 3.14]{bib:schmuedgenTheMomentProblem} for showing the sufficiency of the complete monotonicity condition \eqref{eq:completelyMonotonicSequence} could be altered to show the sufficiency of \eqref{eq:hausdorff-0-b-necessaryCondition}.
    It is also important to note that other characterizations for moment sequences of measures supported on an arbitrary compact interval $[a, b]$ exist, and two conditions are given in \cite[Theorem 3.13]{bib:schmuedgenTheMomentProblem}. We will not further go into these conditions, however:
%    Two conditions equivalent to a real sequence $m$ being an $[a, b]$-moment sequence are given in \cite[Theorem 3.13]{bib:schmuedgenTheMomentProblem}:
%    The first condition requires that both $m$ and $((a+b)E(m) - E(E(m)) - ab*m)$ are positive semidefinite sequences, where $E: (s_n)_{n \in \N_0} \mapsto (s_{n+1})_{n \in \N_0}$ denotes the shift operator.
%    The second condition requires that $L_m(p^2) \geq 0$ and $L_m((b-x)(x-a)q^2) \geq 0$ for all real polynomials $p, q \in \R[x]$. Here $L_m$ is the \emph{moment functional}, the linear functional $L_m: \R[x] \to \R$ uniquely defined by $L_m(x^n) \coloneqq m_n$ for all $n \in \N_0$.
%    These characterizations could be a different way to prove the totality of $\leqtail$ on $\DP_{[0, b]}$.
    The reason we studied conditions like \eqref{eq:hausdorff-1-b-necessaryCondition} and \eqref{eq:hausdorff-0-b-necessaryCondition} was that we hoped to show that two moment sequences satisfying these conditions could not alternate around each other, and therefore show that $\leqtail$ was a total order between distributions in $\Dgeqzero$ with bounded support.
    However in the meanwhile, a counterexample was constructed by Jeremias Epperlein, which will be presented in the following:
    It constructs two discrete distributions with infinite support in $[a, b]$, $0 \leq a < b$, that have alternating moment sequences.
    We first present a lemma needed for the construction, and then proceed with the counterexample.
    
    \begin{lemma}[{\cite{bib:epperleinTailOrderTotalCounterexample}}]
        \label{lem:tailOrderTotalCounterexample-lemmaCondition}
        Let $0 \leq a < b$. Let $n \in \N$ and $(c_l)_{l \in [n]}$, $(d_l)_{l \in [n]}$ be initial parts of sequences in $[a, b)$.
        Then there is an arbitrarily large $k \in \N$ and a $c_{n+1} \in (c_n, b)$ such that
        \begin{gather}
            \sum_{l=1}^{n} \frac{1}{2^l} d_l^k + \frac{1}{2^{n+1}} d_n^k + \sum_{l = n+2}^{\infty} \frac{1}{2^l} b^k
            \,<\, \sum_{l=1}^{n} \frac{1}{2^l} c_l^k + \sum_{l = n+1}^{\infty} \frac{1}{2^l} c_{n+1}^k.
            \label{eq:tailOrderTotalCounterexample-lemmaCondition}
        \end{gather}
    \end{lemma}
    \begin{proof}
        If we let $c_{n+1} = b$, we can cancel out the equal infinite sums on both sides, and then the inequality is equivalent to
        \begin{gather}
            \sum_{l=1}^{n} \frac{1}{2^l} d_l^k + \frac{1}{2^{n+1}} d_n^k
            < \sum_{l=1}^{n} \frac{1}{2^l} c_l^k + \frac{1}{2^{n+1}} b^k.
            \label{eq:tailOrderTotalCounterexample-lemmaCondition-cn+1=b}
        \end{gather}
        As $b > d_n > \dots > d_1 > 0$ and all $c_i$ are non-negative, the exponential growth of $b^k$ makes the right-hand side dominate the left-hand side for sufficiently large $k$.
        Therefore there exists a $\tilde{k} \in \N$ such that with $k = \tilde{k}$, \eqref{eq:tailOrderTotalCounterexample-lemmaCondition-cn+1=b} holds, and therefore \eqref{eq:tailOrderTotalCounterexample-lemmaCondition} holds for $c_{n+1} \coloneqq b$.
        With $k = \tilde{k}$ fixed, the right side of \eqref{eq:tailOrderTotalCounterexample-lemmaCondition} is constant and the left side depends continuously on $c_{n+1}$. We can therefore pick $c_{n+1} \in (c_n, b)$ such that the inequality still holds.
    \end{proof}
    
    \begin{ex}[{\cite{bib:epperleinTailOrderTotalCounterexample}}]
        \textbf{\textsl{Two discrete probability measures with support in a bounded interval $\mathbf{\boldsymbol{[a, b] \subseteq [0, \infty)}}$ that are not $\leqtail$-comparable}.}
        
        Let $0 \leq a < b$.
        We will construct two sequences $x = (x_l)_{l \in \N}$, $y = (y_l)_{l \in \N}$ in $[a, b]$ and define probability measures $P_x, P_y$ as sums of Dirac measures:
        \begin{gather*}
            P_x \coloneqq \sum_{l=1}^{\infty} \frac{1}{2^l} \delta_{x_l}, \quad
            P_y \coloneqq \sum_{l=1}^{\infty} \frac{1}{2^l} \delta_{y_l}.
        \end{gather*}
        The sequences are constructed recursively as follows: $x_1, y_1 \in [a, b)$ are chosen arbitrarily.
        For $n \in \N$, if $(n+1)$ is even, set $y_{n+1} \coloneqq y_n$, and apply Lemma \ref{lem:tailOrderTotalCounterexample-lemmaCondition} to $(x_l)_{l \in [n]}, (y_l)_{l \in [n]}$ to obtain two numbers $k_{n+1} > k_n$, $x_{n+1} \in (x_n, b)$ that satisfy \eqref{eq:tailOrderTotalCounterexample-lemmaCondition}.
        If $(n+1)$ is odd, set $x_{n+1} \coloneqq x_n$, and apply Lemma \ref{lem:tailOrderTotalCounterexample-lemmaCondition} to $(y_l)_{l \in [n]}, (x_l)_{l \in [n]}$ to obtain $k_{n+1} > k_n$, ${y_{n+1} \in (y_n, b)}$ that satisfy \eqref{eq:tailOrderTotalCounterexample-lemmaCondition}.
        In summary, the sequences look like this:
        \begin{gather*}
            x = (x_1, x_2, x_2, x_4, x_4, x_6, \dots),~ y = (y_1, y_1, y_3, y_3, y_5, y_5, \dots).
        \end{gather*}
        Condition \eqref{eq:tailOrderTotalCounterexample-lemmaCondition} ensures that the moment sequences of $P_x, P_y$ alternate: For $n \geq 1$, if $(n + 1)$ is even,
        \begin{multline*}
            m_{k_{n+1}}(P_x) 
            = \sum_{l=1}^{\infty} \frac{1}{2^l} (x_l)^{k_{n+1}}
            \geq  \sum_{l=1}^{n} \frac{1}{2^l} (x_l)^{k_{n+1}} + \sum_{l = n+1}^{\infty} \frac{1}{2^l} (x_{n+1})^{k_{n+1}} \\
            \underset{\eqref{eq:tailOrderTotalCounterexample-lemmaCondition}}
            {>}\, \sum_{l=1}^{n} \frac{1}{2^l} (y_l)^{k_{n+1}} + \frac{1}{2^{n+1}} (y_n)^{k_{n+1}} + \sum_{l = n+2}^{\infty} \frac{1}{2^l} b^{k_{n+1}}
            \geq \sum_{l=1}^{\infty} \frac{1}{2^l} (y_l)^{k_{n+1}}
            = m_{k_{n+1}}(P_y).
        \end{multline*}
        By a symmetric argument, if $(n+1)$ is odd, then $m_{k_{n+1}}(P_y) > m_{k_{n+1}}(P_x)$.
        \qed
    \end{ex}
    
    \begin{ex}\textbf{\textsl{Two absolutely continuous probability measures with support in a bounded interval $\mathbf{\boldsymbol{[a, b] \subseteq [0, \infty)}}$ that are not $\leqtail$-comparable}.}
        
        The previous example requires only slight modifications to produce two absolutely continuous distributions in $\DP_{[a, b]}$ that are incomparable by $\leqtail$.
        First as in Example \ref{ex:acSufficientTailOrderConditionCounterexample}, let $h$ be a continuous density on $[0, 1]$, and for real numbers $c < d$ define the scaled version $h_{[c, d]}$.
        
%        : \R \to \R, x \mapsto (6x - 6x^2)\1_{[0, 1]}(x)$ and a scaled version $h_{[c, d]}: x \mapsto \frac{1}{d-c}h\bigpars{\frac{x-c}{d-c}}$ for $c, d \in \R, c < d$, as in Example \ref{ex:acSufficientTailOrderConditionCounterexample}.
        For two sequences of intervals $([\ubar{x}_l, \bar{x}_l])_{l \in \N}$, $([\ubar{y}_l, \bar{y}_l])_{l \in \N}$, define densities
        $f = \sum_{l=1}^{\infty}\frac{1}{2^l} h_{[\ubar{x}_l, \bar{x}_l]}$, $g = \sum_{l=1}^{\infty}\frac{1}{2^l} h_{[\ubar{y}_l, \bar{y}_l]}$, and $P_x, P_y$ as the probability distributions with these densities.
%        The $k$-th moment of $P_x$ is given by 
        One can easily calculate that the following bounds hold for $m_k(P_x)$ and $m_k(P_y)$:
        \begin{gather}
%            m_k(P_x) = \sum_{l =1}^\infty \frac{1}{2^l} \lintegral{\R}{x^k h_{[\ubar{x}_l, \bar{x}_l]}(x)}{x} \geq \sum_{l =1}^\infty \frac{1}{2^l} \ubar{x}_l^k
            \sum_{l =1}^\infty \frac{1}{2^l} \bar{x}_l^k \,\geq\, m_k(P_x) \,\geq\, \sum_{l =1}^\infty \frac{1}{2^l} \ubar{x}_l^k,\quad
            \sum_{l =1}^\infty \frac{1}{2^l} \bar{y}_l^k \,\geq\, m_k(P_y) \,\geq\, \sum_{l =1}^\infty \frac{1}{2^l} \ubar{y}_l^k
            \label{eq:epperleinTailOrderTotalCounterexample-AC-PxMoment}
        \end{gather}
        We apply the trick of the previous example in modified form:
        If $(n+1)$ is even, set $y_{n+1} \coloneqq y_n$, and apply Lemma \ref{lem:tailOrderTotalCounterexample-lemmaCondition} to $(\ubar{x}_l)_{l \in [n]}, (\bar{y}_l)_{l \in [n]}$ to obtain $\ubar{x}_{n+1}, k_{n+1}$. If $(n+1)$ is odd, set $x_{n+1} \coloneqq x_n$ and apply Lemma \ref{lem:tailOrderTotalCounterexample-lemmaCondition} to $(\ubar{y}_l)_{l \in [n]}, (\bar{x}_l)_{l \in [n]}$ to obtain $\ubar{y}_{n+1}, k_{n+1}$. We also make sure that $\ubar{x}_{n+1} > \bar{x}_{n}, \ubar{y}_{n+1} > \bar{y}_n$, and pick $\bar{x}_{n+1} \in (\ubar{x}_{n+1}, b), \bar{y}_{n+1} \in (\ubar{y}_{n+1}, b)$ arbitrarily.
        Then the estimates in \eqref{eq:epperleinTailOrderTotalCounterexample-AC-PxMoment} and calculations as in the previous example ensure that for even $(n+1)$,  $m_{k_{n+1}}(P_y) < m_{k_{n+1}}(P_x)$, and for odd $(n+1)$, $m_{k_{n+1}}(P_y) > m_{k_{n+1}}(P_x)$.
        
        When we constructed a density in a similar way in Example \ref{ex:acSufficientTailOrderConditionCounterexample},
        we could ensure the continuity of the sum by uniform convergence.
        In the current example, it is not clear if uniform convergence can be achieved in the definitions of $f$ and $g$.
        We would need to show that the $[\ubar{x}_l, \bar{x}_l]$ can be chosen in such a way that their lengths approach zero slower than the probability masses $\frac{1}{2^l}$: Otherwise the maxima of the functions $\frac{1}{2^l} h_{[\ubar{x}_l, \bar{x}_l]}$ would not converge to zero, and $f$ would be discontinuous at the right endpoint $b$ (and similarly for $g$).
        We will not pursue this issue here, so whether the additional requirement of continuous density functions makes $\leqtail$ total is a question for further work.
        \qed
    \end{ex}

    \section{Nash Equilibria in Tail-Ordered Games}
    \label{sec:equilibriaInTailOrderedGames}
    We will now analyze Nash equilibria of distribution-valued games with respect to the stochastic tail order.
    An important result of this section is that, unlike their real-valued counterparts, these games fail to have mixed-strategy Nash equilibria in general.
    
    We restrict our attention to games with finitely supported distributions as payoffs, and we will show that such games already fail to have mixed-strategy Nash equilibria in many cases.
    As shown in Corollary \ref{cor:discreteFiniteTailOrder-RLexEquivalence}, for finitely supported distributions, the tail order can be reduced to a lexicographic comparison.
    To make our discussion notationally easier, we represent payoff distributions with finite common support as real-valued vectors of probabilities, ordered by the reflected lexicographic order. This is justified since the distribution-valued game and its corresponding vector-valued game are isomorphic and therefore have the Nash equilibria, as the subsequent lemma shows.
    \begin{defn}
        A \emph{vector-valued game} of dimension $m \in \N$ is a game $G$ with generalized payoffs in $\R^m$.
        If it is equipped with the preorder $\leqRlex$, we call it a \emph{reflected-lexicographically ordered} (also \emph{ref-lex-ordered}) game.
        \label{def:vectorValuedGame}
    \end{defn}

    \begin{notation}
        If $G$ is a distribution-valued game, we write
        \begin{gather*}
            \supp(G) \coloneqq \bigcup\limits_{\substack{s \in S,\\k \in [n]}} \supp(u_k(s))
        \end{gather*}
        for the common support of its payoff distributions.
    \end{notation}

    \begin{defn}
        Let $G$ be a distribution-valued game with finite common payoff support $\supp(G) = \set{x_1, \dots, x_m}$, $x_1 < \dots < x_m$.
        For each $k \in [n], s \in S$, let the payoff $u_k(s)$ have the probability mass function $f_{k, s}$.
        Then the \emph{probability vector game corresponding to $G$} is the vector-valued game $H$ with the same strategies as $G$ that has the payoffs $v_k(s) \coloneqq (f_{k, s}(x_1), \dots, f_{k, s}(x_m))$ for all $k \in [n], s \in S$.
    \end{defn}

    \begin{lemma}
        Let $G$ be a distribution-valued game such that $\supp(G)$ is finite, and $\supp(G) \subseteq [1, \infty)$.
        Let $H$ be the probability vector game corresponding to $G$.
        Then  $G_{\leqtail}$ and $H_{\leqRlex}$ are isomorphic.
    \end{lemma}
    \begin{proof}
        Let $\DP_{\supp(G)}$ denote the measures from $\DP$ supported on $\supp(G)$.
        Denote by $B \coloneqq \set{(p_1, \dots, p_m) \in [0, 1]^m \mid \sum_{i=1}^{m} p_i = 1} \subseteq \R^m$ the set of $m$-dimensional probability vectors.
        Then $\phi: \DP_{\supp(G)} \to B, P \mapsto (P(\set{x_1}), \dots, P(\set{x_n}))$ is a bijective function, satisfying $\phi(u_k(s)) = v_k(s)$ for all $k \in [n], s \in S$ by construction of $v_k(s)$.
        Since the payoffs of $G$ are from $\DP_{\geq 1}$, we can apply Corollary \ref{cor:discreteFiniteTailOrder-RLexEquivalence}:
%        , mapping distributions ordered by $\leqtail$ to real-valued vectors ordered by $\leqRlex$ preserves the ordering, 
        This shows that for all $s, t \in S, k \in [n]$, it holds that $u_k(s) \leqtail u_k(t)$ if and only if $v_k(s) \leqRlex v_k(t)$.
    \end{proof}

    \begin{defn}
        Let $G$ be an $m$-dimensional vector-valued game. For $i \in [m]$, let its \emph{$i$-th coordinate projected game}
        be defined as the real-valued game $G_i$ with the same strategies as $G$, and payoffs
        $u^i_k(s) \coloneqq (u_k(s))_i$ for all $k \in [n], s \in S$.
    \end{defn}

    \subsection[Existence Conditions for Mixed-Strategy Nash Equilibria in\\ Lexicographically-Ordered Games]{Existence Conditions for Mixed-Strategy Nash Equilibria in Lexicographically-Ordered Games}
    Not all ref-lex-ordered games have mixed-strategy Nash equilibria, which is an important difference from the theory of real-valued games. We will first illustrate this with an example, and then work out under which circumstances such Nash equilibria \emph{do} exist.
    
    \newcommand{\Gproj}[1]{G_{#1}}
    \newcommand{\Gprojsub}[1]{\bar{G}_{#1}}
    
    \begin{ex}
        \label{ex:reflectedLexicographicallyOrderedGameWithoutEquilibria}
%        Consider the vector-valued two-player zero-sum game $G$ represented by
        The following game is an example of a ref-lex-ordered game:
        \begin{gather}
            \centering
            \begin{tabular}{c|c|c|}
            	      &        $b_1$        &        $b_2$        \\ \hline
            	$a_1$ & (0.1,\; 0.8,\; 0.1) & (0.1,\; 0.7,\; 0.2) \\ \hline
            	$a_2$ & (0.6,\; 0.1,\; 0.3) & (0.8,\; 0.1,\; 0.1) \\ \hline
            \end{tabular}
        \end{gather}
        The game is a two-player matrix game (i.e. zero-sum), so we only specify payoffs for the first player.
        An isomorphic tail-ordered distribution-valued game can be constructed from it for any specified support set of size 3 that lies in $[1, \infty)$,
        and which specific support is chosen does not matter for the tail order.
        
        This game has no Nash equilibria with respect to $\leqRlex$.
        To see this, we will first take a closer look at the highest coordinate, projecting the game to it: Consider the projected game
        $\Gproj{3}$ represented by $\smallmat{0.1 & 0.2 \\ 0.3 & 0.1}$, which is a real-valued zero-sum game that we can solve with standard techniques.
        
        Computations done with the game-theoretic library included in the mathematical software {SageMath} (which implements the enumeration support algorithm from Section \ref{subsec:exactComputationNashEquilibriaSupportEnumerationAlgorithm}, see \cite{bib:sageNormalFormGameDocumentation})
        show that $\Gproj{3}$ has exactly one Nash equilibrium, with mixed strategies $s_1 = \bigpars{\frac{2}{3}, \frac{1}{3}}$ for player 1 and $s_2 = \bigpars{\frac{1}{3}, \frac{2}{3}}$ for player 2, yielding a payoff of $\frac{1}{6}$ to player 1.
        Applying this strategy to our original game $G$ yields the payoff vector $u = \bigpars{\frac{28}{90}, \frac{47}{90}, \frac{1}{6}}$.
        From Corollary \ref{cor:equilibriumStrategiesSupportHaveEqualPayoffs} we know that in the real-valued game $\Gproj{3}$, if any player unilaterally deviates from this strategy profile, the payoff stays the same: 
    %    For example, if player 1 plays the pure strategy $(0, 1)$ and player 2 continues to play $s_2$, the payoff for player 1 will still be $\frac{1}{6}$.
        For example, if player 1 plays the pure strategy $(1, 0)$ and player 2 sticks with $s_2$, the payoff for player 1 will still be $\frac{1}{6}$.
        However the payoff does not stay the same in the vector-valued game $G$:
        With the new strategy of player 1, the payoff is $\bigpars{\frac{9}{90}, \frac{66}{90}, \frac{1}{6}} \greaterRlex u$, making her payoff more preferable with respect to $\leqRlex$.
        If instead player 1 keeps her strategy $s_1$ and player 2 plays the pure strategy $(0, 1)$, the payoff for player 1 becomes $\bigpars{\frac{30}{45}, \frac{45}{90}, \frac{1}{6}} \lessRlex u$, which is more preferable for player 2.
        Therefore, $(s_1, s_2)$ is not a Nash equilibrium for $G$.
            
        Also, there can be no other Nash equilibria, as $(s_1, s_2)$ is the only Nash equilibrium for the projected game $\Gproj{3}$: Assume there is another set of strategies $(t_1, t_2) \neq (s_1, s_2)$ which \emph{is} a Nash equilibrium for $G$. But since $(t_1, t_2)$ is not a Nash equilibrium for $\Gproj{3}$, there is an incentive for some player to deviate from $(t_1, t_2)$ which improves their outcome for $G$ in the highest coordinate, thus also improving it with respect to $\leqRlex$, which is a contradiction.        
        At the core of this problem is that the projected games in the different coordinates have different mixed-strategy Nash equilibria.
        Specifically, while $\Gproj{3}$ has the Nash equilibrium $((\frac{2}{3}, \frac{1}{3}), (\frac{1}{3}, \frac{2}{3}))$,
        $\Gproj{2}$ instead has the Nash equilibrium $((1, 0), (0, 1))$, and $\Gproj{1}$ has the Nash equilibrium $((0, 1), (1, 0))$.\qed
    \end{ex}

    Example \ref{ex:reflectedLexicographicallyOrderedGameWithoutEquilibria} already gives an idea why different Nash equilibria in different coordinates can keep ref-lex-ordered games from having Nash equilibria. The rest of this subsection formalizes and refines the conditions seen in the example.
    While the example concerned a zero-sum game, i.e. a vector-valued game with respect to the orders ${(\leqRlex,\geqRlex)}$ (see Definition \ref{def:zeroSumGeneralizedPayoffs}), we will instead focus on games with respect to ${(\leqRlex,\leqRlex)}$ for simplicity. This is not an essential difference: 
    For any vector-valued bimatrix game $G$, we can define $\hat{G}$ as the game where the payoffs of player 2 are negated -- then $G_{{(\leqRlex,\geqRlex)}}$ and $\hat{G}_{{(\leqRlex,\leqRlex)}}$ are isomorphic, and we can apply the following results to $\hat{G}_{{(\leqRlex,\leqRlex)}}$ and thereby obtain information about the Nash equilibria of $G_{{(\leqRlex,\geqRlex)}}$.
    The same approach of course works for $G_{{(\geqRlex,\leqRlex)}}$, which captures the assumption of loss distributions instead of payoff distributions made in \cite{bib:rassGameRiskManagI}. 
    Also, we use the \emph{reflected} lexicographic order instead of the usual lexicographic order merely because it corresponds to the tail order more naturally; all results in this section can be rephrased for vector-valued games ordered by $\leqLex$.
    Recall from Definition \ref{def:distributionValuedGame} that we write $G_{\leqRlex}$ as a short notation for $G_{(\leqRlex, \leqRlex)}$.
    
    \begin{lemma}
        \label{lem:GmHasAllNashEquilibriaOfG}
        Let $G$ be a vector-valued bimatrix game with values in $\R^m$.
%        Then the set of Nash equilibria of $G_{\leqRlex}$ is a subset of the set of Nash equilibria of the projected game $\Gproj{m}$:
        Then any Nash equilibrium $(s_1, s_2)$ of $G_{\leqRlex}$ is also a Nash equilibrium of $\Gproj{m}$.
    \end{lemma}
    \begin{proof}
        Suppose $\Gproj{m}$ does not have $(s_1, s_2)$ as a Nash equilibrium.
        Then one of the players, say (without loss of generality) player 1, has an incentive to deviate in $\Gproj{m}$:
        This means that for some strategy $\tilde{s}_1 \in S_1$, $u_1^{m}(\tilde{s}_1, s_2) > u_1^{m}(s_1, s_2)$.
        But then for the payoffs in $G$, 
        \begin{gather*}
            u_1(\tilde{s}_1, s_2) = \bigpars{u_1^{1}(\tilde{s}_1, s_2), \dots, u_1^{m}(\tilde{s}_1, s_2)}
            \greaterRlex \bigpars{u_1^{1}(s_1, s_2), \dots, u_1^{m}(s_1, s_2)} = u_1(s_1, s_2).
        \end{gather*}
    \end{proof}

    \begin{cor}
        Let $G$ be a vector-valued bimatrix game with values in $\R^m$, and suppose that  $\Gproj{m}$ has exactly one Nash equilibrium $(s_1, s_2)$.
        Then
        \begin{gather*}
            G_{\leqRlex} \text{ has a Nash equilibrium } 
            \lra\quad G_{\leqRlex} \text{ has the Nash equilibrium } (s_1, s_2).   
        \end{gather*}
    \end{cor}

    \begin{thm}
        \label{thm:rlexGameProjectedGameSufficientCondition}
        Let $G$ be a vector-valued bimatrix game with values in $\R^m$.
        Suppose that $\Gproj{m}$ has the 
%        non-pure\footnote{With \emph{non-pure}, we mean “properly mixed”, whereas \emph{mixed equilibria} include pure and properly mixed equilibria (i.e. $(s_1, s_2)$ is non-pure if at least one of $s_1, s_2$ is not a pure strategy)}
        Nash equilibrium $(s_1, s_2)$.    
        Then the following implication holds:
        \begin{align}
            \label{eq:allProjectedGamesHaveTheNashEquilibrium}
                            &\forall i \in [m-1]: \Gproj{i} \text{ has the Nash equilibrium } (s_1, s_2) \\
            \Rightarrow\quad &G_{\leqRlex} \text{ has the Nash equilibrium } (s_1, s_2).
        \end{align}
    \end{thm}
    \begin{proof}
%        The first equivalence follows from the fact that any Nash equilibrium of $G$ must also be a Nash equilibrium of $\Gproj{m}$ by the previous lemma.
        Assume that \eqref{eq:allProjectedGamesHaveTheNashEquilibrium} holds.
%        Suppose for some player $k$ there is an incentive to deviate from $(s_1, s_2)$ in $G$, i.e. there is an $\tilde{s}_k \in S_k$ such that
%        $u_k(\tilde{s}_k) \greaterRlex u_
        Let $k \in [n]$ be an arbitrary player, and $\tilde{s}_k \in S_k$ an alternative strategy.
        Then for each projected game $\Gproj{i}$, because $(s_1, s_2)$ is a Nash equilibrium of $\Gproj{i}$, we have $u_k^{i}(\tilde{s}_k, s_{-k}) \leq u_k^{i}(s_k, s_{-k})$.
        Therefore in $G$,
        \begin{multline*}
            u_k(\tilde{s}_k, s_{-k}) 
            = \bigpars{u_k^{1}(\tilde{s}_k, s_{-k}), \dots, u_k^{m}(\tilde{s}_k, s_{-k})} \\
            \leqRlex \bigpars{u_k^{1}(s_k, s_{-k}), \dots, u_k^{m}(s_k, s_{-k})}
            = u_k(s_k, s_{-k}).
        \end{multline*}
    \end{proof}
    In Example \ref{ex:reflectedLexicographicallyOrderedGameWithoutEquilibria}, the reason why the game did not have a Nash equilibrium was that the highest-coordinate and the second-highest-coordinate projections had different Nash equilibria.
    Keeping this in mind, it seems as if \eqref{eq:allProjectedGamesHaveTheNashEquilibrium} may not only be a sufficient, but also a necessary condition for the existence of Nash equilibria. This is not the case, however: There is a similar, but weaker necessary condition that ref-lex-games with Nash equilibria need to satisfy. The next example illustrates how a ref-lex-game can have a Nash equilibrium even though the mixed-strategy equilibrium in the highest-coordinate projection is not reflected in the second-highest coordinate.
    \begin{ex}
        \label{ex:allProjectedEquilibriaAreEqualIsNotNecessary}
        Consider the ref-lex-ordered bimatrix game $G$ with the following payoffs for player 1/player 2, respectively:
        \begin{gather*}
            \centering
            \begin{tabular}{c|c|c|c|}
            	      & $b_1$  & $b_2$  & $b_3$  \\ \hline
            	$a_1$ & (1, 1) & (2, 2) & (3, 3) \\ \hline
            	$a_2$ & (2, 2) & (1, 1) & (3, 3) \\ \hline
            	$a_3$ & (3, 0) & (3, 0) & (3, 3) \\ \hline
            \end{tabular}\qquad
            \begin{tabular}{c|c|c|c|}
            	      &  $b_1$  &  $b_2$  &  $b_3$  \\ \hline
            	$a_1$ & (1, -1) & (2, -2) & (3, -3) \\ \hline
            	$a_2$ & (2, -2) & (1, -1) & (3, -3) \\ \hline
            	$a_3$ & (3, 0)  & (3, 0)  & (3, -3) \\ \hline
            \end{tabular}
        \end{gather*}
        The projected game $\Gproj{2}$ is a zero-sum game with payoffs $\pm\smallmat{1 & 2 & 3 \\ 2 & 1 & 3 \\ 0 & 0 & 3}$, and in the projected game $\Gproj{1}$ both players have the same payoff matrix $\smallmat{1 & 2 & 3 \\ 2 & 1 & 3 \\ 3 & 3 & 3}$.
        It can be computed that the only Nash equilibrium for $\Gproj{2}$ is $(s_1, s_2) \coloneqq \bigpars{(\frac{1}{2}, \frac{1}{2}, 0), (\frac{1}{2}, \frac{1}{2}, 0)}$.
        This is not a Nash equilibrium for $\Gproj{1}$, since for any player, deviating to $(0, 0, 1)$ gives a better payoff (the last row/column have strictly higher payoffs for both players).
        However $(s_1, s_2)$ indeed \emph{is} a Nash equilibrium for $G$:
        Any deviation that does not mix in the last row or column (for example, if player 1 deviates to $(0.3, 0.7, 0)$) will keep the payoffs the same for both players in both coordinates; any deviation that does mix in the last row or column will make the payoff for the respective player worse in the highest coordinate, and therefore also with respect to $\leqRlex$. \qed
    \end{ex}

    This example shows the obstacle when trying to turn \eqref{eq:allProjectedGamesHaveTheNashEquilibrium} into an equivalence.
    More or less, the idea applied in Example \ref{ex:reflectedLexicographicallyOrderedGameWithoutEquilibria} was that player 1 deviated from the highest-coordinate Nash equilibrium $(s_1, s_2)$, but the new strategy still mixed between the pure strategies in $\supp(s_1)$. That way, the payoff in the highest coordinate stayed the same while the second-coordinate payoff improved.
    However in Example \ref{ex:allProjectedEquilibriaAreEqualIsNotNecessary}, only deviating within the support of $s_1$ (or $s_2$ for player 2) of course keeps the high-coordinate payoff the same, but the low-coordinate payoff can only be improved by playing a strategy outside of the support.
%    The reason why $(s_1, s_2)$ is not a Nash equilibrium for $\Gproj{1}$
%    %TODO Using G1, while calling it second-coordinate payoff, is confusing!
%    is that switching to the third row/column improves payoffs, but in $G$ that would mean decreasing highest-coordinate payoffs, which are prioritized by $\leqRlex$. 
    So although the low-coordinate game $\Gproj{1}$ does not have $(s_1, s_2)$ as a Nash equilibrium, the equilibrium still holds for $G$, as no deviation that mixes only between the strategies in $\supp(s_1)$ gives an $\leqRlex$-better payoff.
    With this observation, we can get the aforementioned weaker necessary condition by restricting the lower-coordinate projections to subgames: We take out the rows and columns which are outside the support of the highest-coordinate Nash equilibrium.
    
    \begin{defn}
        Let $G = (n, (\Delta_1, \dots, \Delta_n), (u_1, \dots, u_n))$ be a real-valued game with pure-strategy sets $S_1, \dots, S_n$. For each $k \in [n]$, let $\bar{S_k} \subseteq S_k$ be a subset of the $k$-th player's strategies. Then the \emph{subgame} corresponding to these strategy sets is the game $\bar{G} \coloneqq (n, (\bar{S_1}, \dots, \bar{S_n}), (\bar{u}_1, \dots, \bar{u}_n))$, where $\bar{S} \coloneqq \bigtimes_{k \in [n]} \bar{S}_k$, and for all $k \in [n]$, $\bar{u}_k \coloneqq u_k\vert_{\bar{S}}$ is the restriction of $u_k$ to the new strategy space.
    \end{defn}

    \begin{notation}
        Let $G$ be an $m$-dimensional vector-valued bimatrix game with pure-strategy sets $S_1, S_2$. Let $\Gproj{m}$ have the Nash equilibrium $(s_1, s_2)$.
        In this case we write $T_1 \coloneqq \supp(s_1) \subseteq S_1, T_2 \coloneqq \supp(s_2) \subseteq S_2$ for the supports of $s_1$ and $s_2$, and $\Gprojsub{i}$ for the subgame corresponding to $T_1, T_2$ of the $i$-th coordinate projection $\Gproj{i}$.
        We also write $t_1, t_2$ for the strategies $s_1, s_2$ restricted to $T_1$ and $T_2$, respectively.
        (All of these notations depend on $(s_1, s_2)$, even though this is not explicitly written out every time for simplicity.)
    \end{notation}
    
    \begin{thm}
        \label{thm:rlexGameProjectedSubgameCharacterization}
        Let $G$ be an $m$-dimensional vector-valued bimatrix game, let $\Gproj{m}$ have the Nash equilibrium $(s_1, s_2)$.
        Then the following implication holds:
        \begin{align}
                             &G_{\leqRlex} \text{ has the Nash equilibrium } (s_1, s_2) \nonumber \\
            \Rightarrow\quad &\forall i \in [m-1]: \Gprojsub{i} \text{ has the Nash equilibrium } (t_1, t_2).
            \label{eq:allProjectedSubgamesHaveTheNashEquilibrium}
        \end{align}   
        If additionally in $\Gproj{m}$,
        all pure-strategy best responses to $s_1$ are in $T_2$, and all pure-strategy best responses to $s_2$ are in $T_1$,
        the reverse direction holds as well.
    \end{thm}
    \begin{proof}
        We use the contrapositive to show “$\Rightarrow$”. Assume \eqref{eq:allProjectedSubgamesHaveTheNashEquilibrium} does not hold, and let $l$ be the maximal index such that $\Gproj{l}$ does not have the Nash equilibrium $(t_1, t_2)$.
        Then one player $k$ has an improving strategy $\tilde{t}_k \in T_k$, i.e.
        $\bar{u}_k^l(\tilde{t}_k, t_{-k}) > \bar{u}_k^l(t_k, t_{-k})$.
        Furthermore since $l$ is maximal, for all $i \in \set{l+1, \dots, m}$, $\Gprojsub{i}$ has $(t_1, t_2)$ as a Nash equilibrium.
        So by Theorem \ref{thm:bestResponseMixing}, $\bar{u}_k^i(\tilde{t}_k, t_{-k}) = \bar{u}_k^i(t_k, t_{-k})$.
        
        Denote by $\tilde{s}_k \in S_k$ the strategy for $G$ corresponding to the restricted strategy $\tilde{t}_k$.
        Obviously for all $i \in [m-1]$, $\bar{u}_k^i(\tilde{t}_k, t_{-k}) = u_k^i(\tilde{s}_k, s_{-k})$.
        Since all projected-game payoffs for $i > l$ stay equal by switching from $s_k$ to $\tilde{s}_k$, but the $l$-th payoff increases, we have:
        \begin{multline*}
                         u_k(\tilde{s}_k, s_{-k}) 
                       = \bigpars{u_k^{1}(\tilde{s}_k, s_{-k}), \dots, u_k^l(\tilde{s}_k, s_{-k}), \dots, u_k^{m}(\tilde{s}_k, s_{-k})} \\
            \greaterRlex \bigpars{u_k^{1}(s_k, s_{-k}), \dots, u_k^l(s_k, s_{-k}), \dots, u_k^{m}(s_k, s_{-k})}
                       = u_k(s_k, s_{-k}).
        \end{multline*}
        So $(s_1, s_2)$ is not a Nash equilibrium of $G$.
        
        For “$\Leftarrow$”, assume that the additional condition holds.
        Without loss of generality, consider an alternative strategy $\tilde{s}_1$ for player 1, for which we want to show $u_1(\tilde{s}_1, s_2) \leqRlex u_1(s_1, s_2)$.
        If $\supp (\tilde{s}_1) \nsubseteq T_1$, by the additional condition, $\tilde{s}_1$ has in its support a non-best-response strategy to $s_2$ in $\Gproj{m}$, so it cannot be a best response by Theorem \ref{thm:bestResponseMixing}. Therefore $u_1^m(\tilde{s}_1, s_2) < u_1^m(s_1, s_2)$.
        If instead $\supp (\tilde{s}_1) \subseteq T_1$, and $\tilde{t}_1 \in T_1$ denotes the corresponding restricted strategy,
        we have $u_1^m(\tilde{s}_1, s_2) = u_1^m(s_1, s_2)$, and  $\forall i \in [m - 1]: u_1^i(\tilde{s}_1, s_2) = \bar{u}_1^i(\tilde{t}_1, t_2) \leq \bar{u}_1^i(t_1, t_2) = u_1^i(s_1, s_2)$ because $(t_1, t_2)$ is a Nash equilibrium in every $\Gprojsub{i}$. Therefore $u_1(\tilde{s}_1, s_2) \leqRlex u_1(s_1, s_2)$.
    \end{proof}

    The additional condition for equivalence holds in particular if $\Gproj{m}$ is non-degenerate (see Definition \ref{def:degenerateRealValuedGame}).
    \begin{cor}        
        \label{cor:degenerateGamesYieldLeqRlexEquivalence}
        Let $G$ be an $m$-dimensional vector-valued bimatrix game, let $\Gproj{m}$ have the Nash equilibrium $(s_1, s_2)$ and be non-degenerate.
        Then $G_{\leqRlex}$ has the Nash equilibrium $(s_1, s_2)$ if and only if for all $i \in [m-1]$, $\Gprojsub{i}$ has the Nash equilibrium $(t_1, t_2)$.
    \end{cor}
    \begin{proof}
%        Non-degenerateness implies that in $G_m$, all pure-strategy best responses to $s_1$ are in $T_2$, and all pure-strategy best responses to $s_2$ are in $T_1$:
        Since $G_m$ is non-degenerate, $s_1$ has at most $\abs{T_1}$ pure-strategy best responses, and $s_2$ has at most $\abs{T_2}$ pure-strategy best responses.
        By Theorem \ref{thm:bestResponseMixing}, all strategies in $T_2$ are best responses to $s_1$, and all strategies in $T_1$ are best responses to $s_2$. Therefore $\abs{T_2} \leq \abs{T_1}$ and $\abs{T_1} \leq \abs{T_2}$, which implies $\abs{T_1} = \abs{T_2}$.
        Since neither $s_1$ nor $s_2$ can have more than $\abs{T_1} = \abs{T_2}$ pure best responses, the additional condition of Theorem \ref{thm:rlexGameProjectedSubgameCharacterization} is satisfied, which yields the equivalence.
    \end{proof}

    If $\Gproj{m}$ is non-degenerate, we can also characterize \emph{pure-strategy} Nash equilibria of $G_{\leqRlex}$.
    \begin{cor}
        \label{cor:degenerateGamesYieldPureEquilibriaGmGEquivalence}
        Let $G$ be an $m$-dimensional vector-valued bimatrix game. If $\Gproj{m}$ is non-degenerate and $s_1, s_2$ are pure strategies, the following equivalence holds:
        \begin{gather*}
            \text{$G_{\leqRlex}$ has the Nash equilibrium $(s_1, s_2) \lra \Gproj{m}$ has the Nash equilibrium $(s_1, s_2)$.}
        \end{gather*}
    \end{cor}
    \begin{proof}
        “$\Rightarrow$” holds by Lemma \ref{lem:GmHasAllNashEquilibriaOfG}.
        For “$\Leftarrow$”, we use Corollary \ref{cor:degenerateGamesYieldLeqRlexEquivalence}: Since $\abs{\supp s_1} = \abs{\supp s_2} = 1$, all subgames $\Gprojsub{i}, i \in [m-1]$ are games with only one strategy for each player. They therefore trivially all have $(t_1, t_2)$ as Nash equilibrium.
        Therefore $G_{\leqRlex}$ has the Nash equilibrium $(s_1, s_2)$.
    \end{proof}

    The additional condition in Theorem \ref{thm:rlexGameProjectedSubgameCharacterization} cannot be left out, as the next example shows.
    \begin{ex}
        Consider the ref-lex-game with the following player 1/player 2 payoffs:
        \begin{gather*}
            \centering
            \begin{tabular}{c|c|c|c|}
            	      & $b_1$  & $b_2$  &  $b_3$  \\ \hline
            	$a_1$ & (0, 1) & (0, 2) & (0, -1) \\ \hline
            	$a_2$ & (0, 2) & (0, 1) & (0, -2) \\ \hline
            	$a_3$ & (1, 2) & (1, 1) & (1, -3) \\ \hline
            \end{tabular}\qquad
            \begin{tabular}{c|c|c|c|}
            	      &  $b_1$   &  $b_2$   &  $b_3$  \\ \hline
            	$a_1$ &  (0, 2)  &  (0, 1)  & (0, 0)  \\ \hline
            	$a_2$ &  (0, 1)  &  (0, 2)  & (0, 0)  \\ \hline
            	$a_3$ & (-1, -1) & (-1, -1) & (-1, 0) \\ \hline
            \end{tabular}
        \end{gather*}
        \sloppypar{
        The game $\Gproj{2}$ has payoffs $\smallmat{1 & 2 & -1 \\ 2 & 1 & -2 \\ 2 & 1 & -3} / \smallmat{\phantom{-}2 & \phantom{-}1 & 0 \\ \phantom{-}1 & \phantom{-}2 & 0 \\ -1 & -1 & 0}$,
        and the unique Nash equilibrium is $(s_1, s_2) = \bigpars{(\frac{1}{2}, \frac{1}{2}, 0), (\frac{1}{2}, \frac{1}{2}, 0)}$. However not just the first and second rows are pure best responses of player 1 to $s_2$, but the third row is a best response as well.
        The game $\Gproj{1}$ is a simple zero-sum game with payoffs $\pm \smallmat{0 & 0 & 0 \\ 0 & 0 & 0 \\ 1 & 1 & 1}$, where the third row is obviously preferable for player 1, and player 2 has no choice over the outcomes.
        In the projected subgame $\Gprojsub{1}$ with payoffs $\pm \smallmat{0 & 0 \\ 0 & 0}$, $(t_1, t_2) = \bigpars{(\frac{1}{2}, \frac{1}{2}), (\frac{1}{2}, \frac{1}{2})}$ is clearly a Nash equilibrium, so \eqref{eq:allProjectedSubgamesHaveTheNashEquilibrium} is satisfied.
        However in $G$, $s_1$ is not a best response to $s_2$ (giving player 1 a payoff of $(0, 1.5)$), since the strategy $(0, 0, 1)$ gives an $\leqRlex$-better payoff of $(1, 1.5)$, so $(s_1, s_2)$ is not a Nash equilibrium of $G$.
        }\qed
    \end{ex}

    To summarize our results: Any Nash equilibrium of a ref-lex-game $G$ must be a Nash equilibrium of its highest-coordinate game $\Gproj{m}$.
    If $\Gproj{m}$ has a Nash equilibrium $(s_1, s_2)$, then the following chain of implications holds:
    \begin{align*}
         &\forall i \in [m - 1]: \Gproj{i} \text{ has the Nash equilibrium } (s_1, s_2) \\
         \Rightarrow\quad &G_{\leqRlex} \text{ has the Nash equilibrium } (s_1, s_2) \\
         \Rightarrow\quad &\forall i \in [m-1]: \Gprojsub{i} \text{ has the Nash equilibrium } (t_1, t_2).
    \end{align*}
    The second implication clearly places strict requirements on ref-lex-games to have mixed-strategy Nash equilibria. For any such game $G_{\leqRlex}$, there are two possibilities: Either $G_m$, and therefore $G_{\leqRlex}$, has a pure Nash equilibrium. Otherwise if $G_m$ only has a mixed Nash equilibrium, the same equilibrium must be supported by all projected subgames $\Gprojsub{i}$ in order for $G_{\leqRlex}$ to have a Nash equilibrium at all.

    \subsubsection{Nash Equilibria in Distribution-Valued Tail-Order Games}
    These results for ref-lex-ordered games have the important consequence for distribution-valued games that even those distribution-valued games whose payoff distributions have only finite support do not have Nash equilibria in general.
    A concrete example of this was already given in Example \ref{ex:reflectedLexicographicallyOrderedGameWithoutEquilibria}: The payoff vectors all sum to $1$ (i.e. represent probability distributions), so the game is isomorphic to a distribution-valued game.
    %    We did not investigate other types of distributions, but we can hardly expect things to be different there: 
    As the existence of Nash equilibria already fails to hold in the case of such simple distributions, it also fails to hold in more general cases.
    For example, any game where the payoffs have finite common support can be converted to an isomorphic game with absolutely continuous payoffs by performing a convolution with an absolutely continuous distribution supported on $[-\epsilon, \epsilon]$ for some small enough $\epsilon$.
    This is illustrated in \autoref{fig:discrete-ac-convolution-example}.
    Therefore, games with absolutely continuous payoffs certainly also fail to have Nash equilibria in general.
    
    \begin{figure}
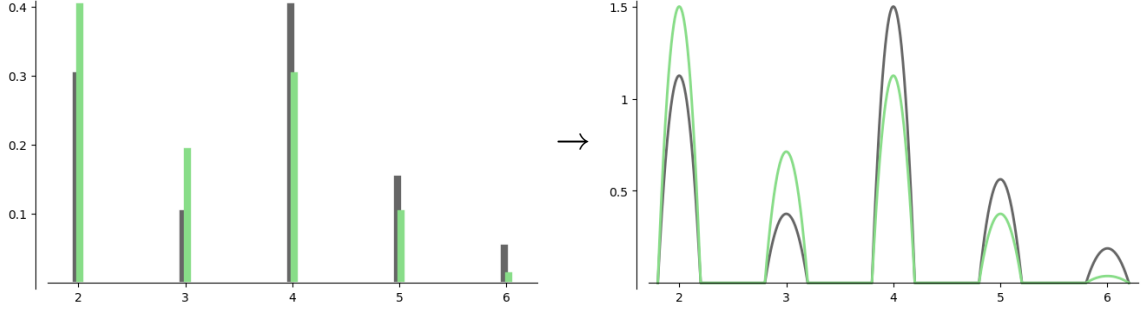

        \centering
        \includegraphics[width=0.47\textwidth]{Pictures/leqtail-rlex-masses-plot}
        \raisebox{60pt}{\LARGE$\shortrightarrow$}
        \includegraphics[width=0.47\textwidth]{Pictures/leqtail-rlex-densities-plot}
        \caption{Masses of discrete distributions $P_{1}, P_{2}$ (gray/green) on $\set{2, \dots, 6}$, and densities of corresponding absolutely continuous distributions $\tilde{P}_{1}, \tilde{P}_{2}$ obtained by convolution. Here $P_{1} \greatertail P_2$, and equivalently $\tilde{P}_1 \greatertail \tilde{P}_2$.}
        \label{fig:discrete-ac-convolution-example}
    \end{figure}
    
    The results of this section contradict an algorithm given in \cite[Section 3.1]{bib:rassGameRiskManagII} which supposedly calculates Nash equilibria of distribution-valued games
    \footnote{More precisely, \emph{multi-goal security strategies} (MGSS) for distribution-valued games. This is in a model where multiple objectives in form of multiple distributions are allowed, but the definition of a MGSS given in \cite[Definition 4.1]{bib:rassGameRiskManagI} coincides with a Nash equilibrium if the number of objectives is one.}.
%    However, this cannot be true if not all distribution-valued games have Nash equilibria.
    A similar algorithm is implemented in the \texttt{R} package \emph{HyRiM} (see \cite{bib:hyrimPackage}), a package that implements algorithms for distribution-valued games and which we will use in the next chapter.
    It is not completely clear what solutions the algorithm outputs for games without Nash equilibria, but it appears as if a Nash equilibrium of the highest-coordinate projection $G_m$ is calculated.

    \subsection{Probability that Lexicographically-Ordered Games have Nash Equilibria}
    
    In light of the strict requirements for ref-lex games to have non-pure Nash equilibria, one may wonder “how many” ref-lex games even have such equilibria. 
    It even seems plausible that “almost no” ref-lex game has a non-pure Nash equilibrium.
    This is indeed the case in a certain precise sense, and is formalized here in probabilistic way, similar to Theorem \ref{thm:probabilityOfPureNashEquilibria}:
    We show that the probability that a randomly chosen $G_{\leqRlex}$ has a Nash equilibrium is the same as the probability that $\Gproj{m}$ has a \emph{pure} Nash equilibrium, or in other words, that $G_{\leqRlex}$ has a non-pure Nash equilibrium with probability zero.
    
    \begin{thm}
        \label{thm:probabilityOfLexOrderedNashEquilibria}
        Let $m \geq 2$ be fixed and let $P \in \DP$ be an absolutely continuous probability distribution.
        Let $G$ be an $m$-dimensional vector-valued bimatrix game
        where all entries of all payoff vectors are picked iid from the distribution $P$.
        Then 
        \begin{gather}
            P(\set{\textit{\small$G_{\leqRlex}$ has a Nash equilibrium}}) = P(\set{\textit{\small$G_m$ has a pure Nash equilibrium}}), \label{eq:probabilityThatGRlexHasAnEquilibrium} \\
%            \frac{l!n!}{(l+n-1)!}, }\\
            P(\set{\textit{\small$G_{\leqRlex}$ has a non-pure Nash equilibrium}}) = 0. \label{eq:probabilityThatGRlexHasANonPureEquilibrium}
        \end{gather}
    \end{thm}
    \begin{proof}[Proof]
        If $G_{\leqRlex}$ has a Nash equilibrium, we first distinguish whether or not it has a pure equilibrium.
        If it has a pure equilibrium, by Corollary \ref{cor:degenerateGamesYieldPureEquilibriaGmGEquivalence} this is equivalent to $\Gproj{m}$ having a pure equilibrium. Therefore we get:
        \begin{align*}
            &P(\set{\textit{$G_{\leqRlex}$ \small has a Nash equilibrium}})
            = P(\set{\textit{\small$\Gproj{m}$ has a pure Nash equilibrium}}) \\
            +~ &P(\set{\textit{\small$G_{\leqRlex}$ has a non-pure Nash equilibrium and no pure Nash equilibria}}).
        \end{align*}
        The second probability on the right can be bounded from above by leaving out one condition:
        \begin{multline*}
            P(\set{\textit{\small$G_{\leqRlex}$ has a non-pure Nash equilibrium and no pure Nash equilibria}}) \\
            \leq P(\set{\textit{\small$G_{\leqRlex}$ has a non-pure Nash equilibrium}})
        \end{multline*}
        To conclude the proof of both \eqref{eq:probabilityThatGRlexHasAnEquilibrium} and \eqref{eq:probabilityThatGRlexHasANonPureEquilibrium}, it remains to show that this probability vanishes.
        One can show that \(P(\set{\textit{\small$\Gproj{m}$ is degenerate}}) = 0\) (this is hinted at in \cite[p.54]{bib:nisanAlgorithmicGameTheoryCh3EquilibriumComputation}, which states that “almost all” games with real-valued payoffs are non-degenerate; we refrain from giving a proof here).
        Since $\Gproj{m}$ is non-degenerate almost surely, we can assume that $\Gproj{m}$ is non-degenerate in our calculations without changing the probabilities.
%        Let $N$ denote the set of Nash equilibria of $\Gproj{m}$.

        Assume $G_{\leqRlex}$ has a non-pure Nash equilibrium.
        By Corollary \ref{cor:degenerateGamesYieldLeqRlexEquivalence}, since we assume $G_m$ is non-degenerate, this is equivalent to the statement that $\Gproj{m}$ has a non-pure Nash equilibrium $(s_1, s_2)$ such that for all $i \in [m-1]: \Gprojsub{i} \text{ has the Nash equilibrium } (t_1, t_2)$.
        Denote by $N_{\Gproj{m}}$ the set of Nash equilibria of $\Gproj{m}$. We can now further calculate the probabilities:
%        Since $\Gproj{m}$ is non-degenerate with probability $1$, we get:
        \begin{flalign}
               &P(\set{\textit{\small $G$ has a non-pure Nash equilibrium}}) \nonumber \\
               =~&P(\set{\textit{\small$\exists s \in N_{\Gproj{m}}$ non-pure such that $\forall i \in [m-1]: \Gprojsub{i}$ has the Nash equilibrium $t$}}) \nonumber \\
               \leq~&P(\set{\textit{\small$\exists s \in N_{\Gproj{m}}$ non-pure such that $\Gprojsub{m-1}$ has the Nash equilibrium $t$}}). \label{eq:probabilityGmAndGm-1HaveSameNonPureEquilibrium}
        \end{flalign}
        Next we show that for any specific non-pure strategy profile $s = (s_1, s_2) \in N_{\Gproj{m}}$, there is zero probability that $\Gprojsub{m-1}$ has the Nash equilibrium $t = (t_1, t_2)$.
        Since $s$ is non-pure and $\Gproj{m}$ can be assumed to be non-degenerate, we have that $r \coloneqq \abs{\supp s_1} = \abs{\supp s_2} \geq 2$.
        We decompose the restricted strategy $t_2$ of player 2 into the weights it assigns to the pure strategies: $t_2 \eqqcolon (q_1, \dots, q_r)$.
        Assume that $\Gprojsub{m-1}$ has payoff matrices $(X_{ij})_{i,j \in [r]}$ for player 1 and $(Y_{ij})_{i, j \in [r]}$ for player 2, where the $X_{ij}$, $Y_{ij}$ are 
        the randomly-picked, iid absolutely-continuously-distributed payoff entries.
        By the equation \eqref{eq:linearSystemNecessaryForNashEquilibrium}, if $\Gprojsub{m-1}$ has the equilibrium $(t_1, t_2)$, it is necessary that $t_2$ makes player 1 indifferent between the first two rows:
        \begin{gather*}
            \sum_{j=1}^{r} X_{1j} q_j = \sum_{j=1}^{r} X_{2j} q_j.
        \end{gather*}
        Therefore,
        \begin{gather*}
            P(\set{\textit{\small$\Gprojsub{m-1}$ has the Nash equilibrium $t$}})
            \leq P\biggpars{\biggset{X_{1j} = \frac{1}{q_j} \biggpars{\sum_{j=1}^{r} X_{2j} q_j - \sum_{j=2}^{r} X_{1j} q_j}}}.
        \end{gather*}
        The last probability can further be represented as the probability that $\tilde{X} \coloneqq (X_{11}, \dots, X_{1r}, X_{21}, \dots, X_{2r})$ lies on a specific hyperplane in $\R^{2r}$. The Lebesgue measure of such a plane is zero. The random vector $\tilde{X}$ is absolutely continuous as a tuple of independent, absolutely continuous random variables. This implies that the probability that $\tilde{X}$ lies in a Lebesgue null set is zero.
        We can therefore conclude that $P(\set{\textit{\small$\Gprojsub{m-1}$ has the Nash equilibrium $t$}}) = 0$.
        
        We use this to show that the probability in \eqref{eq:probabilityGmAndGm-1HaveSameNonPureEquilibrium} is zero.
        Recall from the algorithm in Section \ref{subsec:exactComputationNashEquilibriaSupportEnumerationAlgorithm} that in a non-degenerate game, there can be at most one Nash equilibrium for each pair of index sets $(I, J)$ indexing the pure strategy sets $S_1, S_2$ with $\abs{I} = \abs{J}$.
        Write $\mathcal{I}$ for the set of all such $(I, J)$ with $\abs{I} = \abs{J} \geq 2$.
        For some $(I, J) \in \mathcal{I}$, write $s_{I,J}$ for the Nash equilibrium of $\Gproj{m}$ supported in $I, J$ if it exists, and $s_{I,J} \coloneqq \bot$ (“undefined”) otherwise. Write $t_{I,J}$ for the respective restricted strategy profile.
%         (each $N_{\Gproj{m}}(I, J)$ is either a singleton set or the empty set).
        With this notation, we can calculate:
        \begin{align*}
                 &P(\set{\textit{\small$\exists s \in N_{\Gproj{m}}$ non-pure such that $\Gprojsub{m-1}$ has the Nash equilibrium $t$}}) \\
                 =~ &P\pars{\bigcup_{(I, J) \in \mathcal{I}}\set{\textit{\small $\Gproj{m}$ non-degenerate, $s_{I, J} \neq \bot, \Gprojsub{m-1}$ has the Nash equilibrium $t_{I, J}$}}} \\
            \leq~ &\sum_{(I, J) \in \mathcal{I}} 
                P(\set{\textit{\small $\Gproj{m}$ non-degenerate, $s_{I, J} \neq \bot, \Gprojsub{m-1}$ has the Nash equilibrium $t_{I, J}$}}) \\
                =~ &\sum_{(I, J) \in \mathcal{I}} 0 = 0.
        \end{align*}
        These calculations show that $P(\set{\textit{\small $G$ has a non-pure Nash equilibrium}}) = 0$.
    \end{proof}

    \begin{rem}
        We can use Theorem \ref{thm:probabilityOfPureNashEquilibria} to calculate the probability that $G_m$ has a pure Nash equilibrium:
        If the players have $l$ and $n$ pure strategies, \eqref{eq:probabilityThatGRlexHasAnEquilibrium} gives
        \begin{gather*}
            P(\set{\textit{\small$G_{\leqRlex}$ has a Nash equilibrium}}) = 1 - \sum_{k=0}^{\min(l, n)} (-1)^k k! \binom{l}{k} \binom{n}{k} \pars{\frac{1}{nl}}^k 
            \xrightarrow{l, n \to \infty} 1 - \frac{1}{e}.
        \end{gather*}
        If we consider randomly-generated zero-sum games instead, the proof of Theorem 
        \ref{thm:probabilityOfLexOrderedNashEquilibria} works as well. In this case,
        \begin{gather*}            
            P(\set{\textit{\small$G_{\leqRlex}$ has a Nash equilibrium}}) = \frac{l!n!}{(l+n-1)!} \xrightarrow{l, n \to \infty} 0.
        \end{gather*}
    \end{rem}

    \subsection{Kakutani's Theorem and Lexicographically-Ordered Games}
    In \cite[p.29-30]{bib:rassGameRiskManagI} and \cite[Theorem 3]{bib:rassUncertaintyInGamesGameSec15}, it is argued that the existence of mixed-strategy Nash equilibria for distribution-valued games follows from \emph{Glicksberg's Theorem} (see \cite[Theorem 1.3]{bib:fudenbergGameTheory}). This theorem is a generalization of Nash's existence theorem, as it guarantees the existence of Nash equilibria for games with continuous payoff functions and strategy sets that are compact subsets of a metric space. However it assumes real-valued payoffs and can not simply be applied to distribution-valued payoffs.
    
    Glicksberg's theorem extends Kakutani's fixed point theorem that was used in the proof of Theorem \ref{thm:existenceOfMixedStrategyEquilibria} to show that all real-valued games have a mixed-strategy Nash equilibrium.
    The application of Kakutani's theorem, however, did not depend on the payoff space:
    The proof idea of Theorem \ref{thm:existenceOfMixedStrategyEquilibria} can in principle be applied to other, arbitrary payoff models. This is because Kakutani's fixed point theorem worked on the best-reply correspondence $r: \Delta \to \Pot(\Delta)$, which maps real-valued vectors to sets of real-valued vectors. Since these real-valued vectors do not represent payoffs, but mixed-strategy profiles, the function signature of $r$ does not depend on the payoffs being real numbers.
    
    Since the theorem could in principle be applied, but tail-ordered distribution-valued games (and more specifically, ref-lex-ordered vector-valued games) do not have mixed-strategy Nash equilibria in general, some condition of Kakutani's theorem must be violated. 
    It turns out that the property that fails to hold with a lexicographic ordering is that $r$ must have a \emph{closed graph}.
    The inherent reason for this is that $\leqRlex$ is not closed as a subset of $(\R^m)^2$:
    For example, $(0, 1/n) \geqRlex (1, 0) ~ \forall n \in \N$, but in the limit as $n \to \infty$, $(0, 0) \lessRlex (1, 0)$.
    (Recall that in the proof where Kakutani's theorem was used, we explicitly mentioned that $\leq$ as subset of $\R^2$ is closed -- often stated as the “sandwich theorem” -- but this does not hold for a lexicographic ordering).
    
    This missing property of the ordering in turn leads to $r$ not having a closed graph. For example, again consider the game from Example \ref{ex:reflectedLexicographicallyOrderedGameWithoutEquilibria}, with a Nash equilibrium $(s_1, s_2) = \pars{(\frac{2}{3}, \frac{1}{3}), (\frac{1}{3}, \frac{2}{3})}$.
    If we define a sequence of player-2-strategies converging to $s_2$ by $s_2^{(n)} \coloneqq \pars{\frac{1}{3} + \frac{1}{2n}, \frac{2}{3} - \frac{1}{2n}}, n \in \N$, then one can calculate that in $\Gproj{3}$, player 1's only best response to each $s_2^{(n)}$ is the strategy $(0, 1)$, since the second row gives slightly bigger payoff than the first row (yet, both converge to the same limit with increasing $n$, one from above and one from below). Therefore $(0, 1)$ is also the unique best response in $G$.
    However in the limit, the highest-coordinate payoffs of both rows become equal, and the second coordinate makes the first row preferable by $\leqRlex$ (as already calculated in Example \ref{ex:reflectedLexicographicallyOrderedGameWithoutEquilibria} before). Therefore $(1, 0)$ is the unique best response to $s_2$. This shows that $r$ does not have a closed graph.
    
    \chapter{Tweaking the Stochastic Order: Segmenting Loss Distributions}
    \label{chap:segmentingLossDistributions}
    
    \section{The Tweakable Stochastic Order}
    A modification of the ideas on distribution-valued games and the stochastic order was proposed by Ali Alshawish in \cite{bib:tweakableStochasticOrders}.
    The motivation for the proposed idea is the insight that the tail order can only capture a pessimistic viewpoint: As its name says, the decision is based entirely on the \emph{tail} of the distribution. Since it is applied to loss functions, an arbitrarily small probability of a high loss can overshadow everything further to the left. In other words, the ordering only takes the worst-case scenario into account.
        
    The \emph{tweakable stochastic order} defined in \cite{bib:tweakableStochasticOrders}, on the other hand,
    is designed such that it can be tweaked to a decision maker's risk attitude, represented by a utility function.
    For finitely-supported distributions, it can represent as special cases the maximally risk-averse $\leqtail$, and the expected-value-ordering $\leqE$ (see Example \ref{ex:stochasticOrdersLeqELeqst}) which is considered risk-neutral.
    
    We define the tweakable stochastic order $\leqtw$, slightly adjusted to match the framework presented so far.
    Let $[a, b]$ be an interval where $a < b$. Let $\mathcal{I}$ be an partition of that interval, i.e. a finite subset of $[a, b]$ with elements $a = x_1 < x_2 < \dots < x_m < x_{m+1} = b$.
    For $P \in \DP$ and some Borel set $A \in \B$, let $E_A(P) \coloneqq \lintegral{A}{x}{P(x)}$ be its expected value if restricted to $A$.
    \begin{defn}
        \label{def:tweakableStochasticOrder}
        Let $P_1, P_2 \in \DP_{[a, b]}$ be probability measures.
        Then the \emph{tweakable stochastic order} $\leqtw$ given the partition $\mathcal{I}$ is defined by
        \begin{gather*}
            P_1 \leqtw P_2 \colonlra \bigpars{E_{[x_1, b]}(P_{1}), \dots, E_{[x_m, b]}(P_1)} \leqRlex \bigpars{E_{[x_1, b]}(P_2), \dots, E_{[x_m, b]}(P_2)}.
        \end{gather*}
        Distributions not in $\DP_{[a, b]}$ are incomparable by $\leqtw$.
    \end{defn}

    The definition uses overlapping intervals $[x_1, b], [x_2, b], \dots, [x_m, b]$. In each of these intervals, it takes the expected value; the resulting vectors are compared by the reflected lexicographic order. Alternatively, we can partition $[a, b]$ into intervals $[x_1, x_2), [x_2, x_3), \dots, [x_m, b]$, which is equivalent since a lexicographic comparison is used:

    \begin{lemma}
        Let $P_1, P_2 \in \DP_{[a, b]}$.
        Then
        \begin{alignat*}{3}
                                     &\bigpars{E_{[x_1, x_2)}(P_{1}), \dots, E_{[x_m, b]}(P_1)}\; &\leqRlex \bigpars{E_{[x_1, x_2)}(P_2), \dots, E_{[x_m, b]}(P_2)} \\
            \Longleftrightarrow\quad &\bigpars{E_{[x_1, b]}(P_{1}), \dots, E_{[x_m, b]}(P_1)} &\leqRlex \bigpars{E_{[x_1, x_2)}(P_2), \dots, E_{[x_m, b]}(P_2)}.
        \end{alignat*}
    \end{lemma}
    \begin{proof}
        Observe that for any $P_l$ ($l = 1, 2$), $E_{[x_k, b]}(P_l) = \bigpars{\sum_{i=k}^{m-1} E_{[x_i, b)}(P_l)} + E_{[x_m, b]}(P_l)$ for all $k \in [m]$.
        The reflected lexicographic ordering considers the $k$-th coordinate if and only if all higher coordinates are equal in both vectors:
        In this case, above sum representation implies that $E_{[x_k, b]}(P_1) < E_{[x_k, b]}(P_2) \lra E_{[x_k, x_{k+1})}(P_1) < E_{[x_k, x_{k+1})}(P_2)$, and analogously $E_{[x_k, b]}(P_1) = E_{[x_k, b]}(P_2) \lra E_{[x_k, x_{k+1})}(P_1) = E_{[x_k, x_{k+1})}(P_2)$.
        In other words, it does not matter for the ordering if values of higher coordinates are added/subtracted to lower coordinates consistently on both sides.
    \end{proof}

    Considering that $\leqtail$ is “essentially lexicographic” in the sense of Theorems \ref{thm:acTailOrderSuffConditions}, \ref{thm:discreteTailOrderSuffConditions}, and Corollary \ref{cor:discreteFiniteTailOrder-RLexEquivalence},
    $\leqtw$ can be understood as a coarser version of the tail order, in which the ordering decision is made considering probabilities in larger regions lexicographically instead of only considering single points.
    Since $\leqRlex$ is total, $\leqtw$ is a total order on $\DP_{[a, b]}$ (yet obviously not on $\DP$). It is not antisymmetric, since it is indifferent between any two distributions with equal expected values in each of the partition intervals.
    
    The order can be “tweaked” to a decision maker's risk attitude by choosing suitable partitioning points $\mathcal{I}$. \cite{bib:tweakableStochasticOrders} presents a method based on the decision-theoretic tool of utility functions that model risk-averse or risk-seeking behavior.    
    The utility function $u: [a, b] \to \R$ is assumed to be continuous, monotonically increasing, and without loss of generality $u([a, b]) = [0, 1]$.
    The proposed method partitions the range $[0, 1]$ of $u$ into $m$ equally-sized intervals, with partition points $\set{\frac{1}{m}, \dots, \frac{m-1}{m}}$, and constructs $\mathcal{I} = \set{a, x_2, x_3, \dots, x_m, b}$ such that $u(x_i) = \frac{i-1}{m}$ for $i = 2, \dots, m$.
    In other words, the points $x_i$ in $\mathcal{I}$ are defined as the $\frac{i-1}{m}$-quantiles of $u$, analogously to quantiles of probability distribution functions. Then $\leqtw$ is the tweakable stochastic order tweaked to the utility function $u$.
    
    \section{Games with Multiple Distribution Segments as Objectives}
    The stochastic order $\leqtw$ can of course be used to order distribution-valued games in the sense of Definition \ref{def:distributionValuedGame}, which was the approach chosen in \cite{bib:tweakableStochasticOrders}.
    In a new approach, we will use the segmentation idea in a different way and define multi-objective games: Each of the segments' expected values will be viewed as a distinct objective to be minimized.
    
    \subsection{Multi-Objective Games and Pareto-Nash Equilibria}
    Of course in such a multi-objective setting, one cannot hope to be able to minimize all objectives at once. Quite possibly, one objective can only take on its minimal value if another objective is not minimal. A common way to deal with this is to consider \emph{Pareto-optimal} solutions:
    A vector is Pareto-optimal if no improvement in any coordinate is possible without making another coordinate worse. Phrased in a different way, a vector is Pareto-optimal if no other of the vectors in consideration \emph{dominates} it in all coordinates.
    For example, if one considers the set $\set{v_1 = (4, 2, 4), v_2 = (2, 2, 2), v_3 = (0, 0, 5)}$, the Pareto-minimal vectors are $v_2$ and $v_3$, while $v_1$ is dominated by $v_2$ which has smaller or equal values in all coordinates.
    
    \begin{defn}
        Define the preorder
        \begin{align*}
        	x \leqpareto y ~ & \colonlra~ \exists i: x_i < y_i \vee x = y.
        \end{align*}
    \end{defn}
    We have $x \leqpareto y$ if $x$ is dominated by $y$ in at least one coordinate.
    Note that $\leqpareto$ is not antisymmetric: E.g. if $x_1 < y_1$ and $y_2 < x_2$, then $x \leqpareto y$ and $y \leqpareto x$.
    We have $x \lesspareto y$ iff $x \leqpareto y$ and not $y \geqpareto x$, i.e. if $x$ is dominated by $y$ in all coordinates, and strictly dominated in at least one.
    The set of Pareto-minimal elements in a set $S \subseteq \R^m$ is given by:
    \begin{gather*}
        \set{x \in S \mid \forall y \in S: x \leqpareto y}.
    \end{gather*}    
    
    The notion of multi-objective games was pioneered by David Blackwell in \cite{bib:blackwellVectorPayoffs}, and Lloyd Shapley defined the concept of multi-objective equilibria in the Pareto-optimal sense in \cite{bib:shapleyMultiobjectiveEquilibriumPoints} (yet the terminology of Pareto optimality was only later associated with it).
%    while later sources brought the terminology of Pareto optimality to this context. %TODO cite
    We can use our framework of games with generalized payoffs together with the order $\leqpareto$ to formalize this equilibrium concept.
    
    \begin{defn}
        \label{defn:multiObjectiveGame}
        Let $G$ be a vector-valued game as in Definition \ref{def:vectorValuedGame}, i.e. a game with generalized payoffs in $\R^m$.
        We call the Nash equilibria of $G$ with respect to $\leqpareto$ \emph{Pareto-Nash equilibria} (cf. \cite{bib:paretoNashEquilibria}).
        In this context, we also call $G$ a \emph{multi-objective game}.
    \end{defn}

    \begin{rem}
        A Pareto-Nash equilibrium can be interpreted as a strategy profile where no player can deviate to a different strategy and get a strictly better payoff in one coordinate without getting a strictly poorer payoff in another. 
        Definition \ref{defn:multiObjectiveGame} does not allow different players to have a different number of objectives $m_1, \dots, m_n$: We do not explicitly model this case to avoid cumbersome additional notation. However such payoffs can be represented by letting $m \coloneqq \max_{i \in [n]} m_i$, and padding all lower-dimensional payoff vectors with zeroes so they lie in $\R^m$.
    \end{rem}

    An important result from Shapley's paper \cite{bib:shapleyMultiobjectiveEquilibriumPoints} is that Pareto-Nash equilibria can be found by transforming the multi-objective game into a real-valued game via a weighted sum of the different objectives. In particular, the Pareto-Nash equilibria are exactly the equilibria of such weighted games for different weight vectors.
    
    \begin{thm}[Characterization of Pareto-Nash equilibria, cf. \cite{bib:shapleyMultiobjectiveEquilibriumPoints,bib:paretoNashEquilibria}]~\\
        Let $G$ be a mixed-extension multi-objective game with $n$ players and payoffs in $\R^m$. 
        For some $0 \neq w_1, \dots, w_n \in \Rp^m$, denote by $\tilde{G}^{w_1, \dots, w_n}$ the real-valued game with payoff functions
        $\tilde{u}_k: s \mapsto \innerprod{u_k(s)}{w_k}$ (taking the weighted sum of objectives in $u_k(s)$ by weights in $w_k$; $\innerprod{\dummydot}{\dummydot}$ denotes the dot product on $\R^m$).
        Then some strategy profile $s \in S$ is a Pareto-Nash equilibrium of $G$ if and only if it is a Nash equilibrium of $\tilde{G}^{w_1, \dots, w_n}$ for some weights $w_1, \dots, w_n$.
        \footnote{The theorem can be generalized further, see \cite{bib:shapleyMultiobjectiveEquilibriumPoints,bib:paretoNashEquilibria}: Instead of mixed-extension games, we could allow arbitrary games with continuous payoff functions defined on some convex strategy set.
%        Instead of $m$ objectives each, we could allow a different number of objectives for every player.
        Also we can restrict ourselves to weight vectors whose entries sum to 1, as scaling the payoffs of a real-valued game by a positive scalar preserves equilibria.
        }
        \label{thm:paretoNashEquilibriaWeightingCharacterization}
    \end{thm}

    The characterization gives rise to a simple algorithm which finds one Pareto-Nash equilibrium, as we can simply pick arbitrary weights $w_1, \dots, w_m$ and apply one of the usual algorithms to find Nash equilibria in real-valued games.
    However, it also shows that the equilibrium heavily depends on how the players weigh their objectives. Since there are infinitely many weight vectors, we cannot rule out that there can be infinitely many Pareto-Nash equilibria. 
    
    It is not in the scope of this work to examine the set of Pareto-Nash equilibria in detail, but it is interesting to at least have some idea of its possible structure, 
    the number of different equilibria and their relationships.
%    how many different equilibria there can be and how they relate to each other.
    To get some intuition of which Pareto-Nash equilibria a game can have, we experiment with different weights and visualize the resulting equilibria profiles in a plot where the strategies in the profile are represented as dots.
    \autoref{fig:paretoNashEquilibriaShowcase}
    shows examples of such plots for randomly generated 2-player games with 3 pure strategies for each player, and a varying number $m$ of objectives: For each equilibrium profile, the strategies of player 1/2 are represented by a red/green dot, respectively, projected from the two-dimensional mixed-strategy-simplex in three-dimensional space to the plane 
    (the corners represent pure strategies).
    The examples showcase the possible complexity of the set of Pareto-Nash equilibria:
    In one of the games, all equilibria mix between at most two of the three strategies; in others, equilibria seem to follow certain patterns which can be recognized in the visual representation.
    One could hope that while Pareto-Nash equilibria are not unique, they at least concentrate on a small number of points -- however the examples show that this is not the case in general, as there are many different equilibrium points in all examples.
%    (while in principle, as opposed to just a few that get generated by 
%    , where for example many equilibrium points concentrate on curves.
    Another observation is that in all cases, there are equilibrium profiles far apart from another, so Pareto-Nash equilibria of the same game obtained by different weightings need not be “close”, but can be completely different.
    In particular, this motivates that for an algorithm which calculates a specific Pareto-Nash equilibrium, it is reasonable to take a weighting vector as input instead of choosing one on its own.
    \begin{figure}[h]
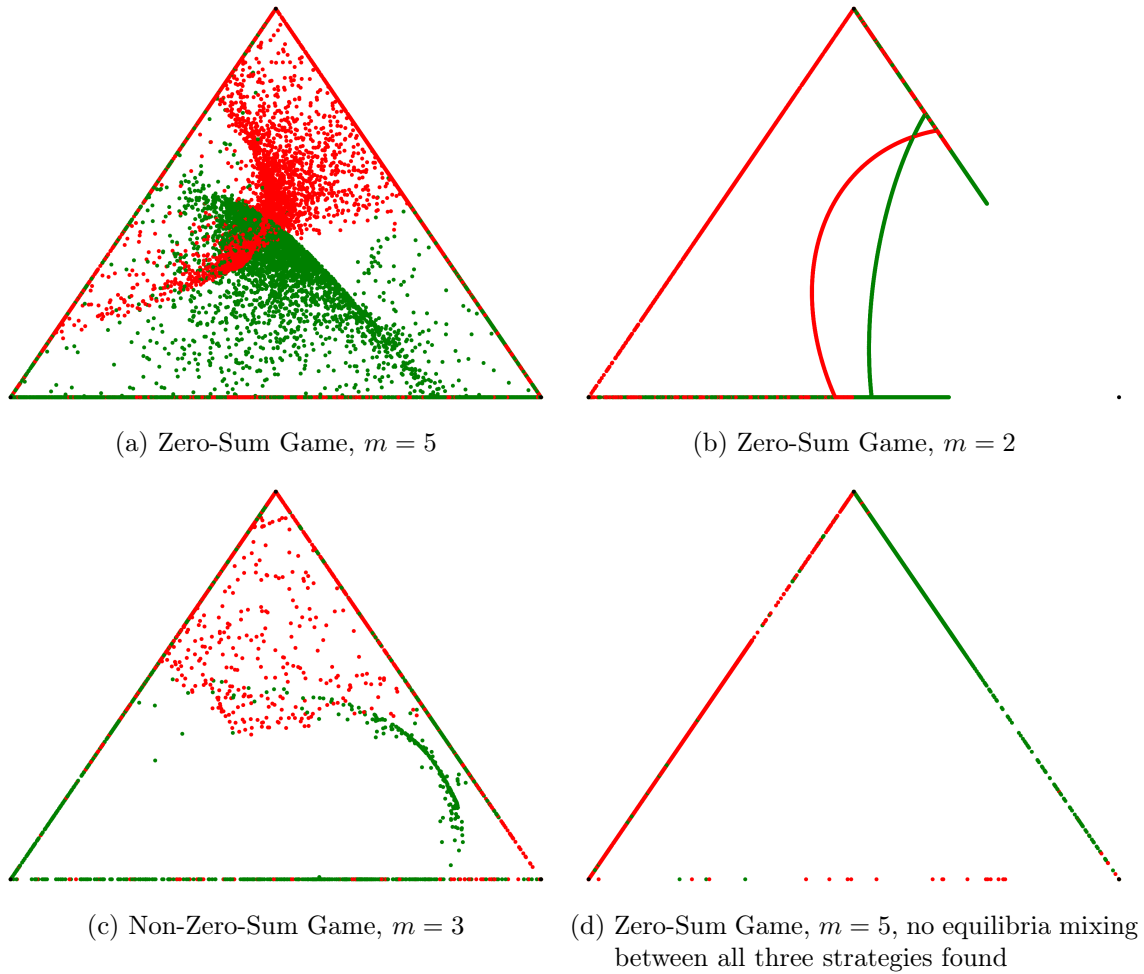

        \centering
        \begin{subfigure}[t]{0.49\textwidth}
            \includegraphics[width=\textwidth]{Pictures/paretoNash-curvyRegionsManyOutliers-k5}
            \caption{Zero-Sum Game, $m=5$}
        \end{subfigure}
        \begin{subfigure}[t]{0.49\textwidth}
            \includegraphics[width=\textwidth]{Pictures/paretoNash-curves1-k2}
            \caption{Zero-Sum Game, $m=2$}
        \end{subfigure}
        
        \vspace*{0.01\textwidth}
        
        \begin{subfigure}[t]{0.49\textwidth}
            \includegraphics[width=\textwidth]{Pictures/paretoNash-bimatrix-k3}
            \caption{Non-Zero-Sum Game, $m=3$}
        \end{subfigure}
        \begin{subfigure}[t]{0.49\textwidth}
            \includegraphics[width=\textwidth]{Pictures/paretoNash-emptyCenter-k5}
            \caption{Zero-Sum Game, $m=5$, no equilibria mixing between all three strategies found}
        \end{subfigure}
%        \vspace*{0.01\textwidth}
%        
%        \begin{subfigure}[b]{0.49\textwidth}
%            \includegraphics[width=\textwidth]{Pictures/paretoNash-clusters-k5.pdf}
%            \caption{}
%        \end{subfigure}
%        \begin{subfigure}[b]{0.49\textwidth}
%            \includegraphics[width=\textwidth]{Pictures/paretoNash-denseCurvyRegions-k3.pdf}
%            \caption{}
%        \end{subfigure}
%        
%        \vspace*{0.01\textwidth}
%        
%        \begin{subfigure}[b]{0.49\textwidth}
%            \includegraphics[width=\textwidth]{Pictures/paretoNash-curves2-k2.pdf}
%            \caption{}
%        \end{subfigure}
%        \begin{subfigure}[b]{0.49\textwidth}
%            \includegraphics[width=\textwidth]{Pictures/paretoNash-scatteredWithLines-k5.pdf}
%            \caption{}
%        \end{subfigure}
        \caption{Different sets of Pareto-Nash equilibria visualized (10\,000 points each)}
        \label{fig:paretoNashEquilibriaShowcase}
    \end{figure}

    \subsection{Multi-Objective Segmented-Distribution Games}
    Putting the pieces together, we construct the multi-objective game based on loss distribution segments as follows:
    Let $G$ be a distribution-valued bimatrix game with payoffs from $\DP_{[a, b]}$, $a < b$, and let $\mathcal{I} = \set{a, x_2, x_3, \dots, x_m, b}$ be a partition of the interval $[a, b]$.
    
    \newcommand{\Gseg}[1][]{G_{\text{seg,$\ifstrempty{#1}{\mathcal{I}}{#1}$}}}
    Define the \emph{segment game} $\Gseg$ as a multi-objective game with the same strategies as in $G$, where each player $k$ has the utility function:
    \begin{gather}
        u_{\text{seg}, k}: S \to \R^m, s \mapsto -\bigpars{E_{[a, x_2)}(u_k(s)), E_{[x_2, x_3)}(u_k(s)), \dots, E_{[x_m, b]}(u_k(s))}.
        \label{eq:segmentsExpectationVector}
    \end{gather}
    We negate the vector as we are in the context of loss distributions, but want to stick to the convention that utilities should be maximized.
    As outlined in the previous section, we can find Pareto-Nash equilibria  of $\Gseg$ by weighing the different objectives and then solving the resulting real-valued game.
    This approach is implemented for the thesis in the language \texttt{R}, and the following describes the details of the implementation.
    
    The \texttt{R} package \emph{HyRiM} \cite{bib:hyrimPackage} by Stefan Rass, Sandra König and Ali Alshawish was developed along with the papers \cite{bib:rassGameRiskManagI,bib:rassGameRiskManagII,bib:rassGameRiskManagIII} and implements data structures and algorithms for distribution-valued games.
    In particular, it provides the class \texttt{lossDistribution} that represents finitely-supported discrete and absolutely continuous loss distributions. Absolutely continuous distributions are approximated by a kernel density estimation method: Given a finite number of samples, their distribution's density function is approximated as a convex combination of Gaussian densities. The package also provides the class \texttt{mosg} that represents distribution-valued games, and implements the computation of Nash equilibria for real-valued games. 
    The abbreviation stands for \emph{multi-objective security game}, as the package (unlike the presentation in this thesis) allows to define games with multiple distribution-valued objectives.
    
    We implement the multi-objective segmentation-based games in the context of the HyRiM package, and as a possible extension to it. As the package focuses on zero-sum games, we also restrict ourselves to zero-sum segmented games.
    Multi-objective segmented games are represented by the \texttt{moseg} class. Such a game can be created from a single-objective distribution-valued game of the built-in \texttt{mosg} class, and a vector of partition points.
    Loss distributions are turned into real-valued expectation vectors by \eqref{eq:segmentsExpectationVector}: This is implemented in the function \texttt{segmentedLossDistribution}, which takes in a \texttt{lossDistribution} object and the partition points and returns the expectation vector.
    Computing the expected value is straightforward in the case of finitely-supported discrete distributions as a sum. For absolutely continuous distributions, the numerical integration function \texttt{integrate} provided by \texttt{R} is used.
    Finally, the method \texttt{moseg.paretoNashEquilibrium} computes a Pareto-Nash equilibrium of a \texttt{moseg} game, given a vector of weights as inputs.
    It first scalarizes the game based on the weights.
    Then for the actual equilibrium computation, it utilizes the HyRiM built-in method \texttt{mgss} which implements equilibrium computation for real-valued games.
    
    The three methods are implemented in the file \texttt{multiobjectiveSegmentGame.R}. The source code is shown on the following pages.

    \paragraph{Creation of Segmented Game}
    The code for creating a \texttt{moseg} game:
    
    \lstset{showspaces=false,
        keywordstyle=\ttfamily\bfseries\color{purple},
        basicstyle=\footnotesize\ttfamily,
        numbers=left,
        numberstyle=\tiny,
        commentstyle=\color{gray},
        breaklines=true,
        postbreak=\mbox{\textcolor{lightgray}{$\hookrightarrow$}\space},
        showstringspaces=false,
        stringstyle=\color{brickred},
    }
\begin{lstlisting}[language=R, firstline=13, lastline=31]
moseg <- function(sosg, partitionPoints) { 
    moseg <- NULL
    if (sosg$dim != 1) {
        stop("only single-objective games can be turned into a multi-objective segment game.")
    }
    # Convert the loss-distribution payoffs into expectation vectors of the segments
    moseg$losses <- mapply(
        function(l) segmentedLossDistribution(l, partitionPoints),
        sosg$losses)
    moseg$partitionPoints <- partitionPoints

    # Copy other properties from the original game
    moseg$nDefenses <- sosg$nDefenses
    moseg$nAttacks <- sosg$nAttacks
    moseg$defensesDescriptions <- sosg$defensesDescriptions
    moseg$attacksDescriptions <- sosg$attacksDescriptions
    class(moseg) <- "moseg"
    moseg
}
\end{lstlisting}
    
    \paragraph{Computing Expectation Vectors}
    The code for converting a \texttt{lossDistribution} to a segment expectation vector is given in the next listing. The implementation relaxes the requirement made in the definition of $\Gseg$ that the interval $[a, b]$ must cover the whole support of all distributions involved: Since absolutely continuous distributions are estimated as combination of Gaussian kernels, their support will always be the whole real line. Instead of placing an arbitrary restriction on the partition (e.g., 99\% of the probability mass must lie in $[a, b]$), we prefer to give the user the flexibility to choose the partition freely. To make this behavior consistent, the same is allowed for discrete distributions. Unlike in \eqref{eq:segmentsExpectationVector}, the expectation vectors need not be negated in the implementation, because the HyRiM packages already interprets payoffs as losses.
\begin{lstlisting}[language=R, firstline=67, deletekeywords={density}]
segmentedLossDistribution <- function(lossDistribution, partitionPoints) {
    segmentedDistribution <- NULL
    expectedValues <- rep(0, length(partitionPoints) - 1)

    # Check that there are at least two partition points
    if (length(partitionPoints) <= 1) {
        stop("at least two partition points required.")
    }
    # Check that the partition points are in ascending order 
    for (i in 1:(length(partitionPoints) - 1)) {
        if (partitionPoints[i+1] <= partitionPoints[i]) {
            stop("partition points must be in ascending order.")
        }
    }

    # Compute expected values in the discrete, finitely-supported case 
    if (lossDistribution$is.discrete) {
        supp <- lossDistribution$supp
        pdf <- lossDistribution$dpdf
        i <- 1 # Index of the current support point
        j <- 2 # Index of the right end of the current partition

        # Skip support points left of the first partition
        while (supp[i] < partitionPoints[1]) {
            i <- i + 1
        } 

        while (i <= length(supp) && j <= length(partitionPoints)) {
            # All intervals except the last one are half-open, the last one is closed.
            if (supp[i] < partitionPoints[j] || (supp[i] == partitionPoints[j] && j == length(partitionPoints))) {
                expectedValues[j - 1] <- expectedValues[j - 1] + supp[i] * pdf[i]
                i <- i + 1 # Go to the next support point
            } else {
                j <- j + 1 # Go to the next partition
            }
        }
    } 
    # Compute expected values in the absolutely-continuous case
    else {
        density <- lossDistribution$lossdistr
        # Go over all partition intervals
        for (i in 1:(length(partitionPoints) - 1)) {
            # Numerically calculate the integral. The normalizationFactor is calculated by HyRiM since the kernel-estimated densities often do not integrate to 1 on the specified support.
            integral <- integrate(density, partitionPoints[i], partitionPoints[i+1])
            expectedValues[i] <- integral$value * lossDistribution$normalizationFactor
        }
    }
    segmentedDistribution$expectedValues <- expectedValues
    class(segmentedDistribution) <- "segmentedLossDistribution"
    return(segmentedDistribution)
}
\end{lstlisting}
    
    \paragraph{Computing Pareto-Nash Equilibria}
    The code for computing Pareto-Nash equilibrium of a \texttt{moseg} based on weights is shown in the next listing.
\begin{lstlisting}[language=R, firstline=42, lastline=55, deletekeywords={weights}]
moseg.paretoNashEquilibrium <- function(moseg, weights) {
    if (length(weights) != length(moseg$partitionPoints) - 1) {
        stop(sprintf("there are %d segments, but %d weights were specified.",
            length(moseg$partitionPoints) - 1, length(weights)))
    }
    # Scalarize the game with the given weights
    scalarLosses <- rep(0, length(moseg$losses))
    for (i in 1:length(scalarLosses)) {
        scalarLosses[i] <- moseg$losses[i]$expectedValues %*% weights
    }
    scalarGame <- mosg(n=moseg$nDefenses, m=moseg$nAttacks, goals=1, losses=scalarLosses)
    # Compute an equilibrium of the scalar game using HyRiM's built-in algorithm
    mgss(scalarGame)
}
\end{lstlisting}

    \paragraph{Examples}
    Examples of using the code are supplied in the file \emph{mosegExamples.R}.
    There are two examples: The first example constructs a 2x2 two-player zero-sum game with discrete distributions supported on $\set{1, \dots, 10}$ as payoffs.
    The second example constructs a 2x2 two-player zero-sum game with absolutely continuous distributions as payoffs.
    In both cases, a segmented game is created and the Pareto-Nash equilibrium given a fixed weights vector is computed.
    For comparison, the “MGSS” (\emph{Multi-Goal Security Strategy}) solution the HyRiM-built-in method \texttt{mgss} computes is output as well.

    \chapter{Conclusion}
    The model of distribution-valued games provides a valuable tool to model games in uncertain situations where the exact outcomes cannot be known beforehand, but can only be modeled on a stochastic basis. Of course, the usefulness of this model depends heavily on the ability to specify suitable preferences in the form of stochastic orders, and the thesis shows that this is a critical point and that the currently considered orderings have some shortcomings. In particular, two problems with the tail order were identified: The first issue is that the order is not total unless one restricts the order space, and in particular for any non-degenerate interval $[a, b] \subseteq [0, \infty)$, there are incomparable distributions supported on $[a, b]$. The second issue is that tail-ordered games can fail to have Nash equilibria, and that mixed-strategy Nash equilibria only exist for games with a specific structure. 
    
    Neither of the two issues is grave enough to stop the tail order from being useful: The first problem is rather of mathematical than of practical importance, and it seems plausible that the cases where the order exhibits incomparability will rarely, if ever, occur in practice. The problem can even be circumvented altogether if one identifies a smaller class of admissible distributions, and shows that the ordering is total within that class (of course, some work is required if one wants to show the totality on such a smaller class -- further work could try to identify easy-to-check sufficient conditions for the order to be total on such a class). The second problem is more severe: The famous theorem by John Nash that all real-valued games have at least one Nash equilibrium is a cornerstone of the classical theory, and it certainly has practical implications that there is no equivalent for tail-ordered distribution-valued games. Yet  the fact that the strict conditions only apply to mixed-strategy Nash equilibria somewhat mitigates the issue, since many distribution-valued games can still have pure Nash equilibria.
    So both problems can be circumvented to a certain extent -- anyway, their existence shows that some care has to be taken when using tail-ordered games, and that not all guarantees that make life easy in the real-valued theory continue to hold in the distribution-valued case.
    
    It must also be kept in mind that the tail order is only one approach to expressing preferences in distribution-valued games, and a route for further work on the topic could be to analyze the behavior with respect to different orderings. An example of this is the tweakable stochastic order $\leqtw$ we showed in Chapter 5. 
    We subsequently introduced a method to turn a distribution-valued game into a multi-objective game and used Pareto-Nash equilibria to solve it: This is an example of how a different solution concept can be used to guarantee the existence of solutions.
    However, this solution concept does not use a lexicographic comparison anymore.
    Further work could try to find another solution concept more fitting than the Nash equilibrium for distribution-valued games, that still uses lexicographic comparisons in the spirit of tail-order Nash equilibria, yet is guaranteed to always exist.
    
    In summary, this thesis analyzed the model of distribution-valued games and the tail order from a mathematical point of view.
    Some questions about distribution-valued games could be cleared up, yet there is a lot of potential for future work on the subject, especially with regard to proving totality of the tail order on smaller classes of distributions, examining different preference orderings for distribution-valued games, and possibly formulating different solution concepts.
    
    \setstretch{1.08}
    \printbibliography

    \checkoddpage
    \ifthenelse{\boolean{oddpage}}{\newpage\thispagestyle{empty}~}{}
    \let\fakeincludegraphics\includegraphics
    \let\includegraphics\realincludegraphics
%    \includepdf{Eigenstaendigkeitserklaerung}
    \let\includegraphics\fakeincludegraphics
\end{document}